\documentclass[reqno,11pt]{amsart}

\usepackage{amsmath,amsfonts,amssymb,graphicx,amsthm,enumerate,url}
\usepackage[noadjust]{cite}
\usepackage{stmaryrd}

\usepackage{xifthen,xcolor,tikz,setspace}
\usetikzlibrary{decorations.pathmorphing,patterns,shapes,calc}
\usepackage[noadjust]{cite}
\definecolor{refkey}{gray}{.75}
\definecolor{labelkey}{gray}{.5}

\usepackage[colorlinks=true]{hyperref}
\colorlet{DarkGreen}{green!50!black}
\colorlet{DarkGray}{gray!60!black}
\colorlet{DarkPurple}{purple!60!black}

\numberwithin{equation}{section}

\renewcommand{\restriction}{\mathord{\upharpoonright}}
\renewcommand{\epsilon}{\varepsilon}
\newcommand{\given}{\;\big|\;}
\newcommand{\one}{\mathbf{1}}
\DeclareMathOperator{\xor}{\triangle}

\usepackage{color}
 \definecolor{refkey}{gray}{.5}
 \definecolor{labelkey}{gray}{.5}
\definecolor{light}{gray}{.9}

\newtheorem{maintheorem}{Theorem}

\newtheorem{theorem}{Theorem}[section]
\newtheorem*{theorem*}{Theorem}
\newtheorem{lemma}[theorem]{Lemma}
\newtheorem{claim}[theorem]{Claim}
\newtheorem{proposition}[theorem]{Proposition}

\newtheorem{fact}[theorem]{Fact}
\newtheorem{corollary}[theorem]{Corollary}

\theoremstyle{definition}{

\newtheorem{definition}[theorem]{Definition}
\newtheorem*{definition*}{Definition}

\newtheorem{remark}[theorem]{Remark}
}

\renewcommand{\P}{\mathbb P}

\newcommand{\Z}{\mathbb Z}

\newcommand{\cC}{\ensuremath{\mathcal C}}

\newcommand{\cF}{\ensuremath{\mathcal F}}

\newcommand{\cP}{\ensuremath{\mathcal P}}

\newcommand{\llb }{\llbracket}
\newcommand{\rrb }{\rrbracket}

\renewcommand{\epsilon}{\varepsilon}

\DeclareMathOperator{\hull}{hull}

\newcommand{\tmix}{t_{\textsc{mix}}}

\newcommand{\tv}{{\textsc{tv}}}
\newcommand{\north}{{\textsc{n}}}
\newcommand{\south}{{\textsc{s}}}
\newcommand{\east}{{\textsc{e}}}
\newcommand{\west}{{\textsc{w}}}
\newcommand{\sd}{{\operatorname{sd}}}

\newcommand{\bP}{{\mathbf{P}}}

\makeatletter
\newcommand{\superimpose}[2]{%
  {\ooalign{$#1\@firstoftwo#2$\cr\hfil$#1\@secondoftwo#2$\hfil\cr}}}
\makeatother
\newcommand{\nconn}{\mathpalette\superimpose{{\longleftrightarrow}{\!\!\!\!\not}}}

\begin{document}

\title[Quasi-polynomial mixing of critical 2D random cluster models]{Quasi-polynomial mixing of critical \\ 2D random cluster models}

\author{Reza Gheissari}
\address{R.\ Gheissari\hfill\break
Courant Institute\\ New York University\\
251 Mercer Street\\ New York, NY 10012, USA.}
\email{reza@cims.nyu.edu}

\author{Eyal Lubetzky}
\address{E.\ Lubetzky\hfill\break
Courant Institute\\ New York University\\
251 Mercer Street\\ New York, NY 10012, USA.}
\email{eyal@courant.nyu.edu}

\begin{abstract}
We study the  Glauber dynamics for the random cluster (FK) model on 
the torus $(\mathbb{Z}/n\mathbb{Z})^2$ with parameters $(p,q)$, for $q  \in (1,4]$ and $p$ the critical point~$p_c$. The dynamics is believed to undergo a \emph{critical slowdown}, with its continuous-time mixing time transitioning from $O(\log n)$ for $p\neq p_c$ to a power-law in $n$ at $p=p_c$.  This was verified at $p\neq p_c$ by Blanca and Sinclair, whereas at the critical  $p=p_c$, with the exception of the special integer points $q=2,3,4$ (where the model corresponds to the Ising/Potts models) the best-known upper bound on mixing  was exponential in $n$.
Here we prove an upper bound of $n^{O(\log n)}$ at $p=p_c$ for all $q\in (1,4]$, where a key ingredient is bounding the number of nested long-range crossings at criticality.
\end{abstract}

{\mbox{}
\vspace{-1.25cm}
\maketitle
}
\vspace{-0.9cm}

\section{Introduction}

The random cluster (FK) model is an extensively studied model in statistical physics, generalizing electrical networks, percolation, and the Ising and Potts models, to name a few, under a single unifying framework. It is defined on a graph $G=(V,E)$ with parameters $0<p<1$ and $q>0$ as the probability measure over subsets $\omega\subset E$ (or equivalently, configurations $\omega \in\{0,1\}^E$), given by
\[\pi_{G,p,q}(\omega)\propto p^{|\omega|}(1-p)^{|E|-|\omega|}q^{k(\omega)}\,,
\]
where $k(\omega)$ is the number of connected components (clusters) in the graph $(V,\omega)$.

At $q=1$, the FK model reduces to independent bond percolation on $G=(V,E)$, and for integer $q\geq 2$ it corresponds via the Edwards--Sokal coupling~\cite{EdSo88} to the Ising ($q=2$) and Potts ($q\geq 3$) models on $V$. Since its introduction around 1970, the model has been well-studied both in its own right and as a means of analyzing the Ising and Potts models, with an emphasis on $\Z^d$ as the underlying graph. There, for every $q\in[1,\infty)$, the model enjoys monotonicity, and exhibits a phase transition at a critical $p_c(q)$ w.r.t.\ the existence (almost surely) of an infinite cluster  (see, e.g.,~\cite{Gr04} and references therein).

Significant progress has been made on the model in $d=2$, in particular  for $1\leq q\leq 4$ where the model is expected to be conformally invariant (see~\cite[Problem 2.6]{Sch07}). It is known~\cite{BeDu12} that $p_c(q)=\frac {\sqrt q}{1+\sqrt q}$ on $\Z^2$ for all $q\geq 1$. Moreover, while the phase transition at this $p_c$ is discontinuous if $q>4$ (as confirmed for all $q>25$ in~\cite{LMR86} and very recently, all $q>4$ in~\cite{DGHMT16}), it is continuous for $1\leq q \leq 4$ (as established in~\cite{DST15}). There, the probability that $x$ belongs to the cluster of the origin decays as $\exp(-c|x|)$ at $p<p_c$, as a power-law $|x|^{-\eta}$ at the critical $p_c$, and is  bounded away from 0 at $p>p_c$.

Here we study  heat-bath Glauber dynamics for the two-dimensional FK model, where the following \emph{critical slowdown} phenomenon is expected: on an $n\times n$ torus, for all $p\neq p_c$ the mixing time of the dynamics should have order $ \log n$ (recently shown by~\cite{BlSi15}), yet at $p=p_c$ it should behave as $ n^z$ for some universal  $z>0$ in the presence of a continuous phase transition, and as $\exp(c n)$ in the presence of a discontinuous phase transition. The critical behavior in the former case (all $q>4$) was established in a companion paper~\cite{GL16}, as was a critical power-law in the cases $q=2$ (\cite{LScritical}) and $q=3$ (\cite{GL16}). In this work we obtain a quasi-polynomial upper bound for non-integer $1<q\leq 4$ at criticality.

More precisely, Glauber dynamics for the FK measure $\pi_{G,p,q}$ is
the continuous-time Markov chain $(X_t)_{t\geq 0}$ that assigns each edge $e\in E$ an i.i.d.\ rate-$1$ Poisson clock, where upon ringing, $X_t(e)$ is resampled via $\pi_{G,p,q}$ conditioned on the values of $ X_t$ on $E-\{e\}$. This Markov chain is reversible by construction w.r.t.\ $\pi_{G,p,q}$, and may hence be viewed both as a natural model for the dynamical evolution of this interacting particle system, and as a simple protocol for sampling from its equilibrium measure. A central question is then to estimate the time it takes this chain to converge to stationarity, measured in terms of the total variation mixing time $\tmix$  (see~\S\ref{sub:prelim-dynamics} for the related definitions).

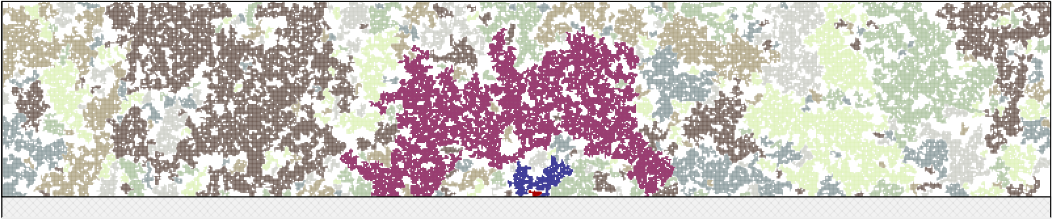
\begin{figure}
  \hspace{-0.15in}
  \begin{tikzpicture}

    \newcommand{\xfigshift}{9pt}
    \newcommand{\yfigshift}{29.5pt}

      \node (plot) at (0,0){
      \includegraphics[width=1\textwidth]{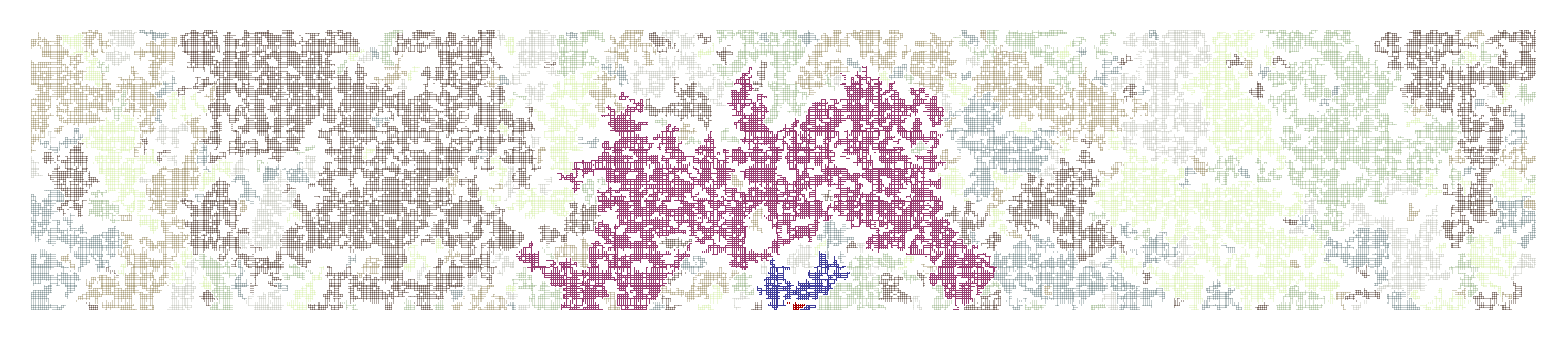}
      };

    \begin{scope}[shift={($(plot.south west)+(\xfigshift,\yfigshift)$)}, x={($(plot.south east)+(\xfigshift,\yfigshift)$)},y={($(plot.north west)+(\xfigshift,\yfigshift)$)}, font=\small,scale=1]

      \draw[color=black] (.95,.58) -- (0.006,.58) -- (.006,-0.179) -- (.95,-0.179);
      \draw[color=black] (0.95,0.58) -- (.95,-.26);
      \draw[color=black] (.006,.58) -- (.006,-.26);

          \fill[fill=gray, fill opacity=0.1] (.006,-.179) rectangle (.95,-.26);
      \draw[pattern=north west lines, pattern color=black, opacity=0.2] (.006,-.179) rectangle (.95,-.26);
      \draw[pattern=north east lines, pattern color=black, opacity=0.2] (.006,-.179) rectangle (.95,-.26);

      \draw[color=white] (.006,-.26) -- (.95,-.26);

    \end{scope}
  \end{tikzpicture}
  \vspace{-0.3in}
  \caption{A critical FK configuration that induces three nested, distinct boundary components (red, blue, purple) called bridges.}
  \label{fig:FK-cluster-bridges}
\end{figure}

For $p\neq p_c$, the fact that $\tmix \asymp \log n$ was established in~\cite{BlSi15} using the aforementioned exponential decay of cluster diameters in the high-temperature regime $p<p_c$:  on finite boxes with certain boundary conditions, this translates to a  property known as \emph{strong spatial mixing}, implying that the number of disagreements between the states of two chains started at  different initial states decreases exponentially fast, thus $\tmix\asymp \log n$; 
this result readily extends  to $p>p_c$ by the duality of the two-dimensional FK model.

At $p=p_c$, where there is no longer  an exponential decay of correlations, polynomial upper bounds on $\tmix$ were obtained for the Ising model in~\cite{LScritical} and the 3-state Potts model---along with a quasi-polynomial bound for the 4-state model---in~\cite{GL16}, using a multiscale approach that reduced the side length of the box by a constant factor in each step  
via a coupling argument; these carry over to the FK model for $q=2,3,4$ by the comparison estimates of~\cite{Ul13}. However, for non-integer $q$, FK configurations may form macroscopic connections along the boundary of smaller-scale boxes, destroying the coupling---this is prevented for integer $q$ thanks to the special relation between FK/Potts models. To control this effect, we prove upper and lower bounds on the total number of disjoint macroscopic connections along the boundaries of the smaller-scale boxes at $p=p_c$ which may be of independent interest (see \S\ref{sub:main-techniques} as well as Fig.~\ref{fig:FK-cluster-bridges}--~\ref{fig:long-range-bc}) 

It was recently shown~\cite{GuJe16} that for $q=2$ the FK Glauber dynamics on \emph{any} graph $G=(V,E)$ has $\tmix\leq |E|^{O(1)}$; the technique there, however, is highly specific to the case of $q=2$. Indeed,  this bound does not hold on $\Z^d$, for any $d\geq 2$,
at $p=p_c$ and $q$ large, as follows from the exponential lower bounds of~\cite{BCFKTVV99,BCT12} 
(see, e.g.,~\cite{CDL12,GL16} for further details).
The best prior upper bound on non-integer $1<q<4$ was  $\tmix \leq \exp(O(n))$.

In the present paper, we prove that for periodic boundary conditions (as well as a wide class of others, including wired and free; see Remark~\ref{rem:mainthm-bc}), the following holds:

\begin{maintheorem} \label{mainthm:1}
Let $q\in (1,4]$ and consider the Glauber dynamics for the critical FK model on $(\mathbb Z/n\mathbb Z)^2$.  There exists $c=c(q)>0$ such that
$\tmix\lesssim n^{c\log n}$.
\end{maintheorem}

\begin{remark}\label{rem:mainthm-bc}
Theorem~\ref{mainthm:1} holds for rectangles with uniformly bounded aspect ratio, under any set of  boundary conditions with the following property: for every edge $e$ on the boundary of the box, there are $O(\log n)$ distinct boundary components connecting vertices on either side of $e$ (see Definitions~\ref{def:xi}--\ref{def:upsilon} and Theorem~\ref{thm:mainthm-fixed-bc}).
This includes, in particular, the wired and free boundary conditions, as well as, with high probability, ``typical'' boundary conditions: those that are sampled from $\pi_{\mathbb Z^2,p_c,q}$ (see Lemma~\ref{lem:retain-bc-1}).
\end{remark}

\begin{remark}\label{rem:q=4} For $q\in \{2,3,4\}$, the comparison estimates of~\cite{Ul13} carry the upper bounds on the mixing time of  the Potts model  to the FK model, yet only for a limited class of boundary conditions (e.g., 
the partition of boundary vertices can have at most one cluster of size larger than $n^\epsilon$, in contrast to  ``typical'' ones as above).
 The above theorem thus extends the class of FK boundary conditions for which  $\tmix$ is quasi-polynomial.
\end{remark}
\begin{remark}\label{rem:mainthm-other-dynamics}
Theorem~\ref{mainthm:1} implies analogous bounds for other single-site dynamics (e.g., Metropolis), as well as global cluster dynamics (e.g.,  Chayes--Machta~\cite{ChMa97}) via~\cite{Ul13}.
\end{remark}

\subsection{Main techniques}\label{sub:main-techniques}
\subsubsection{Disjoint long-range connections}
As pointed out in~\cite{BlSi15} and later in~\cite{GL16}, disjoint long-range clusters along the boundary of rectangles, called bridges (see Fig.~\ref{fig:bridges-no-hull}), are a major obstacle to mixing time upper bounds for the FK Glauber dynamics.

\begin{figure}
\vspace{-0.2in}
  \hspace{-0.15in}
  \begin{tikzpicture}

    \newcommand{\xfigshift}{0pt}
    \newcommand{\yfigshift}{0pt}

     \node (plot) at (0,0){\includegraphics[width=0.6\textwidth]{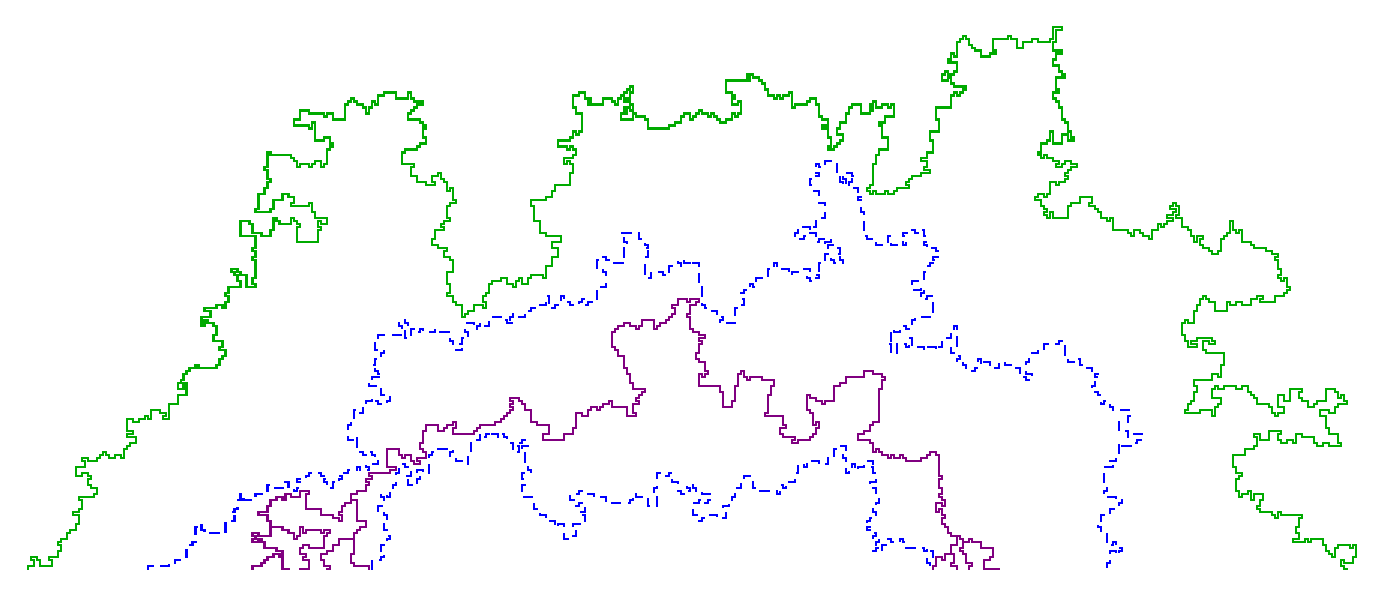}};

    \begin{scope}[shift={($(plot.south west)+(\xfigshift,\yfigshift)$)}, x={($(plot.south east)+(\xfigshift,\yfigshift)$)},y={($(plot.north west)+(\xfigshift,\yfigshift)$)}, font=\small,scale=.91]
     \draw[color=black] (0,0.075) -- (1.1,0.075);

    \end{scope}
  \end{tikzpicture}
  \vspace{-0.2in}
  \caption{A pair of boundary connections (green, purple) constituting two distinct boundary bridges, separated by a (blue) dual connection.}
  \label{fig:bridges-no-hull}
  \vspace{-0.1in}
\end{figure}

\begin{definition}
Let $\Lambda_{n,m}= [0,n]\times [0,m] \cap \mathbb Z^2$. Given an FK configuration $\omega$ on $\mathbb Z^2-\Lambda_{n,m}$ and a boundary edge $e\in \partial \Lambda_{n,m}$, say without loss of generality $e\in \partial_\north \Lambda_{n,m}$, a \emph{bridge} over $e$ is an open FK cluster in $\omega$ that contains at least one vertex in $\partial_\north \Lambda_{n,m}$ to the left of $\{e\}$ and one to the right. Let $\Gamma^e(\omega)$ denote the set of all bridges of $e$.
\end{definition}
(For a more detailed definition, we also refer the reader to Section~\ref{sub:bdy-bridges}.)
In critical bond percolation ($q=1$ and $p=1/2$), the Berg--Kesten (BK) inequality would suggest that $\pi( |\Gamma^e| > K \log n + a) \leq \exp(-c a)$ for universal $K,c>0$ and all $a$. In our setting of FK percolation for $1<q<4$ at $p=p_c$, the classical BK inequality does not hold; 
 nevertheless, Proposition~\ref{prop:bridge-bound} and Lemmas~\ref{lem:disjoint-1}--\ref{lem:disjoint-2} establish such a bound for $|\Gamma^e|$.

It then follows (see Corollary~\ref{cor:typical-bridges}) that if we sample from the FK measure on a $2n\times 2n$ box $\Lambda$  with \emph{arbitrary} boundary conditions, the boundary conditions this induces on the concentric inner $n\times n$ box will have order $\log n$ distinct bridges over a given edge with probability $1-O(n^{-c})$.
A simpler formulation of that is as follows.

\begin{theorem}
For every $q\in (1,4]$, there exist $K'>K>0$ and $c(q)>0$ such that, the critical FK model has, for every $e\in \partial \Lambda_{n,n}$ a, say, $\frac n{10}$ distance from a corner of $\Lambda_{n,n}$,
\begin{align*}
\pi_{\mathbb Z^2,p_c,q}\Big(K \log n \leq |\Gamma^e(\omega)| \leq K'\log n\Big) \geq 1-O(n^{-cK'})\,.
\end{align*}
\end{theorem}

\begin{figure}
\vspace{-0.13in}
  \begin{tikzpicture}
    \node (plot1) at (0,0) {};

    \node (plot2) at (7,0) {};

    \begin{scope}[shift={(plot1.south west)},x={(plot1.south east)},y={(plot1.north west)}, font=\small]

     \draw[color=DarkGreen, thick] (0.,5.9) arc (130:230:2.4) ;
     \draw[color=DarkGreen, thick] (0.,9.4) arc (130:230:2.4) ;
     \draw[color=purple, thick] (0,13.04) arc (148:212:11.6) ;
     \draw[color=blue, thick] (20.05,.9) arc (-34:34:8);
     \draw[color=blue, thick] (20.05,9.8) arc (-42:42:2);
     \draw[thick, color=orange,decoration={bumps, segment length=0.4in, amplitude=0.07in}, decorate] (20,7.63) -- (20, 1.95);

     \draw[draw=black,dashed] (0,7) rectangle (20,14);
     \draw[color=black,thick] (0,0) rectangle (20,14);

     \node (fig1) at (10.02,8.5) {
	\includegraphics[width=159.5pt]{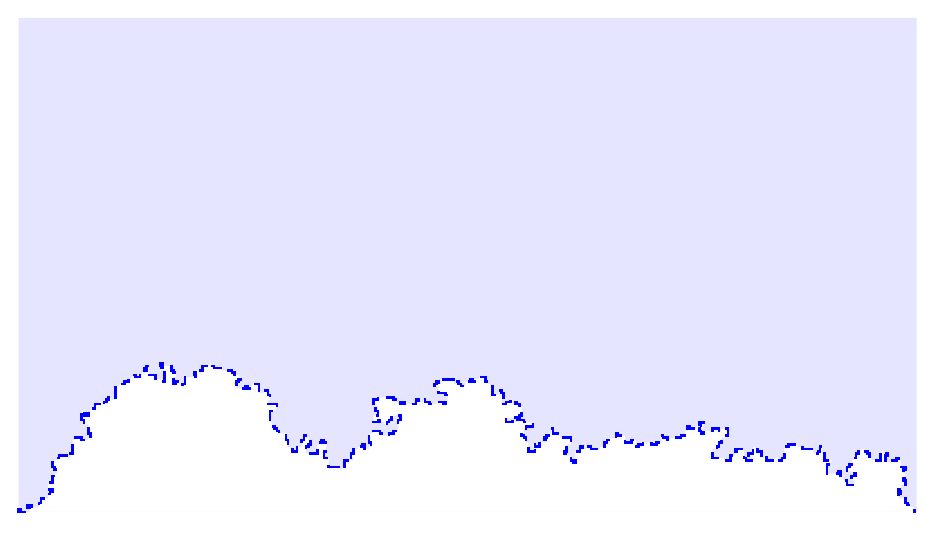}};

    \end{scope}

    \begin{scope}[shift={(plot2.south west)},x={(plot2.south east)},y={(plot2.north west)}, font=\small]

     \draw[color=purple, thick] (0.,5.9) arc (130:230:2.4) ;
     \draw[color=purple, thick] (0.,9.4) arc (130:230:2.4) ;
     \draw[color=purple, thick] (0,13.04) arc (148:212:11.6) ;
     \draw[color=blue, thick] (20.05,9.8) arc (-42:42:2);
     \draw[color=blue, thick] (20.05,.9) arc (-34:34:8);
     \draw[thick, color=purple,decoration={bumps, segment length=0.4in, amplitude=0.07in}, decorate] (20,7.63) -- (20, 1.95);

     \draw[draw=black,dashed] (0,7) rectangle (20,14);
     \draw[color=black,thick] (0,0) rectangle (20,14);

     \node (fig1) at (10.02,8.5) {
	\includegraphics[width=159.5pt]{fig_interface}};

     \node (fig2) at (10.01,1.25) {
	\includegraphics[width=159.5pt]{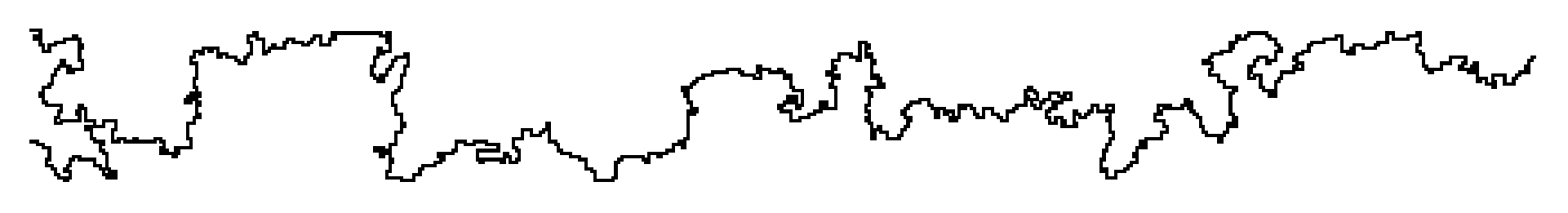}};

    \draw[color=black!50,decoration={bumps, segment length=0.185in, amplitude=0.02in}, decorate] (20,-0.04) -- (0, -0.04);

    \end{scope}

  \end{tikzpicture}
  \vspace{-0.15in}
  \caption{Macroscopic disjoint boundary bridges prevent the coupling of FK configurations sampled under two different boundary conditions on $\partial_\south \Lambda$ from being coupled past a common horizontal dual-crossing.}
  \label{fig:long-range-bc}
    \vspace{-0.07in}
\end{figure}

The lower bound on $|\Gamma^e|$ demonstrates that the behavior of the number of bridges at $p=p_c$ is truly different than at $p\neq p_c$; there, by the exponential decay of correlations at $p<p_c$ (dual connections at $p>p_c$), the typical number of bridges over an edge is~$O(1)$.
The upper bound in the above theorem arises in two crucial ways in the proof of Theorem~\ref{mainthm:1}: (a) the typical number of bridges over an edge $f$ being $O(\log n)$ is used to disconnect all potentially destructive bridges at a cost of $e^{|\Gamma^f|} \lesssim n^{c}$ (Eq.~\eqref{eq:tmix-M}), and (b) the exponential upper tail beyond that is used to sustain a union bound over $n^{O(1)}$  many attempts at coupling (see ~\S\ref{subsub:dynamical} for more details).    

To be more precise about the obstacle posed by having multiple bridges over an edge, recall the following. In~\cite{LScritical} and then~\cite{GL16}, the upper bounds on the mixing time of the Potts models at $\beta=\beta_c$ for $q\in\{2,3,4\}$ relied on RSW bounds~\cite{DHN11, DST15} to expose dual-interfaces in the FK representation, beyond which block dynamics chains could be coupled. However, the fact that chains, started from any two initial configurations, could be coupled past a dual-interface, relied on a certain conditional event, implicit in the relation between the FK and Potts models at integer $q$ (that no distinct boundary components were connected in the interior configuration). Without this conditioning, connections between two components on one side of a rectangle alter the boundary conditions elsewhere via bridges over the dual-interface, preventing coupling (see Fig.~\ref{fig:long-range-bc}). 

Similar difficulties were pointed out in~\cite{BlSi15} at $p< p_c$, and later in~\cite{GL16} at criticality when  $q>4$. In both of these cases, the exponential decay of correlations under $\pi_{\mathbb Z^2}^0$ ensured that all such bridges would be, with high probability, microscopic: in~\cite{BlSi15} these bridges were negligible after restricting attention to  \emph{side-homogenous} (wired or free on sides) boundary conditions, while in~\cite{GL16}, relevant boundary segments could be disconnected from one another by brute-force modifications. In contrast, in the present setting at the critical point of a continuous phase transition, the power-law decay of correlations precludes such  techniques; thus,  at $p=p_c(q)$ obtaining sharp bounds on the number of bridges becomes not only necessary, but also substantially more delicate.
  
To convert the upper bounds on bridges to an upper bound on the mixing time, our dynamical analysis  restricts its attention to boundary conditions with $|\Gamma^e| = O(\log n)$  for every $e\in \partial \Lambda$; we call such boundary conditions ``typical" (see Definitions~\ref{def:xi}--\ref{def:upsilon}) and observe that wired and free boundary conditions are both typical. 

\subsubsection{Refined dynamical scheme}\label{subsub:dynamical}

To maintain ``typical" (as opposed to worst-case) boundary conditions throughout the multi-scale analysis,  we turn to the Peres--Winkler censoring inequalities~\cite{PW13} for monotone spin systems, that were used in~\cite{MaTo10} (then later in~\cite{LMST12}) for the Ising model under ``plus" boundary, a class of boundary conditions that \emph{dominate the plus phase} (observe that, in contrast, ``typicality" is not monotone).

A major issue when attempting to carry out this approach---adapting the analysis of the low temperature Ising model to the critical FK model---is the stark difference between the nature of the corresponding equilibrium estimates needed to drive the multi-scale analysis. In the former, crucial to maintaining ``plus'' boundary conditions throughout the induction of~\cite{MaTo10} was that in the presence of favorable boundary conditions, the multiscale analysis could be controlled except with super-polynomially small probability. 
This yielded a bound on coupling the dynamics started at the extremal (plus and minus) initial configurations, which a standard union bound over the $O(n^2)$ sites of the box (see Fact~\ref{fact:init-config-comparison}) then transformed to a bound on $\tmix$.

We wish to couple the dynamics from the extremal (wired and free) initial configurations, since arbitrary starting states may induce boundary conditions on the smaller scales that are not ``typical.'' Even in the ideal scenario where the induced boundary conditions have \emph{no} bridges, though, the probability that we fail to couple the dynamics from wired and free initial states is at least $1-\epsilon$  (as per the RSW estimates). In particular, even in this ideal setting, 
we could not afford  the $O(n^2)$ factor of translating this to a bound on $\tmix$---the actual setting is far worse, replacing the failure probability by $1-n^{-c}$ (Proposition~\ref{prop:point-to-point-crossing}). An approach based on the classical block dynamics recursion on the spectral gap (see~\cite{Ma97} and the proofs in~\cite{LScritical,GL16}) would force one to analyze the dynamics under worst-case boundary, whereas we would like to restrict attention to the typical boundary conditions encountered throughout the dynamical process.

Therefore, in Definition~\ref{def:censoring-block} we construct a censored dynamics that mimics a block dynamics chain, and bound the total variation distance between their distributions in terms of the probability that we encounter unfavorable boundary conditions on the sub-blocks before mixing (Proposition~\ref{prop:censoring-block-dynamics}). By doing so, we compare the censored dynamics to the block dynamics with boundary conditions modified to eliminate all $O(\log n)$ bridges over certain edges, paying a cost of $n^{c}$ in the mixing time. We then let the block dynamics run some $n^{c}$ rounds (paying a union bound for the probability of ever encountering atypical boundary conditions, bounded in Corollary~\ref{cor:retain-bc-2}). The block dynamics would be making $n^c$ many independent attempts at coupling (possible due to the absence of bridges over a particular edge) beyond a dual-crossing whose existence has probability $n^{-c/2}$: see Lemma~\ref{lem:systematic-bound}. This polynomial bound on the block dynamics coupling time translates to a quasi-polynomial bound for the censored dynamics via $O(\log n)$ recursions onto smaller scale blocks, yielding the bound on $\tmix$. 

Finally, since periodic boundary conditions do not fall in our class of ``typical" boundaries, in Section~\ref{sub:proof-mainthm} we extend this bound first to cylinders, and then to the torus. 

\section{Preliminaries}\label{sec:preliminaries}

In this section we define the random cluster (FK) model and the FK dynamics that will be the object of study in this paper. We also recall various important results from the equilibrium theory of the FK model in~\S\ref{sub:prelim-equilibrium}, and the general theory of Markov chain mixing times (\S\ref{sub:prelim-dynamics}), including, in particular, that of monotone chains. Throughout the paper, for sequences $f(N),g(N)$ we will write $f\lesssim g$ if there exists a constant $c>0$ such that $f(N)\leq cg(N)$ for all $N$ and $f\asymp g$ if $f(N)\lesssim g(N)\lesssim f(N)$.

For a more detailed exposition of much of~\S\ref{sub:prelim-equilibrium} see~\cite{Gr04}, and for a more detailed exposition of the main ideas in~\S\ref{sub:prelim-dynamics} see~\cite{LPW09}.

\subsection{The FK model}\label{sub:prelim-equilibrium}
Throughout the paper, we identify an FK configuration $\omega \subset E$ with an assignment $E\to \{0,1\}$, referring to an edge $e$ with $\omega(e)=1$ as \emph{open} and to an edge $e$ with $\omega(e)=0$ as \emph{closed}.  We will drop the subscripts $p,q$ from $\pi_{G,p,q}$ whenever their value is clear from the context.

\subsubsection*{Boundary conditions}
For a graph $G$, one can fix an arbitrary subset of the vertices to be the boundary $\partial G\subset V(G)$ so that we can define \emph{boundary conditions} $\xi$ on $\partial G$ as follows. First augment $G$ to $G'$ by adding edges between any pair of vertices in $\partial G$ that do not share an edge in $E$; then letting $E'(\partial G)$ be the set of all edges between pairs of vertices in $\partial G$, a boundary condition $\xi$ is just an FK configuration in $\{0,1\}^{E'(\partial G)}$. Every boundary condition $\xi$ can therefore alternatively be thought of as a partition of $\partial G$ given by the clusters of $\xi$. The random cluster measure with these boundary conditions is then given by counting the number of clusters in a configuration as the number of clusters in the configuration on $G'$, fixing the restriction to $E'(\partial G)$ to be $\xi$. 

The \emph{wired} boundary condition consists of just one component consisting of all $v\in \partial G$, and the \emph{free} boundary condition consists of only singletons, each corresponding to one vertex in $\partial G$. For ease of notation, in the former case we say $\xi=1$ and in the latter case we say $\xi=0$. Denote the interior of $G$ as the subgraph of $G$ given by $G^o=(V-\partial G,E- E'(\partial G))$.

For a graph $G= (V,E)$ and a subgraph $(R,E(R))$ where $R\subset V$, denote by $\omega \restriction_{R}$ or $\omega\restriction_{E(R)}$ the restriction of the configuration $\omega\in \{0,1\}^E$ to $E(R)$. 
   
For two domains $R_1\subset R_2$, we say that a configuration $\omega$ on $R_2$ with boundary condition $\xi$ induces a boundary condition $\zeta$ on $R_1$ if $\zeta$ is the boundary condition induced by $\omega\restriction_{R_2-R_1^o}\cup \xi$: here the union of two boundary conditions denotes the partition arising from all connections through $\omega\restriction_{R_2-R_1^o}$ and $\xi$ together. In such situations, when we write $\omega\restriction_{\partial R_1}$ we mean the boundary condition induced on $R_1$ by $\omega$ on $R_2-R_1^o$ and $\xi$. If two sites $x,y$ are in the same component of a boundary condition $\xi$, we write $x\stackrel{\xi}\longleftrightarrow y$.

\subsubsection*{Domain Markov property}
For any $q$, the FK model satisfies the Domain Markov property: that is to say, for any graph $G$ and any boundary conditions $\xi$ on $\partial G$, for every subgraph $F\subset G$ and FK configuration $\eta$ on $E(G)-E(F)$,
\[\pi_{F}^{\xi\cup \eta}(\omega\in \cdot)=\pi_{G}^\xi(\omega\restriction_{E(F)}\in \cdot\mid \omega\restriction_{E(G)-E(F)}=\eta)\,.
\]

\subsubsection*{Monotonicity and FKG inequalities}

There is a natural partial ordering to configurations and boundary conditions in the FK model: for two configurations $\omega,\omega'\in \Omega$ we say $\omega\geq \omega'$ if $\omega(e)\geq \omega'(e)$ for every edge $e\in E$, and for any two boundary conditions $\xi,\xi'$ we say that $\xi\geq \xi'$ if $x\stackrel{\xi'}\longleftrightarrow y$ implies $x\stackrel{\xi}\longleftrightarrow y$ for every pair of sites $x,y\in V(\partial G)$, which is to say that $\xi'$ corresponds to a finer partition than $\xi$ of the vertices $V(\partial G)$.

An event $A$ is \emph{increasing} if it is closed under addition of edges so that if $\omega\leq \omega'$, then $\omega\in A$ implies $\omega'\in A$; analogously, it is decreasing if it is closed under removal of edges. The FK model satisfies \emph{FKG inequalities} for all $q\geq 1$ (i.e., it is positively correlated) so that for any two increasing events $A,B$,
\[\pi_G^\xi(A\cap B)\geq \pi_G ^\xi (A)\pi_G^\xi(B)\,.
\]
This leads to \emph{monotonicity in boundary conditions} for all $q\geq 1$. For any pair of boundary conditions $\xi,\xi'$ with $\xi'\leq \xi$, and any increasing event $A$,
\[\pi ^{\xi'}_G(A)\leq \pi^{\xi}_G(A)\,,
\]
whence we say that $\pi_G^{\xi}$ stochastically dominates ($\succeq$) $\pi_G^{\xi'}$.

\subsubsection*{Planar duality}
For the purposes of this paper, we now restrict our attention to graphs that are subsets of $\mathbb Z^2$, the graph with vertices at $\mathbb Z^2$ and edges between nearest-neighbors in Euclidean distance. For a connected graph $G\subset \mathbb Z^2$, let $\partial G$ consist of all $v\in V$ having a $\mathbb Z^2$-neighbor in $\mathbb Z^2-G$.

For a graph $G\subset \mathbb Z^2$ (in fact for any planar graph), there is a powerful duality between the FK model on $G$ and the FK model on the planar dual graph of $G$, denoted $G^\ast$. 
Given a planar graph $G$, we can identify to any configuration $\omega$ a dual configuration $\omega^\ast$ on $G^\ast$ where (identifying to each $e\in E(G)$, the unique dual edge $e^\ast$ passing through $e$), $\omega^\ast(e^\ast)=1$ if and only if $\omega(e)=0$. We sometimes identify edges with their midpoints.

For any boundary condition $\xi$ on a planar graph $G$, for all $q\geq 1$, the map $p\mapsto p^*$ where $p p^\ast=q(1-p)(1-p^\ast)$ can be seen to satisfy
\[\pi_{G,p,q}^{\xi}\stackrel{d}=\pi_{G^\ast,p^\ast,q}^{\xi^\ast}
\]
where the boundary condition $\xi^\ast$ is determined on a case by case basis so that $(\xi^\ast)^\ast =\xi$ (in particular, the wired and free boundary conditions are dual to each other).

\subsubsection*{Planar notation}

The graphs we consider will be rectangular subsets of $\mathbb Z^2$, denoted,
\[\Lambda_{n,m}=\llb 0,n \rrb \times \llb 0,m\rrb\,,
\]
where throughout the paper, $\llb 0,n\rrb:=\{k\in \mathbb Z: 0\leq k\leq n\}$. When $n,m$ are fixed and understood from context, we drop them from the notation. Then we denote the sides of $\partial \Lambda$ by $\partial_\west \Lambda=\{0\}\times \llb 0,m\rrb$ and the analogously defined $\partial_\north \Lambda,\partial_\south \Lambda, \partial_\east \Lambda$. We collect multiple sides into their union by including both subscripts, e.g., $\partial_{\north,\south}  \Lambda=\partial_\north\Lambda\cup \partial_\south\Lambda$.

Consider the FK model on a rectangular graph $\Lambda$. For any $x,y\in \Lambda$, we write $x\longleftrightarrow y$ if $x$ and $y$ are part of the same component of $\omega$ on $\Lambda-\partial \Lambda$ (there exists a connected set of open edges with one edge adjacent $x$ and one adjacent $y$). For a subset $R\subset \Lambda$, we write $x\stackrel{R}\longleftrightarrow y$ to denote the existence of such a crossing within $R-\partial R$, and for two sets $A,B\subset \Lambda$ we write $A\longleftrightarrow B$ if there exists $a\in A,b\in B$ such that $a\longleftrightarrow b$.

We now define the vertical crossing event for a rectangle $\Lambda$ as
\[\cC_v(\Lambda)=\partial_\south \Lambda \stackrel{\Lambda}\longleftrightarrow \partial_\north \Lambda\,,
\]
and analogously define the horizontal crossing event $\cC_h(\Lambda)$. One can similarly define the dual-crossing events $\cC^\ast_v(\Lambda),\cC^\ast_h (\Lambda)$ (where abusing notation, the fact that the crossings occur on $\Lambda^\ast$ is understood) and more generally, writing $x^\ast\stackrel{\ast}\longleftrightarrow y^\ast$ denotes the existence of a connection in the dual graph. Then, crucially, planarity and self-duality of $\mathbb Z^2$ imply that for a rectangle $\Lambda$, we have $\cC_v (\Lambda)=(\cC_h^\ast (\Lambda))^c$.

Finally for two rectangles $\Lambda'\subset \Lambda$, an annulus $A=\Lambda-\Lambda'$, denote the existence of an \emph{open circuit} (connected set of open edges with nontrivial homology w.r.t.\ $A$) by $\cC_o(A)$.

\subsubsection*{Gibbs measures and the FK phase transition}
Infinite-volume Gibbs measures can be derived by taking limits of $\pi_{\Lambda_{n,n}}^{\xi_n}$ as $n\to\infty$ for a prescribed sequence of boundary conditions $\xi_n$: natural choices of such boundary conditions are $\xi_n=1,0$ or periodic so that the graph is $(\mathbb Z/n\mathbb Z)^2$. If such limits exist weakly, we denote them by $\pi_{\mathbb Z^2}^\xi$, and they satisfy the DLR conditions (see, e.g.,~\cite{Gr04}).

By the self-duality of $\mathbb Z^2$ (up to translation), one sees that at the fixed point of $p\mapsto p^\ast$, ($p_{\sd}=\frac {\sqrt{q}}{1+\sqrt{q}}$), one has $\pi_{\mathbb Z^2}^1\stackrel{d}=\pi_{(\mathbb Z^2)^\ast}^0$, and we say the model is self-dual. The FK model for $q\geq 1$ exhibits a sharp phase transition between a high temperature phase ($p$ small) where there is no infinite component, and a low temperature phase ($p$ large) where there is almost surely an infinite component, through a critical point $p_c(q)=\inf \{p\in[0,1]:\pi_{\mathbb Z^2,p,q}(0\longleftrightarrow \infty)>0\}$.
It was proved in~\cite{BeDu12} that for all $q\geq 1$, $p_c(q)=p_{\sd}(q)$, and later in~\cite{DST15} that for all $q\in [1,4]$, we have that $\pi_{\mathbb Z^2,p_c,q}^1(0\longleftrightarrow \infty)=0$, implying $\pi_{\mathbb Z^2,p_c,q}^1\stackrel{d}= \pi_{\mathbb Z^2,p_c,q}^0$ and continuity of the phase transition (these were established much earlier for the cases of bond percolation $q=1$ and the Ising model $q=2$).

\subsubsection*{Russo--Seymour--Welsh estimates}
A key ingredient in  the proof of the continuity of the phase transition for all $q\in [1,4]$ was the following set of Russo--Seymour--Welsh (RSW) type estimates on crossing probabilities of rectangles uniform in the boundary conditions (such results were obtained for $q=1$ in~ \cite{Ru78} and for $q=2$ in~\cite{DHN11}), which were central to all available mixing time upper bounds at $p_c$ on $\Z^2$ (see~\cite{LScritical,GL16}):

\begin{theorem}[{\cite[Theorem 3]{DST15}}] \label{thm:RSW}
Let $q\in(1,4]$ and consider the critical FK model on $\Lambda_{n,n'}$ where $n'=\lfloor \alpha n\rfloor $ for some $\alpha>0$. For every $\epsilon>0$, if $R_\epsilon=\llb \epsilon n,(1-\epsilon)n \rrb \times \llb \epsilon n',(1-\epsilon)n' \rrb$, there exists a $p(\alpha,\epsilon,q)>0$ such that,
\[\pi_\Lambda^0 (\cC_v (R_\epsilon))\geq p\,.
\]
\end{theorem}

\begin{corollary}\label{cor:annulus-circuit}
Let $q\in(1,4]$ and consider the critical FK model on $\Lambda_{n,n'}$. Let $R_\epsilon$ be as in Theorem~\ref{thm:RSW}; then there exists a $p(\alpha,\epsilon,q)>0$ such that
\[\pi_{\Lambda}^0(\cC_o(\Lambda-R_\epsilon))\geq p\,.
\]
\end{corollary}

When $1< q<4$, we have the a stronger bound uniform in boundary conditions:
\begin{proposition} [{\cite[Theorem~7]{DST15}}]\label{prop:true-RSW}
Let $q\in(1,4)$ and consider the critical FK model on $\Lambda_{n,n'}$ where $n'=\lfloor \alpha n\rfloor $ for $\alpha>0$. There exists $p(\alpha,q)>0$ such that,
\[\pi_\Lambda^0 (\cC_v (\Lambda))\geq p\,.
\]
\end{proposition}
Such a bound is in fact not expected to hold for $q=4$, where, for instance, it is believed (see~\cite{DST15}) that under free boundary conditions the crossing probability goes to 0 as $N\to\infty$.

\subsection{Markov chain mixing times}\label{sub:prelim-dynamics}
In this section we introduce the dynamical notation we will be using along with several important results in the theory of Markov chain mixing times, and in particular the theory of Markov chains on monotone spin systems, that we will use in the proof of Theorem~\ref{mainthm:1}.

\subsubsection*{Mixing times} Consider a Markov chain $(X _t)_{t\geq 0}$ with finite state space $\Omega$, and (in discrete time) transition kernel $P$ with invariant measure $\pi$. In the continuous-time setup, instead of $P^t$ we consider, for $\omega_0,\omega\in\Omega$, the heat kernel
\[H_t(\omega_0,\omega)=\mathbb P_{\omega_0}(X_t=\omega)=e^{t\mathcal L}(\omega_0,\omega)\,,
\]
where $\mathbb P_{\omega_0}$ is the probability w.r.t.\ the law of the chain $(X_t)_{t\geq 0}$ given $X_0=\omega_0$, and $\mathcal L$ is the infinitesimal generator for the Markov process.

For two measures $\mu,\nu$ on $\Omega$, define the total variation distance
\[\|\mu-\nu\|_\tv=\sup_{A\subset \Omega} |\mu(A)-\nu(A)|=\inf\{\mathbb P(X\neq Y)\mid X\sim\mu,Y\sim\nu\}\,,
\]
where the infimum is over all couplings $(\mu,\nu)$.
The worst-case total variation distance of $X_t$ from $\pi$ is denoted
\[d_{\tv}(t)=\max_{\omega_0\in\Omega}\|\mathbb P_{\omega_0}(X_t\in \cdot)-\pi\|_\tv\,,
\]
and the \emph{total variation mixing time} of the Markov chain is given by (for $\epsilon\in (0,1)$),
\[\tmix(\epsilon)=\inf \{t\geq 0:d_\tv(t)\leq \epsilon\}\,.
\]

For any $\epsilon\leq \frac 14$, $\tmix(\epsilon)$ is submultiplicative and the convergence to $\pi$ in total variation distance is thenceforth exponentially fast. As such, we write $\tmix$, omitting the parameter $\epsilon$ to refer to the standard choice $\epsilon=1/(2e)$.

\subsubsection*{The FK dynamics}
The present paper is almost exclusively concerned with continuous-time \emph{heat-bath Glauber dynamics} $(X_t)_{t\geq 0}$ for the random cluster model on $\Lambda$ with boundary conditions $\xi$: this is a reversible Markov chain w.r.t.\ $\pi_{\Lambda}^\xi$ defined as follows: assign i.i.d.\ rate-$1$ Poisson clocks to every edge in $\Lambda-\partial \Lambda$; whenever the clock at an edge rings, resample its edge value according to $\pi_{\Lambda}^\xi(\omega\restriction_{e}\in \cdot\mid \omega\restriction_{\Lambda-\{e\}} =X_t\restriction_{\Lambda-\{e\}})$.
In particular, for $e=(v,w)\in \Lambda-\partial \Lambda$, the transition rate from $\omega$ to $\omega \cup \{e\}$ is
\begin{align*}
	\begin{cases}
	p & \mbox{if }v\longleftrightarrow w\mbox{ in }{\Lambda-\{e\}\cup \xi}\,,\\
	p/[p+q(1-p)] & \mbox{otherwise}\,.
	\end{cases}
\end{align*}
An alternative view of the heat-bath dynamics is the \emph{random mapping representation} of this dynamics: the edge updates correspond to a sequence $(J_i,U_i,T_i)_{i\geq 1}$, in which $T_1<T_2<\ldots$ are the clock ring times, the $J_i$'s are i.i.d.\ uniformly selected edges in $\Lambda-\partial \Lambda$, and the $U_i$'s are i.i.d.\ uniform random variables on $[0,1]$: at time $T_i$, for $J_i=(v,w)$, the dynamics replaces the value of $\omega(J_i)$ by $\one\{U_i\leq p\}$ if $v\longleftrightarrow w$ in $\Lambda-\{J_i\}\cup \xi$ and by $\one\{U_i\leq p/[p+q(1-p)]\}$ otherwise.

\subsubsection*{Monotonicity}As a result of the monotonicity of the FK model for $q\geq 1$, the heat-bath Glauber dynamics for the FK model is \emph{monotone}: for two initial configurations $\omega'\geq \omega$, we have that for all times $t\geq 0$,
\[H_t (\omega',\cdot)\succeq H_t (\omega,\cdot)\,.
\]
Using the random mapping representation, we define the \emph{grand coupling} of the set of Markov chains with all possible initial configurations, which corresponds to the identity coupling of all three random variables $(J_i,U_i,T_i)_{i\geq 1}$ amongst all the chains; for $q\geq 1$, this coupling preserves the partial ordering on initial states for all subsequent times.

The following standard fact is obtained via the grand coupling (see, e.g.,~\cite[Eq.~2.10]{MaTo10} in the context of the Ising model, as well as ~\cite{GL16} in the context of the FK model).

\begin{fact}\label{fact:init-config-comparison} Consider a set $E$ and a monotone Markov chain $(X_t)_{t\geq 0}$ on $\Omega=\{0,1\}^{E}$ with extremal configurations $\{0,1\}$. For every $t\geq 0$,
\[d_{\tv}(t) \leq |E| \|\mathbb P_1(X_t\in \cdot)-\mathbb P_0 (X_t\in \cdot)\|_\tv\,.
\]
\end{fact}
Combined with the triangle inequality one obtains for the sub-multiplicative quantity
\[\bar d_\tv(t)=\max_{\omega_1,\omega_2\in \Omega} \|P_{\omega_1}(X_t\in\cdot)-\mathbb P_{\omega_2} (X_t\in\cdot)\|_{\tv}\,,
\]
that, in the FK setting,
\begin{equation}\label{eq-dbar-upper-bound}
d_\tv(t)\leq \bar d_\tv(t) \leq 2|E(G)\|\mathbb P_1 (X_t\in \cdot)-\mathbb P_0(X_t\in\cdot)\|_\tv\,.
\end{equation}

\subsubsection*{Censoring}
Key to our proof will be the Peres--Winkler~\cite{PW13}  censoring inequality for monotone systems. While the theorem of~\cite{PW13} and its subsequent applications in e.g.,~\cite{MaTo10,LMST12} are stated for spin systems whose sites are the vertices of the underlying graph, one can view the edges as the sites by considering the appropraite line graph; it is then easy to verify that the FK Glauber dynamics satisfies the conditions of~\cite[Theorem 1.1]{PW13}. Further, while Theorem~1.1 of \cite{PW13} is stated for the discrete-time dynamics, its formulation in continuous-time follows from the same proof: see also~\cite[Theorem 2.5]{MaTo10}. 
 
\begin{theorem}[\cite{PW13}]\label{thm:censoring}
Let $\mu_T$ be the law of continuous-time Glauber dynamics at time $T$ of a monotone system on $\Lambda$ with invariant measure $\pi$, whose initial distribution $\mu_0$ is such that $\mu_0/\pi$ is increasing. Set $0=t_0 < t_1 <\ldots < t_k = T$ for some $k$, let $(B_i)_{i=1}^k$ be subsets of $\Lambda$, and let $\bar\mu_T$ be the law at time $T$ of the censored dynamics, started at $\mu_0$, where only updates within $B_i$ are kept in the time interval $[t_{i-1},t_i)$. Then $\|\mu_T-\pi\|_\tv \leq \|\bar\mu_T-\pi\|_\tv$ and $\mu_T \preceq \bar\mu_T$; moreover, $\mu_T/\pi$ and $\bar\mu_T/\pi$ are both increasing.
\end{theorem}

\subsubsection*{Boundary modifications}
Let $\xi,\xi'$ be a pair of boundary conditions on $\Lambda$ with corresponding mixing times $\tmix,\tmix'$; define \[M_{\xi,\xi'}=\| {\pi_\Lambda^\xi} / {\pi_\Lambda^{\xi'}}\|_\infty \;\vee\; \| {\pi_\Lambda^{\xi'}}/{\pi_\Lambda^{\xi}}\|_\infty\,.\]
It is well-known (see, e.g.,~\cite[Lemma~2.8]{MaTo10}) that for some $c$ independent of $n,\xi,\xi'$,
\begin{align}\label{eq:tmix-M}
\tmix\leq cM_{\xi,\xi'}^3 |E(\Lambda)|\tmix'
\end{align}
(this follows from first bounding $\tmix$ via its spectral gap, then using the variational characterization of the spectral gap: the Dirichlet form, expressed in terms of local variances, gives a factor of $M_{\xi,\xi'}^2$, and the variance produces another factor of $M_{\xi,\xi'}$).

\section{Equilibrium estimates}\label{sec:equilibrium-estimates}

In what follows, fix $q\in(1,4]$, let $p=p_c(q)$ and drop $p,q$ from the notation henceforth.

\subsection{Crossing probabilities}
In this subsection we present estimates on crossing probabilities that will be used to prove the desired mixing time bounds. The following is a slight extension of {\cite[Theorem 3.4]{GL16}}.

\begin{proposition}\label{prop:point-to-point-crossing} Let $q\in (1,4]$ and fix $\alpha \in (0,1]$. Consider the critical FK model on $\Lambda=\Lambda_{n,n'}$ with $\lfloor \alpha n\rfloor \leq n' \leq \lceil \alpha^{-1}n \rceil$. For every $\epsilon>0$, there exists $c_\star(\alpha,\epsilon, q)>0$ such that for every $x\in \llb \epsilon n,1-\epsilon n\rrb$, and every boundary condition $\xi$ on $\partial \Lambda$, one has
\[\pi^\xi_\Lambda \big(( x,0)\longleftrightarrow ( x,\lfloor n'\rfloor)\big)\gtrsim n^{-c_\star}\,.
\]
\end{proposition}
\begin{proof}
The proposition was proved in the case $n'=\lfloor \alpha n\rfloor$ in {\cite[Theorem 3.4]{GL16}} by stitching together crossings of rectangles and using the RSW estimates of Theorem~\ref{thm:RSW}. Since the crossing probabilities of Theorem~\ref{thm:RSW} are monotone in the aspect ratio, each is bounded away from zero for aspect ratios in $[\alpha,\alpha^{-1}]$, yielding the desired extension.
\end{proof}

The next two results are for  $q=4$ (Proposition~\ref{prop:true-RSW} implies both  for $1<q<4$).

\begin{lemma}\label{lem:wired/free} Let $q=4$ and fix  $\alpha\in(0,1]$. Consider the critical FK model on $\Lambda=\Lambda_{n,n'}$ with $\lfloor \alpha n \rfloor \leq n'\leq \lceil \alpha^{-1} n\rceil$ and $(1,0)$ boundary conditions denoting wired on $\partial_\south \Lambda$ and free elsewhere. For every $\epsilon>0$, there exists $p(\alpha,\epsilon)>0$ such that
\[\pi^{1,0}_{\Lambda} \left (\cC_v(\llb 0,n\rrb \times \llb 0,(1-\epsilon) n'\rrb)\right)\geq p(\epsilon).
\]
\end{lemma}
\begin{proof}Note that for an $n\times n$ square with wired boundary conditions on the $\north,\south$ sides, and free boundary conditions elsewhere, the probability of a vertical crossing is, by self-duality, $1/2$. By bounding the Radon--Nikodym derivative, it is easy to see that under the same boundary conditions but with the north and south sides disconnected from each other, the same probability is bounded below by some $p_0(q)>0$.

Moreover, by Theorem~\ref{thm:RSW} and monotonicity in boundary conditions, there exists $p_1(\epsilon)>0$ such that
\[\pi^{1,0}_{\Lambda}\left(C_h\left(\llb \tfrac\epsilon4 n,(1-\tfrac\epsilon4) n\rrb \times \llb (1-\epsilon) \alpha n,(1-\tfrac\epsilon2) \alpha n \rrb\right)\right)\geq p_1\,.
\]
The measure on $\llb(\epsilon/4)n,(1-\epsilon/4)n\rrb \times \llb 0,(1-\epsilon)n\rrb$ conditioned on the above crossing event stochastically dominates the measure induced on it by wired on the $\north,\south$ sides and free on the $\east,\west$ sides of $\llb (1-\alpha+\tfrac {\epsilon \alpha}2)\tfrac n2,(1+\alpha-\tfrac{\epsilon \alpha}2)\tfrac n2\rrb \times \llb 0, (1-\tfrac \epsilon 2)\alpha n\rrb$.
By monotonicity in boundary conditions inequality the probability of a vertical crossing in $\llb 0,n\rrb \times \llb 0,(1-\epsilon)\alpha n\rrb$ is thus bigger than $p_0 p_1$. Finally, by Corollary~\ref{cor:annulus-circuit} and monotonicity of crossing probabilities in aspect ratio, there exists $p_2(\alpha,q)>0$ such that
\[\pi_{\Lambda}^{1,0} \left(\cC_0(\Lambda-\llb (1-\alpha)\tfrac n2,(1+\alpha)\tfrac n2\rrb \times \llb (1-\epsilon) \alpha n,(1-\epsilon)n'\rrb)\right)\geq p_2\,
\]
holds for every $\lfloor \alpha n\rfloor\leq n'\leq \lceil \alpha^{-1}n\rceil$. By the FKG inequality, stitching the three crossings together implies the desired lower bound for $p=p_0p_1p_2$.
\end{proof}

\begin{corollary}\label{cor:wired-free-wired-free}
Let $q=4$ and fix  $\alpha\in (0,1]$. Consider the critical FK model on $\Lambda=\Lambda_{n,n'}$ with $\lfloor \alpha n\rfloor \leq n'\leq \lceil \alpha^{-1}n\rceil$ and boundary conditions, denoted by $(1,0,1,0)$, that are wired on $\partial_{\north,\south} \Lambda$ and free on $\partial_{\east,\west} \Lambda$. Then there exists $p(\alpha)>0$ such that
\[\pi_{\Lambda}^{1,0,1,0}(\cC_v (\Lambda))\geq p\,.
\]
\end{corollary}
\begin{proof}
For all $n'\leq n$ this follows immediately from self-duality and monotonicity in boundary conditions. For $n\leq n'\leq \lceil \alpha^{-1} n\rceil$, by monotonicity in boundary conditions and Lemma~\ref{lem:wired/free}, for any $\epsilon\in (0,1)$, there is a $p(1,\epsilon)>0$ such that,
\[\pi_{\Lambda}^{1,0,1,0}\left(\cC_v(\llb 0,n \rrb \times \llb 0,\epsilon n\rrb)\right)\geq p\,,
\]
and by reflection symmetry, $\pi_{\Lambda}^{1,0,1,0}\left(\cC_v(\llb 0,n \rrb \times \llb n'-\epsilon n, n' \rrb )\right)\geq p$.
Let
\[
A_\epsilon =\Lambda-\llb \epsilon n, (1-\epsilon)n \rrb \times \llb \epsilon n, n'-\epsilon n\rrb\,.
\]
Since $\cC_o(A_\epsilon)$ can be lower bounded by four crossings of rectangles, each of whose probabilities is monotone in the aspect ratio and thus bounded away from $0$ uniformly over $n\leq n'\leq \lceil \alpha^{-1}n\rceil$, we have that $\pi_{\Lambda}^{1,0,1,0}(\cC_v(A_\epsilon))\geq p'$ uniformly over $n\leq n'\leq \lceil \alpha^{-1} n\rceil$ for some $p'(\alpha,\epsilon)$.
Now observe that
\[\bigg(\cC_v(\llb 0,n \rrb \times \llb 0,\epsilon n\rrb)\cap \cC_v(\llb 0,n \rrb \times \llb n'-\epsilon n, n' \rrb) \cap \cC_o(A_\epsilon)\bigg)\subset \cC_v(\Lambda)\,.
\]
After fixing any small $\epsilon>0$, by the FKG inequality, there exists some $p(\alpha)>0$ such that for every $n\leq n'\leq \lceil \alpha^{-1}n\rceil$, one has $\pi_\Lambda^{1,0,1,0}(\cC_v(\Lambda))\geq p$, as required.
\end{proof}

\begin{figure}
  \hspace{-0.15in}
  \begin{tikzpicture}

    \newcommand{\xfigshift}{0pt}
    \newcommand{\yfigshift}{0pt}

      \node (plot) at (0,0){\includegraphics[width=0.8\textwidth]{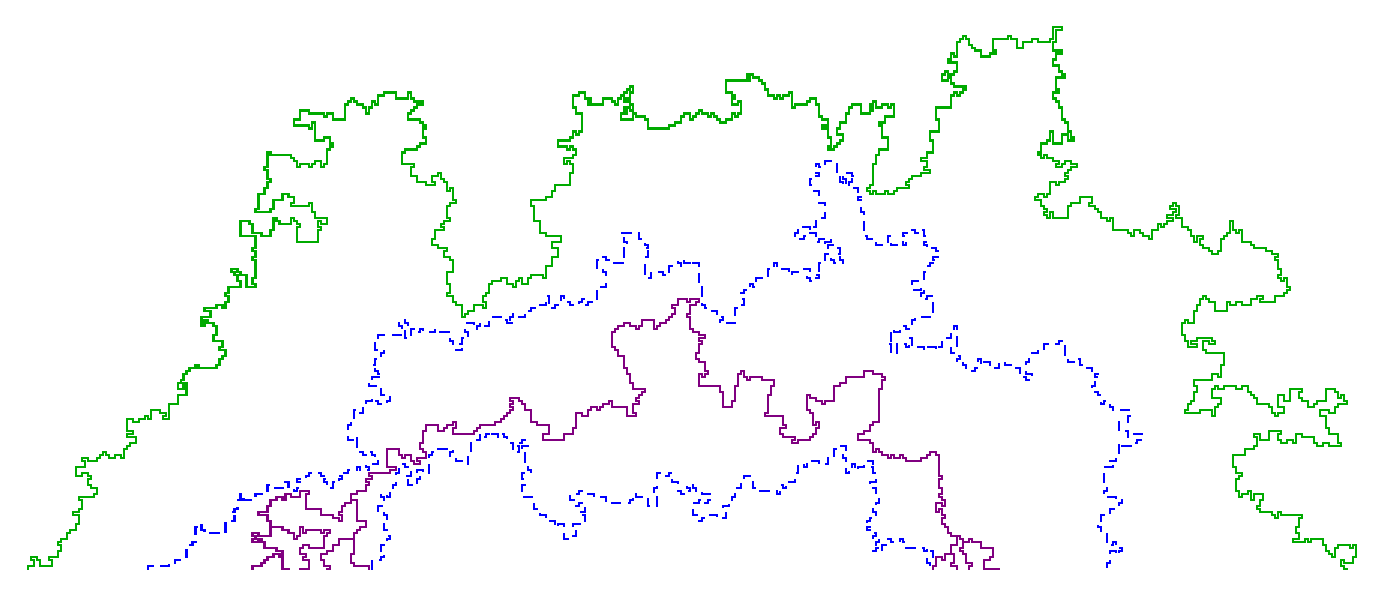}};

    \begin{scope}[shift={($(plot.south west)+(\xfigshift,\yfigshift)$)}, x={($(plot.south east)+(\xfigshift,\yfigshift)$)},y={($(plot.north west)+(\xfigshift,\yfigshift)$)}, font=\small,scale=.91]
     \draw[color=black] (0,1.05) -- (1.1,1.05);
     \draw[color=black] (0,0.075) -- (1.1,0.075);

      \node[color=DarkPurple, font=\tiny] at (.59,-.02) {$\mbox{hull}_\east(\gamma_{i})$};
      \node[color=DarkPurple, font=\tiny] at (.365,-.02) {$\mbox{hull}_\west(\gamma_{i})$};
      \node[color=DarkPurple] at (.38,.4) {$\gamma_{i}$};
      \node[color=DarkGreen] at (.1,.5) {$\gamma_{i+1}$};

     \draw[color=black,fill=black]  (.44,.075) circle (0.004) node[above right] {};
     \draw[color=black,fill=black]  (.42,.075) circle (0.004) node[above right] {};
      \node[font=\tiny] at (0.43,.1) {$e$};

     \filldraw[draw=gray, thin,opacity=0.755,fill=gray, fill opacity=0.2] (.44,.075) rectangle (.735,0.02);
      \draw[pattern=north east lines, pattern color=DarkPurple, opacity=0.5] (0.44,.075) rectangle (.735,0.02);

           \filldraw[draw=gray, thin,opacity=0.755,fill=gray, fill opacity=0.2] (.42,.075) rectangle (.3,0.025);
      \draw[pattern=north west lines, pattern color=DarkPurple, opacity=0.5] (0.42,.075) rectangle (.3,0.025);

     \filldraw[draw=gray, thin,opacity=0.5,fill=gray, fill opacity=0.15] (.44,.075) rectangle (1.057,0.045);
      \draw[pattern=north east lines, pattern color=DarkGreen, opacity=0.3] (0.44,.075) rectangle (1.057,0.045);

      \filldraw[draw=gray, thin,opacity=0.5,fill=gray, fill opacity=0.15] (.42,.075) rectangle (.035,0.045);
      \draw[pattern=north west lines, pattern color=DarkGreen, opacity=0.3] (0.42,.075) rectangle (.035,0.045);

    \end{scope}
  \end{tikzpicture}
  \caption{A pair of boundary bridges, $\gamma_i,\gamma_{i+1}$, over $e\in \partial_\north R$ induced by a configuration on $\Lambda-R$, and separated by a dual-bridge over $e$.}
  \label{fig:bridges-hull}
\end{figure}

\subsection{Boundary bridges} \label{sub:bdy-bridges}

In this subsection we define boundary bridges of the FK model and related notation. As explained in detail in~\S\ref{sub:main-techniques}, the presence of boundary bridges will be the key obstacle to coupling and, in turn, to mixing time bounds.

\begin{definition}\label{def:bridges}
Consider a rectangle $\Lambda=\Lambda_{n,n'}$ with boundary conditions $\xi$, and a connected segment $L=\llb a,b\rrb\times \{n'\} \subset \partial_\north \Lambda$. A component $\gamma\subset \partial_\north \Lambda$ of $\xi$ is a \emph{bridge over $L$} if there exist $v=(v_1,v_2),w=(w_1,w_2)\in\gamma$ such that $v\stackrel{\xi}\longleftrightarrow w$ and
\[v_1< a \qquad \mbox{and} \qquad w_1 >b\,.
\]
Note that every two distinct bridges $\gamma_1\neq\gamma_2$ over $L$ are disjoint in $\xi$. 
Denote by $\Gamma^L=\Gamma^L(\xi)$ the set of all bridges over the segment $L$.
Define bridges on subsets of $\partial \Lambda_{\south},\partial \Lambda_\east,\partial \Lambda_\west$ analogously.
\end{definition}

\begin{definition}[hull and length of a bridge]
 The \emph{west and east hulls} of
 a bridge $\gamma$ over $L=\llb a,b\rrb\times \{n'\}$ are defined as
\begin{align*}
\mbox{hull}_\west(\gamma)&=\llb \max \{x\leq a:{(x,n')\in   \gamma}\},a \rrb \times \{n'\}\,,\\
\mbox{hull}_\east(\gamma)&=\llb b,\min \{x\geq b:{(x,n')\in  \gamma}\} \rrb \times \{n'\}\,,
\end{align*}
so that the hulls of a bridge $\gamma$ are connected subsets of $\partial_\north \Lambda$ (see Fig.~\ref{fig:bridges-hull}. The \emph{west and east lengths} of $\gamma$ are defined to be
\[\ell_\west(\gamma)=|\hull_\west(\gamma)|\,, \qquad \ell_\east(\gamma)=|\hull_\east(\gamma)|\,.
\]
\end{definition}

Given the above convention, for any $L$ and $\xi$ we can define a \emph{east-ordering} of $\Gamma^L(\xi)$ as $(\gamma_1,\gamma_2,...,\gamma_{|\Gamma^L|})$ where, for all $i<j$,
\[ \ell_\east(\gamma_i) < \ell_\east(\gamma_j)\,.\]
Note that, in this ordering of the bridges, $\hull_\east(\gamma_i) \subsetneq \hull_\east(\gamma_j)$ for all $i<j$.
Define a \emph{west-ordering} of $\Gamma^L$ analogously.

\begin{definition} \label{def:planar-bridges}

For a subset $ R\subset \Lambda$, an \emph{induced boundary condition on $\partial R$} is one that can be identified with the component structure of an edge configuration $\omega\restriction_{\Lambda-R^o}$ along with the boundary condition on $\Lambda$. 

\end{definition}

Using the above definitions, and planarity, one can check the following useful facts (depicted in Fig.~\ref{fig:bridges-hull}). For concreteness we use the east-ordering of $\Gamma^L=\{\gamma_1,...,\gamma_{|\Gamma^L|}\}$.

\begin{fact}\label{fact:disjoint-primal}
Let  $\Lambda\supset R$ with boundary conditions $\xi$, and let $L\subset \partial_\north R$.
If $\gamma_i$, for $i < | \Gamma^L|$, is the $i$-th bridge in the east-ordering of $\Gamma^L$, then either the two connected components of $\partial_\north R-( \hull_\west(\gamma_i) \cup L \cup \hull_\east(\gamma_i) )$ are connected in $\Lambda-R$, or each of these components is connected to $\partial\Lambda $ in $\Lambda-R$.
\end{fact}

\begin{fact}\label{fact:disjoint-dual}
Let  $\Lambda\supset R$ with boundary conditions $\xi$, and let $L\subset \partial_\north R$.
For every two induced  bridges $\gamma_1 \neq \gamma_2$ over a segment $L$ such that $\hull_\east(\gamma_1)\subset \hull_\east(\gamma_2)$, either the two sets $(\hull_\west(\gamma_2)\xor \hull_\west(\gamma_1))$ and $(\hull_\east(\gamma_2)\xor \hull_\east(\gamma_1))$ are dual-connected in $\Lambda-R$, or each of these sets is dual-connected to $\partial\Lambda$ in $\Lambda-R$.
\end{fact}

\subsection{Estimating the number of boundary bridges}

In this section, we bound the number of distinct induced boundary bridges over a segment of $\partial R$.

When sampling boundary conditions on $R\subset \Lambda$ under $\pi_\Lambda^\xi$, the induced bridges over $e$ and all properties of them, are measurable w.r.t.\ $\omega \restriction _{\Lambda-R^o}$. For any configuration $\omega$, we denote by $\Gamma^e=\Gamma^e(\omega \restriction_{\Lambda-R^o},\xi)$ the set of all bridges over $e$ corresponding to that configuration on $\Lambda$, with the above defined west and east orderings.

The main estimate on $|\Gamma^e|$, that will be key to the proof of Theorem~\ref{mainthm:1}, is the following.

\begin{proposition}\label{prop:bridge-bound}
Let $q\in(1,4]$ and fix $\alpha\in (0,1]$. Consider the critical FK model on $\Lambda=\Lambda_{n,n'}$ with $n'\geq \lfloor \alpha n\rfloor$, along with the subset $R=\Lambda_{n,n'/2}$. There exists $c(\alpha,q)>0$ such that for every $e\in \partial_\north R$, every boundary condition $\xi$, and every $K>0$,
\begin{align}\label{eq:bridge-bound}
\pi_\Lambda^\xi(\omega: |\Gamma^e|\geq K\log n)\lesssim n^{-cK}\,.
\end{align}
Moreover, there exists $c'(\alpha,q)>0$, and for every $\epsilon>0$ there is some $K_0(\epsilon)$, such that for every $e\in \llb n^{\epsilon}, n-n^{\epsilon}\rrb \times \{\lfloor \tfrac{n'}2\rfloor\}$, every boundary condition $\xi$, and every $K<K_0$,
\[\pi_\Lambda^\xi(\omega: |\Gamma^e|\geq K\log n)\gtrsim n^{-c'K}\,.
\]
\end{proposition}

In addition to the tail behavior of $|\Gamma^e|$, we can also classify its typical behavior, showing that a fixed edge indeed has order $\log n$ bridges over it with high probability (cf.~the case of $p\neq p_c(q)$ where this quantity is typically $O(1)$).

\begin{corollary}\label{cor:typical-bridges}
Let $q\in(1,4]$ and $\alpha \in (0,1]$. Consider a rectangle $\Lambda=\Lambda_{n,n'}$ with $n'\geq \lfloor \alpha n \rfloor$, along with $R=\Lambda_{n,n'/2}$. There exists $c(\alpha,q)>0$, and for every $\epsilon>0$, there exist $K'>K(\epsilon)>0$, such that for every $e\in \llb n^{\epsilon}, n-n^{\epsilon} \rrb \times \{\lfloor \frac {n'}2\rfloor \}$ and every $\xi$,
\[\pi_{\Lambda}^\xi\left( |\Gamma^e|\notin \llb K\log n, K'\log n\rrb \right) \leq n^{-c}\,.
\]
\end{corollary}

The model at $p=p_c(q)$, $q\in (1,4]$ is believed to be scale-invariant; in line with this, having nested bridges whose length grow exponentially induces $c\log n$ clusters ranging in scale between $O(1)$ and $O(n^\epsilon)$. Indeed, this is how the lower bound of Corollary~\ref{cor:typical-bridges} is obtained. It will, therefore, be important for us to split the set $\Gamma^e$ into those bridges according to their proximity to their interior bridges, as well as the boundary.  

For the rest of this subsection, since $e$ is fixed, if $e$ is in the western half of $\partial_\north R$ then we will use the east-ordering of $\Gamma^e$ and otherwise we will use the west-ordering of $\Gamma^e$. If $e$ is in the western half of $\partial_\north R$ define the following subsets of $\Gamma^e$:
\begin{align*}
\Gamma^e_1=\Gamma^e_1(\omega\restriction_{\Lambda-R^o},\xi)= & \left\{\gamma_i\in \Gamma^e\,:\;\ell_\east(\gamma_{i-1})\leq \tfrac n6,\,\ell_\east(\gamma_i)\leq 2\ell_\east(\gamma_{i-1})\right\}\,, \\
\Gamma^e_2=\Gamma^e_2(\omega\restriction_{\Lambda-R^o},\xi)= & \left\{\gamma_i\in \Gamma^e\,:\;\ell_\east(\gamma_{i-1})\geq \tfrac n6,\,n-x-\ell_\east(\gamma_i)\geq \tfrac12 \left(n-x- \ell_\east(\gamma_{i-1})\right)
\right\}.
\end{align*}
For $e$ in the eastern half of $\partial_\north R$, define $\Gamma_1^e$ and $\Gamma_2^e$ analogously, by replacing $\ell_\east$ with $\ell_\west$ and $n-x$ with $x$.
For convenience, let $\gamma_0$ be the possibly nonexistent bridge given by the two vertices incident to the edge $e$, 
which will allow us to treat $\gamma_1$ as we would treat the other $\gamma_i$'s.

Before proving Proposition~\ref{prop:bridge-bound}, we present the two lemmas central to the upper bound~\eqref{eq:bridge-bound} of Proposition~\ref{prop:bridge-bound}, proving exponential tails on each of $|\Gamma_1^e|$ and $|\Gamma_2^e|$ beyond $O(\log n)$.  Together, these will imply the $O(\log n)$ upper bound on $|\Gamma^e|$, so we defer the proofs of the two lemmas until after completing the proof of Proposition~\ref{prop:bridge-bound} using the lemmas.  We conclude this section with a proof of Corollary~\ref{cor:typical-bridges}.

\begin{lemma}\label{lem:disjoint-1}
There exists $c_1(\alpha,q)>0$ such that for every $e\in \partial_\north R$, $\xi$ and $K> 0$,
\[\pi_\Lambda^\xi\left(\omega : |\Gamma_1^e|
\geq K\log n\right)\lesssim n^{-c_1K}\,.
\]
\end{lemma}

\begin{lemma}\label{lem:disjoint-2}
There exists $c_2(\alpha,q)>0$ such that for every $e\in \partial_\north R$, $\xi$ and $K> 0$,
\[
\pi_\Lambda^\xi\left(\omega : |\Gamma_2^e|\geq K\log n \right)\lesssim n^{-c_2K}\,.
\]
\end{lemma}

With these two lemmas in hand the proof of Proposition~\ref{prop:bridge-bound} is greatly simplified.

\begin{proof}[\textbf{\emph{Proof of Proposition~\ref{prop:bridge-bound}}}]

We begin with the upper bound. Fix an edge $e\in \partial_\north R$ and a boundary condition $\xi$ on $\partial \Lambda$. Without loss of generality suppose that $e$ is in the western half of $\partial_\north R$ and use the east-ordering of $\Gamma^e=\{\gamma_1,\gamma_2,...,\gamma_{|\Gamma^e|}\}$.
Observe that violating the second condition in $\Gamma_1^e$ means that $\ell_\east(\gamma_i)$ has at least doubled the length of its predecessor, whereas violating the second condition in $\Gamma_2^e$ means that $n-x-\ell_\east(\gamma_i)$ is at most half the corresponding quantity of its predecessor. Noting that violating the length condition of $\ell_\east(\gamma_{i-1})$ (compared to $n/6$) is disjoint between $\Gamma_1^e$ and $\Gamma_2^e$, and since $\ell_\east(\gamma_i)\leq n$ and $n-x-\ell_\east(\gamma_i)\geq 1$ for all $i$, we deterministically have
\begin{align*}
|\Gamma^e-(\Gamma_1^e\cup\Gamma_2^e)|\leq 2\log_2 n\leq 3\log n\,.
\end{align*}
Using a union bound,
\begin{align*}
\pi_\Lambda^\xi\left(|\Gamma_1^e\cup\Gamma_2^e|\geq (K-3)\log n\right)\leq  \pi_\Lambda^\xi\left(|\Gamma_1^e|\geq\tfrac {K-3}2 \log n\right)  +\pi_\Lambda ^\xi\left(|\Gamma_2^e|\geq\tfrac {K-3}2\log n\right)\,.
\end{align*}
The bounds on the two terms on the right-hand side are given by Lemmas~\ref{lem:disjoint-1}--\ref{lem:disjoint-2}, respectively. Taking the minimum of $c_1,c_2$ in those lemmas then implies that there exists $c(\alpha,q)>0$ such that
\[\pi_\Lambda^\xi(|\Gamma^e|\geq K\log n)\lesssim n^{-c(K-3)/2}\,.
\]

In order to prove the lower bound, for any $\epsilon > 0$, fix any edge $e=(x,\lfloor \tfrac {n'}2\rfloor )$ with $x\in \llb n^{\epsilon},n-n^{\epsilon}\rrb$. For $i\geq 1$, suppressing the dependence on $e$, define the sets
\begin{align}
\tilde R^\north_i & =\llb x-2^{i+1}, x+2^{i+1} \rrb \times \llb \lfloor \tfrac {n'}2\rfloor+2^i,\lfloor \tfrac {n'}2\rfloor +2^{i+1}\rrb\,, \nonumber \\
\tilde R^\east_i & =\llb x+2^i, x+2^{i+1}\rrb \times \llb \lfloor \tfrac {n'}2\rfloor-2^i,\lfloor \tfrac {n'}2\rfloor+2^{i+1}\rrb\,, \label{eq:r-tilde} \\
\tilde R^\west_i & = \llb x-2^{i+1}, x-2^i \rrb \times \llb \lfloor \tfrac {n'}2\rfloor-2^{i}, \lfloor \tfrac{n'}2\rfloor +2^{i+1}\rrb\,, \nonumber
\end{align}
and their respective subsets,
\begin{align}
 R^\north_i & =\llb x-2^{i+1}+ {2^{i-1}}, x+2^{i+1}-{2^{i-1}} \rrb \times \llb \lfloor \tfrac {n'}2  \rfloor+ 2^i,\lfloor \tfrac {n'}2\rfloor +2^{i}+2^{i-1}\rrb\,, \nonumber\\
 R^\east_i & =\llb x+2^i, x+2^{i}+2^{i-1}\rrb \times \llb \lfloor \tfrac {n'}2\rfloor,\lfloor \tfrac {n'}2\rfloor+2^{i}+2^{i-1}\rrb\,, \label{eq:r-no-tilde}\\
  R^\west_i & = \llb x-2^{i}-2^{i-1}, x-2^i \rrb \times \llb \lfloor \tfrac {n'}2\rfloor, \lfloor \tfrac{n'}2\rfloor +2^{i}+2^{i-1}\rrb\,. \nonumber
\end{align}
When $K < K_0:=\frac{\epsilon \log 4}4$, for every $i\leq 2K\log n$, we have $\tilde R_i^\west,\tilde R_i^\north,\tilde R_i^\east \subset \Lambda$. Also define the following crossing events.
\begin{align*}
\mathcal A_{i}=\cC_v(R_i^\west)\cap \cC_h(R_i^\north)\cap \cC_v(R_i^\east)\,, \\
\mathcal A^\ast_{i}=\cC_v^\ast(R_i^\west)\cap \cC_h^\ast(R_i^\north)\cap \cC_v^\ast(R_i^\east)\,.
\end{align*}
Then by definition of distinct bridges in $\Lambda-R^o$, we observe that for each $k$,
\begin{align}\label{eq:A-A-star}
\{|\Gamma^{e}|\geq K\log n\} \supset \bigcap_{i=1}^{K\log n} \mathcal A_{2i-1}\cap \mathcal A^\ast_{2i}\,.
\end{align}
By monotonicity in boundary conditions, the FKG inequality, and Theorem~\ref{thm:RSW}, there exists $p(\alpha,q)>0$ such that for every $i\leq 2K\log n$,
\begin{align*}
\pi_\Lambda^{\xi} (\mathcal A_i) & \geq \pi^0_{\tilde R_i^\west}(\cC_v(R_i^\west))\pi^0_{\tilde R_i^\north}(\cC_h(R_i^\north))\pi^0_{\tilde R_i^\east}(\cC_v(R_i^\east))\geq p\,, \\
\pi_\Lambda^{\xi} (\mathcal A^\ast_i) & \geq \pi^1_{\tilde R_i^\west}(\cC^\ast_v(R_i^\west))\pi^1_{\tilde R_i^\north}(\cC^\ast_h(R_i^\north))\pi^1_{\tilde R_i^\east}(\cC^\ast_v(R_i^\east))\geq p\,.
\end{align*}
Thus, if $K< K_0$, we have $\pi_{\Lambda}^\xi (|\Gamma^e|\geq K\log n)\geq p^{2K\log n}$, as required.
\end{proof}

We now prove Lemmas~\ref{lem:disjoint-1}--\ref{lem:disjoint-2}, whose proofs constitute the majority of the work in obtaining Proposition~\ref{prop:bridge-bound}.

\begin{proof}[\textbf{\emph{Proof of Lemma~\ref{lem:disjoint-1}}}]
Assume without loss of generality that $e=\llb x-1 , x\rrb \times \lfloor \frac {n'}2\rfloor$ is such that $x\leq \frac n2$ and use the east-ordering of $\Gamma^e=\{\gamma_1,\gamma_2,...,\gamma_{|\Gamma^e|}\}$. In order to obtain an upper tail on $|\Gamma_1^e|$, let us describe a revealing procedure for the FK configuration $\omega$ on $\Lambda-R^o$ under $\pi_\Lambda^\xi$. 

Let $F=\llb 0,n\rrb \times \llb \lfloor \frac{n'}2\rfloor ,n'\rrb = \Lambda - R^o$ (so that $\omega\restriction_{F}, \xi$ is the set of connections with respect to which the existence/properties of bridges are measurable).
We can sequentially reveal $\gamma_1,\ldots,\gamma_{|\Gamma^e|}$ by exposing the open clusters (in $F$) containing vertices of $\partial_\north \Lambda$ one at a time, starting from those adjacent to the right-vertex of $e$. Such procedures for exposing the clusters have been used in related settings (see, e.g.,~\cite{LScritical,BlSi15,GL16}); we formally describe the procedure here since in our case it involves long-range interactions imposed by the FK boundary conditions. One can reveal the open cluster $\mathcal C_v$ containing a vertex $v$ in the set $F$ by 
\begin{enumerate}
\item initializing the set $\mathcal C = \{v\}$;
\item exposing the values of $\omega$ on  all edges in $E(F)$ that contain vertices in $\mathcal C$; 
\item adding to $\mathcal C$ any vertices that are now connected by a path of open edges to $\{v\}$ (including possibly the connections imposed by the boundary condition $\xi$);
\item  repeating the process from step (1) with the new set $\mathcal C$. 
\end{enumerate}
Any open cluster containing vertices in $\partial_\north R$ on both the right and left sides of $e$ is a bridge over $e$. 
In order to reveal the first $m$ bridges over the edge $e$, we can iteratively reveal the open clusters of $\partial_\north R$ in $F$, starting initially with the cluster of $(x,\lfloor \frac {n'}{2}\rfloor)$, and continuing  to the right along $\partial_\north R$, until $m$ distinct bridges have been exposed. Using this revealing procedure, the edges which are revealed in order to expose the first $m$ bridges over $e$ are either enclosed by $\gamma_m$ and $\partial_\north R$, or belong to the outer boundary of $\gamma_m$ and are closed, thus forming a bounding dual-path.  

Let $(\mathcal F_m)$ be the filtration associated with the above revealing process for the bridges ($\cF_m$ reveals $\gamma_1,\ldots,\gamma_m$) over the edge $e$. 
Our aim is to prove that for every $m\geq 1$, 
\begin{align}\label{eq:need-to-show-bridge-1}
\pi_{\Lambda}^{\xi} (\gamma_m \in \Gamma_1^e \mid \mathcal F_{m-1}) \leq p \,,
\end{align}
(if $\gamma_m$ doesn't exist, we vacuously say $\gamma_m \notin \Gamma_1^e$) for the choice of  \begin{align}\label{eq:p-choice}
p=1-p_1p_2p_3<1\,,
\end{align} where $p_1(\alpha,q), p_2(\alpha,q),p_3(\alpha,q)>0$ are defined as follows:
\begin{itemize}
\item $p_1$ is given by Proposition~\ref{prop:true-RSW} with aspect ratio $1$ for $1<q<4$ and by Lemma~\ref{lem:wired/free} with the choice $\epsilon=1/2$ and aspect ratio $1/2$ for $q=4$,
\item $p_2$ is the probability given by Theorem~\ref{thm:RSW} for $\epsilon=1/4$ and aspect ratio $6\vee \alpha^{-1}$,
\item $p_3$ is the probability given by Theorem~\ref{thm:RSW} for $\epsilon=1/3$ and aspect ratio $1$.
\end{itemize}

Let us first conclude the proof of Lemma~\ref{lem:disjoint-1} given~\eqref{eq:need-to-show-bridge-1}. By iteratively conditioning on $(\mathcal F_{i})_{i\geq 1}$ we see that the sequence of indicators $(\mathbf 1\{\gamma_i \in \Gamma_1^e\})_{i\geq 1}$ is stochastically dominated by the i.i.d.\ sequence $(Z_i)_{i\geq 1}$ where $Z_i \sim \mbox{Bernoulli}(p)$. At the same time, by definition of the set $\Gamma_1^e$, through this revealing process, as soon as $\lceil \log_2 n\rceil$ many of the indicators $(\mathbf 1\{\gamma_i\in \Gamma_1^e\})$ are zero, all subsequent ones are deterministically zero (note that once $\ell_\east(\gamma_i)>n/6$, every subsequent bridge will also have this property). Therefore, we can bound 
\begin{align*}
\pi_\Lambda^\xi (|\Gamma_1^e| \geq r) \leq \mathbb P \big(\mbox{Bin}(r+\lceil \log_2 n\rceil -1, p)\geq r\big)
\end{align*}
which, upon taking $r = K\log n$ and using, say, Hoeffding's inequality once $K\geq 2p^{-1}$, yields the desired estimate. 

We now turn to proving the conditional estimate of~\eqref{eq:need-to-show-bridge-1}. 
First observe that by Fact~\ref{fact:disjoint-dual} and the definition of $\hull_\east(\gamma_{{m-1}})$, if $L_\west,L_\east$ are the two connected subsets of $\partial_\north R-\hull_\east(\gamma_{{m-1}})$, the event
\[\mathcal E_{m}= \left\{L_\west \stackrel{F^\ast}\longleftrightarrow L_\east \mbox{ or } L_\east\stackrel{F^\ast}\longleftrightarrow \partial \Lambda\right\}
\]
satisfies $\mathcal E_{m}\supset \{|\Gamma^e|\geq m\} \supset \{\gamma_{m}\in \Gamma_1^e\}$. 
In fact, the revealing process of $\cF_{m-1}$ reveals precisely the dual-path that bounds the open cluster of $\gamma_{m-1}$, and that dual-path is either a dual connection from $L_\west$ to $L_\east$ in $F^\ast$, or it is the west-most dual crossing from $L_\east$ to $\partial \Lambda$ that is to the right of $\gamma_{m-1}$. Either way, denote by $\zeta$ the dual-bridge/crossing revealed as such by $\cF_{m-1}$ (see Fig.~\ref{fig:disjoint-1}), and let $(z,\lfloor \frac{n'}2\rfloor)$ be the west-most point of $\zeta \cap \partial_\north R$.

Let $k = \ell_\east(\gamma_{m-1})$; in order for $\gamma_m \in \Gamma_1^e$, necessarily
$\ell_\east(\gamma_m)\leq  2k$ and
\[ (z,\lfloor \tfrac{n'}2\rfloor) \in \llb x+k,x+2k\rrb \times \{\lfloor \tfrac {n'}2\rfloor \} =: I\,.\]
We will establish the desired upper bound of~\eqref{eq:need-to-show-bridge-1} uniformly over $\mathcal F_{m-1},\zeta$ and $k$. It suffices to only consider $k\leq \frac n6$ because otherwise, $\ell_\east(\gamma_m)>\frac n6$ and therefore $\gamma_m \notin \Gamma_1^e$ deterministically. 

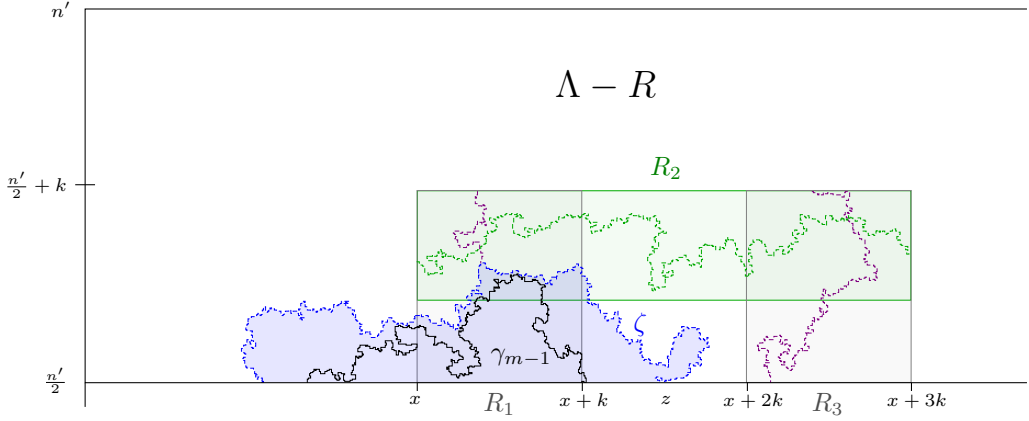
\begin{figure}
  \hspace{-0.15in}
  \begin{tikzpicture}

    \newcommand{\xfigshift}{-50pt}
    \newcommand{\yfigshift}{31.5pt}

      \node (plot) at (0,0){\includegraphics[width=0.64\textwidth]{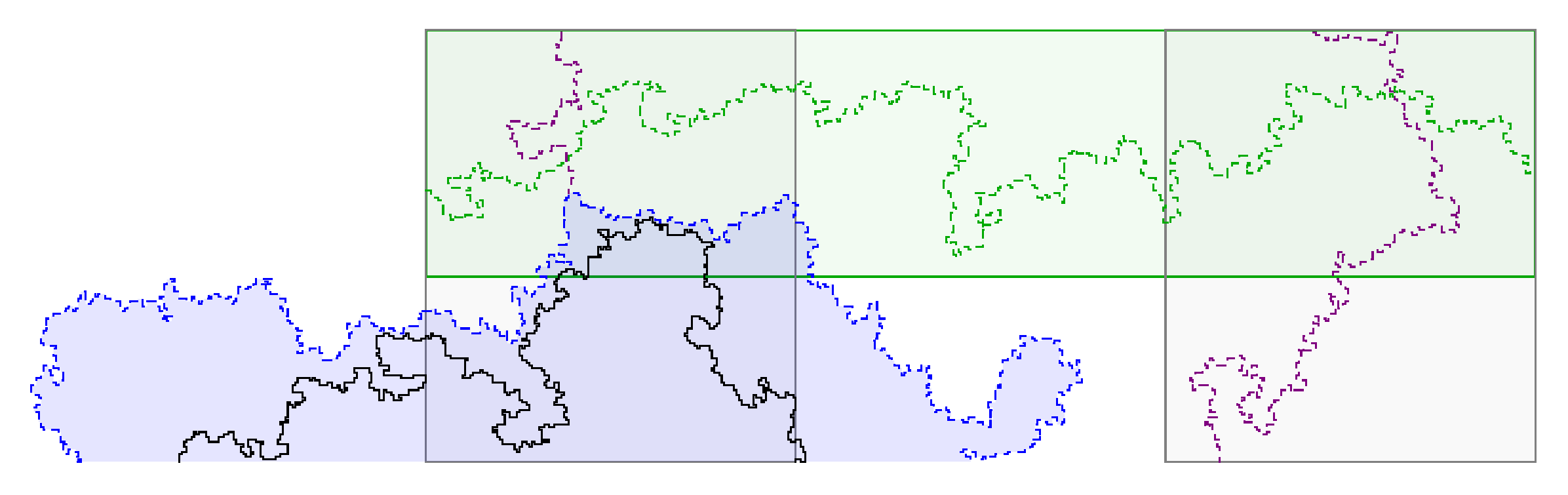}};

    \begin{scope}[shift={($(plot.south west)+(\xfigshift,\yfigshift)$)}, x={($(plot.south east)+(\xfigshift,\yfigshift)$)},y={($(plot.north west)+(\xfigshift,\yfigshift)$)}, font=\small,scale=1.38]

      \draw[color=black] (.95,.95) -- (0,.95) -- (0,-0.177) -- (.95,-0.177);
      \draw[color=black] (0.01,0.42) -- (-0.01,0.42);

      \draw[color=black] (0,-0.177) -- (0,-.25);
      \draw[color=black] (.335,-0.202) -- (.335,-0.177);
      \draw[color=black] (.5015,-0.202) -- (.5015,-0.177);
      \draw[color=black] (.668,-0.202) -- (.668,-0.177);
      \draw[color=black] (.8335,-0.202) -- (0.8335,-0.177);

      \node[font=\tiny] at (-0.03,-0.17) {$\tfrac{n'}2$};
      \node[font=\tiny] at (-0.05,0.42) {$\tfrac{n'}2+k$};
      \node[font=\tiny] at (-0.024,0.95) {$n'$};

      \node[font=\tiny] at (0.333,-0.23) {$x$};
      \node[font=\tiny] at (0.503,-0.23) {$x+k$};
      \node[font=\tiny] at (0.674,-0.23) {$x+2k$};
      \node[font=\tiny] at (0.839,-0.23) {$x+3k$};
      \node[font=\tiny] at (.586,-0.23) {$z$};
   \node[font=\Large] at (.525,.72) {$\Lambda-R$};

      \node[color=blue] at (0.56,0) {$\zeta$};
      \node[color=DarkGray] at (0.749,-.24) {$R_3$};
      \node[color=DarkGreen] at (0.586,.47) {$R_2$};
      \node[color=black] at (0.44,-0.1) {$\gamma_{{m-1}}$};
      \node[color=DarkGray] at (0.418,-.24) {$R_1$};

    \end{scope}
  \end{tikzpicture}
  \caption{After conditioning on $\zeta$ (via the configuration in the blue shaded region), the probability of the purple and green dual-crossings is greater than $p_1p_2p_3$, bounding the probability of $\{\gamma_m \in \Gamma_1^e\}$.}
  \label{fig:disjoint-1}
\end{figure}

Note that conditional on $\mathcal F_{m-1}$ (which contains the $\sigma$-algebras generated by $\zeta$ and $k$), by Fact~\ref{fact:disjoint-primal}, the event $\{\gamma_m \in \Gamma_1^e\}$ implies the event $S$, stating that either $\zeta$ is a dual-bridge and $I$ is primal-connected in $F\cup\xi$ to the left component of $\partial_\north R - \hull_\east(\zeta)$, or alternatively $\zeta$ is a dual-crossing to $\partial \Lambda$ and $I$ is primal-connected to $\partial \Lambda$ in $F$. Thus, in this conditional space,
\begin{align}\label{eq:dual-bridge-condition}
\{\gamma_m \notin \Gamma_1^e\} \supset \left\{ \zeta \stackrel{F^*} \longleftrightarrow \llb x+2k, x+3k\rrb \times \{\lfloor\tfrac {n'}2\rfloor\}\right\}\,,
\end{align}
since the right-hand side of Eq.~\eqref{eq:dual-bridge-condition} implies $S^c$ which implies the left-hand side.

In order to lower bound the probability of the last event in Eq.~\eqref{eq:dual-bridge-condition}, let $D^*$ be the outer (if $\zeta$ is a dual-crossing in $F$, then eastern) connected component of $F^*-\zeta$, and let $D$ be its dual. 
Define also the following subsets of $\Lambda$:
\begin{align*}
R_1= & \llb x,x+k\rrb \times \llb \lfloor \tfrac {n'}2\rfloor,\lfloor\tfrac{n'}2\rfloor+ \min\{k,\tfrac{\alpha n}4\}\rrb\,, \\
R_2= & \llb  x,x+3k\rrb \times \llb \lfloor\tfrac{n'}2\rfloor+\min\{\tfrac k2,\tfrac{\alpha n}8\},\lfloor\tfrac{n'}2\rfloor+\min\{k,\tfrac{\alpha n}4\}\rrb\,, \\
R_3= & \llb x+2k,x+3k\rrb \times \llb \lfloor\tfrac{n'}2\rfloor,\lfloor\tfrac{n'}2\rfloor+\min\{k,\tfrac{\alpha n}4\}\rrb\,,
\end{align*}
whereby, the event in the right-hand side of Eq.~\eqref{eq:dual-bridge-condition} can be written as $\{\zeta\stackrel{F^\ast}\longleftrightarrow \partial_\south R_3\}$.
For any $i=1,2,3$, define the following crossing events (see Fig.~\ref{fig:disjoint-1}):
\begin{align}\label{eq:Ch-Cv}
\begin{array}{l}\cC^\ast_{v} (R_i \cap D) =\left \{ \partial_\east R_i \stackrel{R_i\cap D} \nconn \partial_\west R_i \right\}\,, \\
\noalign{\medskip}
\cC^\ast_{h} (R_i \cap D) =\left \{\partial_\north R_i \stackrel{R_i\cap D} \nconn \partial_\south R_i \right\}\,.
\end{array}
\end{align}
(observe that  implicit in $(\cC^\ast_v(R_i\cap D))^c$  is the event $\{\partial_{\east} R_i \cap D\neq \emptyset\}\cap\{ \partial_\west R_i \cap D\neq \emptyset\}$, and similarly, implicit in $(\cC^\ast_h(R_i\cap D))^c$ is  the event $\{\partial_{\north} R_i \cap D\neq \emptyset \}\cap \{  \partial_\south R_i \cap D\neq \emptyset\}$).
\begin{claim}\label{claim:Gamma_1-bound}
Conditional on $\mathcal F_{m-1}$ (and in particular also $\zeta$ and $k$),
\[\{\gamma_{m} \notin \Gamma_1^e\} \supset \Big( \cC^*_v(R_1\cap D)\cap \cC_h^*(R_2 \cap D) \cap \cC^*_v(R_3\cap D)\Big)\,.
\]
\end{claim}
\begin{proof}
Suppose that $\omega$ satisfies the events on the right-hand. Recall that  $\zeta$ is such that  $\partial_\east R_3 \cap D\neq \emptyset$ and $\partial_\west R_3\cap D\neq \emptyset$, and
  $\partial_\east R_3 \nconn \partial_\west R_3$ in $R_3 \cap D$ since $\omega\in \cC_v^\ast (R_3\cap D)$.
 Consider $R_3\cap D$ with boundary conditions wired on $\partial_{\east,\west} R_3\cap D$ and free on $\zeta$ and $\partial_{\north,\south} R_3\cap D$; then the boundary conditions on $R_3\cap D$ alternate between free and wired on boundary curves ordered clockwise as $L^w_1,L^f_1,L^w_2,L^f_2...$; by planarity and the choice of generalized Dobrushin boundary conditions, for any two wired boundary curves $L^w_i,L^w_{i+1}$, either $L^w_i\longleftrightarrow L^w_{i+1}$, or  $L^f_{i} \stackrel{\ast}\longleftrightarrow L^f_{j}$ for some $j\neq i$. Picking the two wired segments of $\partial_{\east,\west}R_3\cap D$ closest to $\partial_\south R_3$, the aforementioned fact  that $\partial_\east R_3\nconn \partial_\west R_3$ in $R_3\cap D$ implies that either $\partial_\south R_3\stackrel{\ast}\longleftrightarrow \zeta$ or $\partial_\south R_3\stackrel{\ast}\longleftrightarrow \partial_\north R_3$. In the former, $\{\gamma_m\notin \Gamma_1^e\}$ holds by Eq.~\eqref{eq:dual-bridge-condition}, so suppose only the latter holds and call the dual-crossing $\zeta_3$.

Since $\partial _\south R_3\stackrel{\ast}\longleftrightarrow \partial_\north R_3$, both $\partial_\south R_2\cap D$ and $\partial_\north R_2\cap D$ are nonempty. Clearly, $\zeta_3$ splits $R_2\cap D$ into the subset to its east, $U_\east$, and that to its west, $U_\west$. Consider the set to its east, $U_\east$, with boundary conditions that are wired on $\partial_{\south,\north} R_2\cap D$ and free on $\zeta$ and on $\partial_{\east,\west} R_2\cap D$. Since $\zeta_3$ and $\zeta$ are vertex-disjoint (by our assumption that $\partial_\south R_3\stackrel{*}\nconn \zeta$ in $R_3\cap D$), and the wired boundary segments adjacent to $\zeta_3$ are disconnected in $U_\east$, it must be that either $\zeta_3\stackrel{\ast}\longleftrightarrow \zeta$ or $\zeta_3\stackrel{\ast}\longleftrightarrow \partial_\east R_2$ in $U_\east$. Using the same reasoning on $U_\west$, either $\zeta_3\stackrel{\ast}\longleftrightarrow \zeta$ or $\zeta_3\stackrel{\ast}\longleftrightarrow \partial_\west R_2$ in $U_\west$. Combining these, either $\zeta\stackrel{\ast}\longleftrightarrow \zeta_3$, in which case $\zeta\stackrel{\ast}\longleftrightarrow \partial_\south R_3$, or alternatively $\partial_\east R_2\stackrel{\ast}\longleftrightarrow \zeta_3\stackrel{\ast}\longleftrightarrow \partial_\west R_2$ in $R_2\cap D$. In the former case, by Eq.~\eqref{eq:dual-bridge-condition}, $\{\gamma_m \notin \Gamma_1^e\}$; assume therefore that only the latter case holds, and let $\zeta_2$ be a dual-crossing between $\partial_\east R_2$ to $\partial_\west R_2$ that intersects $\zeta_3$.

Finally, we can deduce that $\partial_\east R_1\cap D$ and $\partial_\west R_1\cap D$ are nonempty as $\zeta_2$ and $\zeta$ are vertex-disjoint (by our assumptions $\zeta \stackrel{*}\nconn \zeta_3$ and $\zeta_2 \stackrel*\longleftrightarrow\zeta_3$). Considering now $U_\south$, the subset of $R_1\cap D$ south of $\zeta_2$ with wired boundary conditions on $\partial_{\east,\west}R_1\cap D$ and free elsewhere, as before we deduce that either $\zeta_2\stackrel{\ast}\longleftrightarrow \zeta$ or $\zeta_2\stackrel{\ast}\longleftrightarrow \partial_\south R_1$ in $U_\south$. Since, by definition of $\zeta$, deterministically $\partial_\south R_1\cap D=\emptyset$, the former must hold, and $\zeta\stackrel{\ast}\longleftrightarrow \partial_\south R_3$ through $\zeta_2$ and $\zeta_3$, and Eq.~\eqref{eq:dual-bridge-condition} concludes the proof.
\end{proof}

We will next bound the probability of each of the events $\cC_v^*(R_1\cap D)$, $\cC_h^*(R_2\cap D)$ and $\cC_v^*(R_3\cap D)$, which, using the above claim, will translate to a bound on $\{\gamma_m \notin \Gamma_1^e\}$.

To see this, first note that by planarity, for all $i=1,2,3$ and every subset $D$,
\begin{align}\label{eq:crossing-inclusion}
\cC_{v}^\ast (R_i\cap D)\supset (\cC_h(R_i))^c = \cC_{v}^\ast(R_i)\,,
\end{align}
and likewise for horizontal crossing events.
Define the rectangle $\tilde R_1\supset  R_1$ by
\[\tilde R_1=\llb x,x+k \rrb \times \llb \lfloor \tfrac{n'}2\rfloor,n'\rrb\subset \Lambda\,.
\]
Let the boundary conditions $(1,0)$ on $\tilde R_1$ be free on $\partial_\south \tilde R_1$ and wired on $\partial _{\north,\east,\west} \tilde R_1$.
 Combining  Eq.~\eqref{eq:crossing-inclusion}, monotonicity in boundary conditions, and the domain Markov property, we get for $p_1(\alpha,q)>0$ given by Eq.~\eqref{eq:p-choice},
\begin{align*}
\pi_\Lambda^\xi\left(\cC_v^\ast(R_1 \cap D)\mid \ell_\east(\gamma_{m-1})=k,\mathcal F_{m-1},\zeta\right)& \geq \pi ^{1}_{\tilde R_1}(\cC_v^\ast(R_1\cap D)\mid\ell_\east(\gamma_{m-1})=k,\mathcal F_{m-1}, \zeta)\, \\
& \geq \pi ^{1,0}_{\tilde R_1} (\cC^\ast_v(R_1))\geq p_1\,,
\end{align*}
where the last inequality follows from Proposition~\ref{prop:true-RSW}, Lemma~\ref{lem:wired/free} and self-duality.
We stress that wiring of $\partial_{\north,\east,\west}\tilde R_1$ allowed us to ignore the information revealed on $R_1 - D$ as far as the configuration in $R_1\cap D$ is concerned, and the fact that $\omega\restriction_\zeta$ is closed allowed us to place a free boundary on $\partial_\south \tilde R_1$, supporting Lemma~\ref{lem:wired/free}.

Next, consider the rectangle $\tilde R_2\supset R_2$ defined by
\[\tilde R_2=\llb x-k,x+4k\rrb \times \llb \lfloor \tfrac{n'}2 \rfloor,n'\rrb\,,
\] so that $\tilde R_2\subset \Lambda$ since $k=\ell_\east(\gamma_{{m-1}})\leq n/6$. By monotonicity in boundary conditions and Eq.~\eqref{eq:crossing-inclusion}, we get that for the choice of $p_2(\alpha,q)>0$ given by Eq.~\eqref{eq:p-choice},
\begin{align*}
\pi ^\xi_\Lambda (\cC_h^\ast(R_2\cap D)\mid \ell_\east(\gamma_{m-1})=k,\mathcal F_{m-1},\zeta)& \geq \pi ^1_{\tilde R_2} (\cC^\ast_h(R_2\cap D)\mid \ell_\east(\gamma_{m-1})=k,\mathcal F_{m-1},\zeta)\, \\
& \geq \pi^1_{\tilde R_2} (\cC^\ast_h (R_2))\geq p_2\,.
\end{align*}
Similarly, applying the exact same treatment of $\tilde R_2$ to
\[\tilde R_3 = \llb x+k,x+4k\rrb \times \llb \tfrac{n'}4,n'\rrb\subset \Lambda\,,
\]
(it is possible to encapsulate $R_3$ by a rectangle with wired boundary conditions since $\zeta$ does not intersect $\partial_\south R_3$ in our conditional space) shows that
\begin{align*}
\pi ^\xi_\Lambda (\cC_v^\ast(R_3\cap D)\mid \ell_\east(\gamma_{m-1})=k,\mathcal F_{m-1},\zeta)& \geq \pi ^1_{\tilde R_3} (\cC^\ast_v(R_3\cap D)\mid \ell_\east(\gamma_{m-1})=k,\mathcal F_{m-1},\zeta)\, \\
& \geq \pi^1_{\tilde R_3} (\cC^\ast_v (R_3))\geq p_3\,,
\end{align*}
for $p_3(\alpha,q)>0$ as defined in Eq.~\eqref{eq:p-choice}.

Putting these all together, by the FKG inequality and Claim~\ref{claim:Gamma_1-bound},
\[
\pi_\Lambda^\xi \left( \gamma_m \notin \Gamma_1^e \given \mathcal F_{m-1} \right) \geq p_1p_2p_3\,,\]
implying the desired~\eqref{eq:need-to-show-bridge-1}, and concluding the proof. 
\end{proof}

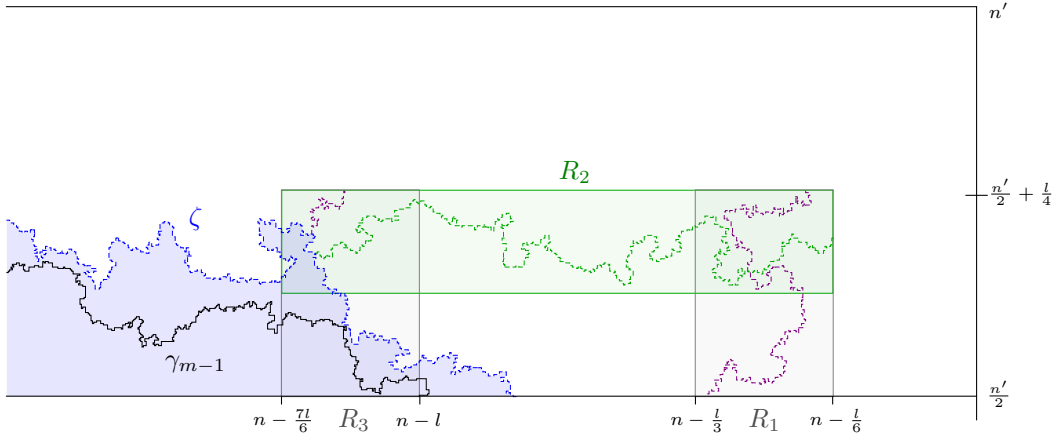
\begin{figure}
  \hspace{-0.15in}
  \begin{tikzpicture}

    \newcommand{\xfigshift}{10pt}
    \newcommand{\yfigshift}{33pt}

      \node (plot) at (0,0){\includegraphics[width=0.79\textwidth]{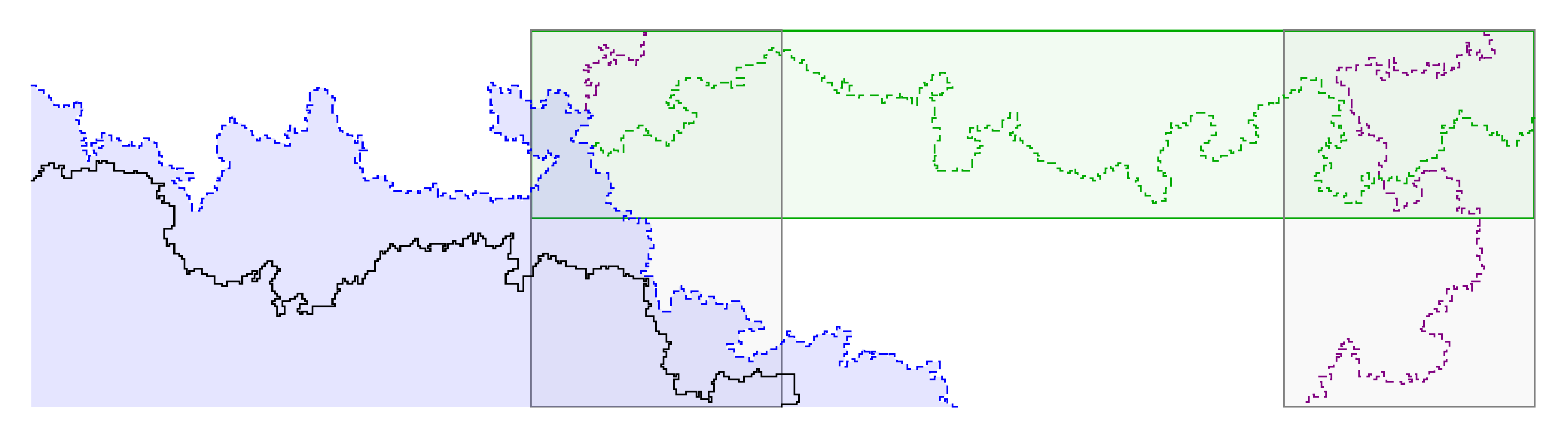}};

    \begin{scope}[shift={($(plot.south west)+(\xfigshift,\yfigshift)$)}, x={($(plot.south east)+(\xfigshift,\yfigshift)$)},y={($(plot.north west)+(\xfigshift,\yfigshift)$)}, font=\small,scale=1.265]

      \draw[color=black] (0,1) -- (0.87,1) -- (.87,-0.177) -- (0,-0.177);
      \draw[color=black] (0.88,0.43) -- (.86,0.43);

      \draw[color=black] (.87,-0.177) -- (.87,-.25);
      \draw[color=black] (.247,-0.202) -- (.247,-0.177);
      \draw[color=black] (.371,-0.202) -- (0.371,-0.177);
      \draw[color=black] (.618,-0.202) -- (.618,-0.177);
      \draw[color=black] (.741,-0.202) -- (.741,-0.177);

      \node[font=\tiny] at (.89,-0.17) {$\tfrac{n'}2$};
      \node[font=\tiny] at (.91,0.44) {$\tfrac{n'}2+\tfrac l4$};
      \node[font=\tiny] at (.89,0.99) {$n'$};

      \node[font=\tiny] at (0.249,-0.25) {$n-\tfrac{7l}6$};
      \node[font=\tiny] at (0.37,-0.25) {$n-l$};
      \node[font=\tiny] at (0.62,-0.25) {$n-\tfrac {l}3$};
      \node[font=\tiny] at (0.745,-0.25) {$n-\tfrac l6$};

      \node[color=blue] at (.17,.36) {$\zeta$};
      \node[color=DarkGray] at (0.312,-0.25) {$R_3$};
      \node[color=DarkGreen] at (.51,.5) {$R_2$};
      \node[color=black] at (0.17,-0.08) {$\gamma_{{m-1}}$};
      \node[color=DarkGray] at (0.68,-.25) {$R_1$};

    \end{scope}
  \end{tikzpicture}
  \caption{After revealing $\zeta$ (via the blue shaded region), the existence of the three dual-crossings depicted precludes $\{|\Gamma_2^e|\geq m\}$.}
  \label{fig:disjoint-2}
\end{figure}

\begin{proof}[\textbf{\emph{Proof of Lemma~\ref{lem:disjoint-2}}}]
Without loss of generality suppose $e$ is in the western half of $\partial_\north R$ and use the east-ordering of bridges so that $\Gamma^e=\{\gamma_1,...,\gamma_{|\Gamma^e|}\}$. 

The proof follows the same argument used to prove Lemma~\ref{lem:disjoint-1}. In what follows we describe the necessary modifications that are needed here. Recall the prescribed revealing process for the configuration on $F = \llb 0,n\rrb \times \lfloor \frac{n'}2\rfloor$ described in the proof of Lemma~\ref{lem:disjoint-1}; recall also that $(\mathcal F_{m})$ is the filtration corresponding to the process of sequentially revealing the distinct bridges over the edge $e$. 
Our goal is to prove the following analogue of~\eqref{eq:need-to-show-bridge-1}, that for every $m\geq 1$, 
\begin{align}\label{eq:need-to-show-bridge-2}
\pi_{\Lambda}^{\xi} (\gamma_m \in \Gamma_2^e \mid \mathcal F_{m-1}) \leq p\,,
\end{align}  
 for the choice of $p=1-p_1p_2p_3<1$ where,
\begin{itemize}
\item $p_1$ is given by Proposition~\ref{prop:true-RSW} with aspect ratio $\alpha$ for $1<q<4$ and by Lemma~\ref{lem:wired/free} with the choice $\epsilon=1/2$ and aspect ratio $\alpha/2$ for $q=4$\,,
\item $p_2$ is the probability given by Theorem~\ref{thm:RSW} for $\epsilon=1/8$ and aspect ratio $6/\alpha$\,,
\item $p_3$ is the probability given by Theorem~\ref{thm:RSW} for $\epsilon=1/3$ and aspect ratio $\alpha$\,.
\end{itemize}
Indeed, by iteratively conditioning on $(\cF_i)_{i\geq 1}$, the bound~\eqref{eq:need-to-show-bridge-2} allows us to stochastically dominate the sequence of indicators $(\mathbf 1\{\gamma_i \in \Gamma_1^e\})_{i\geq 1}$ by the i.i.d.\ sequence $(Z_i)_{i\geq 1}$ where $Z_i \sim \mbox{Bernoulli}(p)$, and moreover by definition of $\Gamma_2^e$, as soon as $\lceil \log_2 n\rceil$ of the indicators are zero, all subsequent ones are deterministically zero. The desired inequality then follows by comparison to  $\P(\mbox{Bin}(r+\lceil \log_2 n\rceil -1, p) \geq r)$ for $r=K\log n$.    

As before, we consider a fixed $m$, and let $L_\west, L_\east$ be the left and right connected components of $\partial_\north R-\hull_\east(\gamma_{m-1})$.
As in  the proof of Lemma~\ref{lem:disjoint-1},
reveal $\gamma_{m-1}$, in which case we reveal the enclosing dual-path $\zeta$ attaining 
\[
\mathcal E_{m}=\left\{L_\east \stackrel{F^\ast}\longleftrightarrow L_\west \mbox{ or } L_\east\stackrel{F^\ast}\longleftrightarrow \partial \Lambda\right\}\,,
\]
whose west-most vertex of intersection with  $\partial_\north R$ is marked by $(z,\lfloor\tfrac{n'}2\rfloor)$. Conditionally on $\mathcal F_{m-1}$, which contains the $\sigma$-algebras of $\gamma_{m-1}, \zeta$ and $\ell_\east(\gamma_{m-1})=k$,

Also for any instance of the configuration revealed by $\mathcal F_{m-1}$, we can set $k = \ell_\east(\gamma_{m-1})$ as before, and let
\[l := n - (x + \ell_\east(\gamma_{{m-1}}))\,.
\]
If $n - z< l/2$, deterministically $\gamma_m \notin \Gamma_2^e$ (as argued in the proof of Proposition~\ref{prop:bridge-bound}), hence we may assume that $n-z\geq l/2$; moreover, since $k\geq \tfrac n6$,  it must be that $\tfrac l6\leq k$. Define the following subsets of $\Lambda$:
\begin{align*}
R_1=  & \llb n-l-\tfrac l6, n-l\rrb \times \llb\lfloor \tfrac{n'}2\rfloor,\lfloor \tfrac{n'}2\rfloor+ \tfrac {\alpha l}6\rrb\,, \\
R_2= & \llb  n-l-\tfrac l6,n-\tfrac l6\rrb \times \llb \lfloor\tfrac {n'}2\rfloor,\lfloor\tfrac{n'}2\rfloor+\tfrac {\alpha l}6 \rrb\,, \\
R_3=   & \llb n-\tfrac l3,n-\tfrac l6\rrb \times \llb \lfloor\tfrac{n'}2\rfloor,\lfloor\tfrac {n'}2\rfloor+\tfrac{\alpha l}6\rrb\,.
\end{align*}
Define $\cC^\ast_{v} (R_1\cap D),\cC_h^\ast(R_2\cap D),\cC^\ast_v(R_3\cap D)$ as in Eq.~\eqref{eq:Ch-Cv}.
As in Claim~\ref{claim:Gamma_1-bound},
\[
\{\gamma_m \notin \Gamma_2^e\} \supset \Big(\cC^*_v(R_1\cap D)\cap \cC_h^*(R_2 \cap D) \cap \cC^*_v(R_3\cap D)\Big)\,.
\]
Finally, for $\tilde R_i$, $i=1,2,3$ given by
\begin{align*}
\tilde R_1 = & \llb n-l-\tfrac l6,n-l \rrb \times \llb \lfloor\tfrac{n'}2\rfloor, \lfloor\tfrac{n'}2\rfloor+\tfrac {\alpha l}3\rrb \,, \\
\tilde R_2 = & \llb n-l-\tfrac l3,n\rrb \times \llb \lfloor\tfrac{n'}2\rfloor-\tfrac{\alpha l}6, \lfloor\tfrac{n'}2\rfloor+\tfrac{\alpha l}3\rrb\,, \\
\tilde R_3 = & \llb n-\tfrac l2,n,\tfrac{l}2\rrb \times \llb\lfloor\tfrac{n'}2\rfloor-\tfrac{\alpha l}6, \lfloor\tfrac{n'}2\rfloor+\tfrac{\alpha l}3 \rrb\,,
\end{align*}
(note that all three are subsets of $\Lambda$, by the fact that $l\leq n$ and $n'\geq \lfloor \alpha n\rfloor$),
the same monotonicity argument used in the proof of Lemma~\ref{lem:disjoint-1} now implies (see Fig.~\ref{fig:disjoint-2}) that
\[
\pi_\Lambda^\xi \left( \gamma_{m-1} \notin \Gamma_2^e \given \mathcal F_{m-1} \right) \geq p_1p_2p_3\,,\]
implying~\eqref{eq:need-to-show-bridge-2} and concluding the proof. 
\end{proof}

By matching the tail estimate of Proposition~\ref{prop:bridge-bound} with a lower bound, we can straightforwardly see that an order $\log n$ bridges over a fixed edge is indeed typical. 

\begin{proof}[\textbf{\emph{Proof of Corollary~\ref{cor:typical-bridges}}}]

Fix an abitrary $\epsilon>0$,  any boundary condition $\xi$, and any $e\in \llb n^{\epsilon}, n-n^{\epsilon}\rrb \times \{\lfloor \frac{n'}{2}\rfloor \}$. For this $e$, recall the definitions of the rectangles $\tilde R_i^{\north},\tilde R_i^{\east},\tilde R_i^{\west}$ as well as their subsets $R_i^\north, R_i^\east, R_i^\west$ from~\eqref{eq:r-tilde}--\eqref{eq:r-no-tilde}. As before, when $M<M_0 := \frac{\epsilon\log 4}{4}$, for every $i\leq 2M\log n$, all these are subsets of $\Lambda$ and we can define the crossing events  
\begin{align*}
\mathcal A_{i}=\cC_v(R_i^\west)\cap \cC_h(R_i^\north)\cap \cC_v(R_i^\east)\,, \qquad \mbox{and}\qquad \mathcal A^\ast_{i}=\cC_v^\ast(R_i^\west)\cap \cC_h^\ast(R_i^\north)\cap \cC_v^\ast(R_i^\east)\,.
\end{align*}
Now for each $i\leq M\log n$, we can define the event $\chi_i : = \mathcal A_{2i-1} \cap \mathcal A_{2i}^\ast$ and notice that 
\begin{align*}
|\Gamma^e| \geq \sum_{i=1}^{M\log n} \mathbf 1\{\chi_i\}\,.
\end{align*}
Observe that for each $i$, the event $\mathcal A_i$ is measurable with respect to the configuration $\omega$ on the half-annulus $R_{i}^\west \cup R_{i}^\north \cup R_{i}^\east$. By a similar reasoning as before, there exists some $p=p(\alpha,q)>0$ such that for every $i = 1,..., 2M\log n$, and every configuration $\eta$,  
\begin{align*}
\pi_\Lambda^{\xi} (\mathcal A_i\mid \omega\restriction_{\Lambda-\tilde R_i^{\north,\east,\west}}=\eta) & \geq \pi^0_{\tilde R_i^\west}(\cC_v(R_i^\west))\pi^0_{\tilde R_i^\east}(\cC_v(R_i^\east))\pi^0_{\tilde R_i^\north}(\cC_h(R_i^\north))\geq p\,, \\
\pi_\Lambda^{\xi} (\mathcal A^\ast_i\mid \omega\restriction_{\Lambda-\tilde R_i^{\north,\east,\west}}=\eta) & \geq \pi^1_{\tilde R_i^\west}(\cC^\ast_v(R_i^\west))\pi^1_{\tilde R_i^\east}(\cC^\ast_v(R_i^\east))\pi^1_{\tilde R_i^\north}(\cC^\ast_h(R_i^\north))\geq p\,.
\end{align*}
Observe that for $i\neq j$ the interiors of $\tilde R_i^{\north,\east,\west}$ and $\tilde R_j^{\north,\east,\west}$ are disjoint. 
As a consequence, we can also deduce by the domain Markov property and monotonicity, that for every configuration $\eta$, for every $i = 1,...,M\log n$,
\begin{align*}
\pi_\Lambda^{\xi} \big(\chi_i \mid \omega\restriction_{\Lambda - \tilde R_{2i-1}^{\north,\east,\west} - \tilde R_{2i}^{\north,\east,\west}}= \eta\big) \geq p^2\,.
\end{align*}
In particular, the sequence of indicators $(\mathbf 1\{\chi_i\})_{i=1,...,M\log n}$ stochastically dominates a sequence of i.i.d.\ $\mbox{Bernoulli}(p^2)$ random variables. We therefore deduce that  
\begin{align*}
|\Gamma^e| \geq \sum_{i=1}^{M\log n} \mathbf 1\{\chi_i\} \succeq \mbox{Bin}(M\log n, p^2)
\end{align*}
Choosing $K < \frac {p^2}{2}M$, and using Hoeffding's inequality to bound the probability that the binomial random variable on the right-hand side is at most $K\log n$, we see that 
\begin{align*}
\pi_{\Lambda}^{\xi} \big(|\Gamma^e| \leq K\log n\big) \leq  \exp\big[- \tfrac 14 Mp^4 \log n\big]\,.
\end{align*} 
Combining this via a union bound with the upper bound from~\eqref{eq:bridge-bound} in Proposition~\ref{prop:bridge-bound} implies the desired. 
\end{proof}

\subsection{Disjoint crossings}
To extend our mixing time bound from favorable boundary conditions (see~\S\ref{sub:typical-bc}) to periodic boundary conditions (which are not in that class) in~\S\ref{sub:proof-mainthm}, we need an analogous bound on the number of disjoint crossings of a rectangle.

For a rectangle $R$ and a configuration $\omega \restriction_{R}$, let $\Psi_R=\Psi_R(\omega\restriction_{R})$ be the set containing  every component $A\subset V(R)$ (connected via the edges of $\omega\restriction_R$) that intersects both $\partial_\south R$ and $\partial_\north R$. We will need the following equilibrium estimate similar to Proposition~\ref{prop:bridge-bound}.

\begin{proposition}\label{prop:disjoint-crossings}
Let $q\in(1,4]$ and $\alpha\in (0,1]$. Consider the critical FK model on $\Lambda=\Lambda_{n,n'}$ with $n'\geq \lfloor \alpha n\rfloor$, and the subset $R=\llb 0,n\rrb \times \llb \frac {n'}3,\frac {2n'}3\rrb$. There exists $c(\alpha,q)>0$ such that for every boundary condition $\xi$ and every $m\geq 3$,
\[\pi_\Lambda^\xi\left(|\Psi_R|\geq m\right)\leq e^{-c m}\,.
\]
\end{proposition}
\begin{proof}
We will prove by induction that, for all $m\geq 1$,
\begin{equation}\label{eq:Psi-m-induction}
\pi_\Lambda^\xi(|\Psi_R|\geq m) \leq (1-p)^{m-2}\,,
\end{equation}
where $p>0$ is as given by Proposition~\ref{prop:true-RSW} with aspect ratio $3/\alpha$ when $1<q<4$, and is as given by Corollary~\ref{cor:wired-free-wired-free} with aspect ratio $\alpha/3$ when $q=4$.

The cases $m=1,2$ are trivially satisfied for any $0<p<1$. Now let $m\geq 3$, and suppose that Eq.~\eqref{eq:Psi-m-induction} holds for $m-1$; the proof will be concluded once we show that
\[\pi_{\Lambda}^\xi\left(|\Psi_R|\geq m\given |\Psi_R|\geq m-1\right)\leq 1-p\,.
\]
Conditioned on the existence of at least $m-1$ distinct components in $\Psi_R$, we can condition on the west-most component in $\Psi_R$ (by revealing all dual-components of $\omega\restriction_R$ incident to $\partial_\west R$, then revealing the primal-component of the adjacent primal-crossing). We can also condition on the $m-2$ east-most components in $\Psi_R$ (by successively repeating the aforementioned procedure from east to west, i.e., replacing $\partial_\west R$ above by $\partial_\east R$ to reveal some component $C \in \Psi_R$, then by its western boundary $\partial_\west C$, etc.).

Through this process, we can find two disjoint vertical dual-crossings $\zeta_1,\zeta_2$ of $R$, each one a simple dual-path; the set $(R^\ast-\zeta_1-\zeta_2)^\ast$ consists of three connected subsets of $R$; let $D$ denote the middle one. There are exactly $m-1$ elements of $\Psi_R$ in $R-D$, thus its $m$-th element, if one exists, must belong to $D$.
Since every edge in $\zeta_1\cup \zeta_2$ is dual-open, for any such choice of $\zeta_1,\zeta_2$, we then have
\[\pi_{\Lambda}^\xi\left(|\Psi_R|\geq m\given  |\Psi_R|\geq m-1,\zeta_1,\zeta_2\right)=\pi_{\Lambda}^\xi\left(\cC_v(D)\given \zeta_1,\zeta_2\right)\,,
\]
Using the domain Markov property and monotonicity of boundary conditions,
\[\pi_{\Lambda}^\xi\left(\cC_v(D)\given \zeta_1,\zeta_2\right)\leq \pi_{D}^{1,0,1,0}(\cC_v(D))\,,
\]
where $(1,0,1,0)$ boundary conditions on $D$ denote those that are free on $\zeta_1,\zeta_2$ and wired on $\partial R\cap D$.  Again by monotonicity (in boundary conditions and crossing events),
\[\pi_{D}^{1,0,1,0}(\cC_v(D))\leq
\pi_R^{1,0,1,0}\left(\cC_v(D) \given \omega_{\zeta_1}=0,\omega_{\zeta_2}=0\right) \leq
 1-\pi_R^{1,0,1,0}(\cC_h^\ast(R))\,,
\]
where, following the notation of Corollary~\ref{cor:wired-free-wired-free}, $(1,0,1,0)$ boundary conditions on a rectangle $R$ are wired on $\partial_{\north,\south} R$ and free on $\partial_{\east,\west} R$.
By monotonicity in boundary conditions and the definition of $p$, the right-hand side is bounded above by
\[1-\pi_R^{(1,0,1,0)}\big(\cC_h^\ast(\llb 0,n\rrb\times \llb \tfrac {n'}3,\tfrac {n'}3+\tfrac {\alpha n}3\rrb)\big) \leq 1-p\,. \qedhere
\]
\end{proof}

\section{Dynamical tools}\label{sec:dynamical-techniques}

In this section, we introduce the main techniques we use to control the total variation distance from stationarity for the random cluster heat-bath Glauber dynamics.

\subsection{Modifications of boundary conditions}
Crucial to the proof of Theorem~\ref{mainthm:1} is the modification of boundary bridges so that we can couple beyond FK interfaces as done in~\cite{LScritical}; in this subsection we define boundary condition modifications and control the effect such modifications can have on the mixing time.

\begin{definition}[segment modification] \label{def:modification}
Let $\xi$ be a boundary condition on a rectangle $\Lambda$ which corresponds to a partition $\{\mathcal P_1,...,\mathcal P_k\}$ of $\partial \Lambda$, and let $\Delta\subset \partial \Lambda$. The \emph{segment modification on $\Delta$}, denoted by $\xi^{\Delta}$, is the boundary condition that corresponds to the partition $\{\mathcal P_1-V(\Delta),...,\mathcal P_k-V(\Delta)\}\cup \bigcup_{v\in V(\Delta)} \{v\}$ of $\partial \Lambda$. 
\end{definition}

\begin{definition}[bridge modification] \label{def:bridge-modification}
Let $\xi$ be a boundary condition on $\partial \Lambda$, corresponding to a partition $\{\cP_1,\ldots,\cP_k\}$ of $\partial \Lambda$. Let $\Gamma^e$ be the set of disjoint bridges in $\xi\restriction_{\partial_\north \Lambda}$ over the edge $e=(x,y)\in \partial_\north\Lambda$, corresponding to the components $\{\cP_{i_j}\}_{j=1}^\ell$, as per Definition~\ref{def:bridges}.  The \emph{bridge modification} of $\xi$ over $e$, denoted  $\xi^e$, is the boundary condition associated to the partition where every $\mathcal P_{i_j}$ is split into two components,
\[\mathcal P_{i_j}^\west=\{(v_1,v_2)\in \mathcal P_{i_j}:v_1-x< 0\} \qquad and \qquad \mathcal P_{i_j}^\east=\{(v_1,v_2)\in \mathcal P_{i_j}:v_1-x >0\}\,.
\]
(Observe that, in particular, $\xi^e$ has no bridges over $e$.)
Define the bridge modification w.r.t.\ the other sides of $\partial\Lambda$ analogously.
\end{definition}

\begin{definition}[side modification] \label{def:side-modification}
Let $\xi$ be a boundary condition on  $\partial \Lambda$, corresponding to a partition $\{\cP_1,\ldots,\cP_k\}$ of $\partial \Lambda$.
The \emph{side modification} $\xi^s$ is defined as follows. Split every $\cP_j$ into its four sides, that is, for $i=\north,\south,\east,\west$, let
\[\mathcal P^i_j=\{v\in \mathcal P_j: v\in \partial _i \Lambda\}\,,
\]
where for the corner vertices the choice is arbitrary (for concreteness, associate the corner with the side that follows it clockwise). Then for every $\xi$, the modified $\xi^s$ has no components that contain vertices in more than one side of $\partial \Lambda$.
\end{definition}

It will be useful to have a notion of distance between boundary conditions.

\begin{definition}
For any pair of boundary conditions, $\xi,\xi'$ define the symmetric distance function $d(\xi,\xi')$ as follows: if $\xi''$ is the unique smallest (in the previously defined partial ordering) boundary condition with $\xi''\geq \xi$ and $\xi''\geq \xi'$, define $d(\xi,\xi')=k(\xi'')-k(\xi)+k(\xi'')-k(\xi')$, where $k(\xi)$ is the number of distinct components in the partition induced by $\xi$.
\end{definition}

If $\xi$ is a boundary condition on $\Lambda$ and $\xi'$ is a any of the above boundary modifications of $\xi$, then $\xi' \leq \xi$ and the partition associated to it is a refinement of $\xi$; this implies that $d(\xi,\xi')=k(\xi)-k(\xi')$.
One can easily verify the following.
\begin{fact}\label{fact:modification-distance}
For a segment $\Delta$, we have $d(\xi,\xi^{\Delta})\leq |V(\Delta)|$; for an edge $e$, we have $d(\xi,\xi^e)= |\Gamma^e|$; for the side modification $\xi^s$, we have that $d(\xi,\xi^s)$ is bounded above by three times the number of components in $\xi$ with vertices in multiple sides of $\partial \Lambda$.
\end{fact}

We now present a lemma bounding the effect on total variation mixing from modifying the boundary conditions. Recall that for two boundary conditions $\xi,\xi'$ on $\Lambda$, we defined in the preliminaries the quantity $M_{\xi,\xi'}=\|\tfrac {\pi_\Lambda^\xi}{\pi^{\xi'}_\Lambda}\|_\infty \vee \|\tfrac {\pi_\Lambda^{\xi'}}{\pi^{\xi}_\Lambda}\|_\infty$, and we have from Eq.~\eqref{eq:tmix-M}, that $\tmix \lesssim M_{\xi,\xi'}^3|E(\Lambda)|\tmix'$. Moreover, using the notation of~\cite{MaTo10} and~\cite{LMST12}, for an initial configuration $\omega_0$, and boundary condition $\xi$, let
\[d_\tv^{(\omega_0,\xi)}(t)=\|\mathbb P^{\xi}_{\omega_0}(X_t\in \cdot)-\pi^\xi_\Lambda\|_\tv\,,
\]
where here and throughout the paper, for any Markov chain $(X_t)_{t\geq 0}$, $\mathbb P^{\xi}_{\omega_0}(X_t \in\cdot)=\mathbb P(X_t\in \cdot\mid X_0=\omega_0)$ with boundary conditions $\xi$; when clear from the context we may drop the boundary condition superscript from the notation.

\begin{lemma}  \label{lem:bc-perturbation}
Let $\xi,\xi'$ be a pair of boundary conditions on $\partial \Lambda$. Then,
\begin{align}\label{eq:bc-pert-1}
M_{\xi,\xi'}\leq q^{d(\xi',\xi)}\,,
\end{align}
and consequently, there exists an absolute $c>0$ such that for every $t>0$,
\begin{align}\label{eq:bc-pert-2}
\max_{\omega_0\in \{0,1\}} d_\tv ^{(\omega_0,\xi)} (t) \leq 8\max_{\omega_0\in \{0,1\}} d_\tv ^{(\omega_0,\xi')}\left(c  |E(\Lambda)|^{-2} q^{-4d(\xi',\xi)}\,t\right)+\exp\left(-q^{d(\xi',\xi)}\right)\,.
\end{align}
\end{lemma}

\begin{proof}
Adapting an argument of~\cite{MaTo10} to the FK setting, Lemma 5.4 of~\cite{GL16} proves a version of this lemma for two coupled probability measures $\bP,\bP^{\Delta}$ over pairs $\xi,\xi^\Delta$. The proof for arbitrary pairs of boundary conditions, $\xi,\xi'$, is identical; letting $\bP$ be a point mass at $\xi$ completes the proof.
\end{proof}

\subsection{Censored block dynamics}\label{sub:censoring-systematic}
We next define the censored and systematic block dynamics whose coupling is the core of the dynamical analysis used to prove Theorem~\ref{mainthm:1}. This coupling may be of general interest in the study of mixing times of monotone Markov chains, where one only has control on mixing times in the presence of favorable boundary conditions.  We therefore present it in more generality than necessary for the proof of Theorem~\ref{mainthm:1}: consider the heat-bath dynamics for a monotone spin or edge system on a graph $G$ with boundary $\partial G$ that satisfies the domain Markov property and has extremal configurations $\{0,1\}$ and invariant measure $\pi_{G}^{\xi}$.

\begin{definition}[systematic block dynamics] \label{def:systematic-block} Let $B_0,\ldots,B_{s-1}$ denote a finite cover of 
 $E(G)$ (or $V(G)$ for a spin system) and for $k\geq 1$ let $i_k := (k-1)\bmod s$.

The \emph{systematic block dynamics} $(Y_k)_{k\geq0}$ is a discrete-time flavor of the block dynamics w.r.t.\ $\{B_i\}$, with blocks that are updated in a sequential deterministic order: at time~$k$, the chain
updates $B_{i_k}$ by resampling $\omega\restriction_{B_{i_k}^o} \sim \pi^\xi_G( \cdot\mid \omega\restriction_{G-(B_{i_k}^o)})$. 
\end{definition}

\begin{remark}
The systematic block dynamics as defined has unique invariant measure $\pi_G^\xi$, but it is neither time-homogenous nor reversible. If one wanted a time-homogenous and reversible analogue, one could, e.g., in each time step update all $s$ blocks sequentially in forward and then reverse order, i.e., in the order $(B_0, ..., B_{s-1}, ..., B_0)$. 
\end{remark}

\begin{definition}[censored block dynamics]\label{def:censoring-block}
Let $B_0 , ..., B_{s-1}$ be as before and consider a set $\Gamma_i$ of \emph{permissible} boundary conditions for $B_i$, and fix $\epsilon>0$.
The \emph{censored block dynamics} $(\bar{X}_t)_{t\geq 0}$ is the continuous-time single-bond (single-site) heat-bath dynamics that simulates $Y_k$ as follows. For a given $\epsilon > 0$, define
\begin{align}\label{eq:T}
 T = T(\epsilon) = \max_{i} \max_{\xi\in\Gamma_i} \tmix^{\xi,B_i}(\epsilon)\,,
\end{align}
where $\tmix^{\xi,B_i}$ is the mixing time of standard heat-bath dynamics on the block $B_i$ with boundary conditions $\xi$. Let $i_k : = (k-1) \mod s$ and let the chain $\bar{X}_t$ be obtained from the standard heat-bath dynamics by censoring, as  in Theorem~\ref{thm:censoring}, for every integer $k\geq 1$, along the interval $((k-1) T, kT]$, all updates except those in $B_{i_k}$.
\end{definition}

\begin{proposition}[comparison of censored / systematic block dynamics]\label{prop:censoring-block-dynamics}
Let $(\bar{X}_t)_{t\geq 0}$ and $(Y_k)_{k\geq 0}$ be the censored and systematic block dynamics, respectively, w.r.t.\ some blocks $B_0,\ldots B_{s-1}$ and permissible boundary conditions $\Gamma_i$ on $G$ with boundary conditions $\xi$ and initial state $\omega_0$, as per Definitions~\ref{def:systematic-block}--\ref{def:censoring-block}. Let
\begin{equation}\label{eq:def-rho}
\rho:=\max_{k\ge 1}\max_{ i\in \llb 0,s-1\rrb} \P_{\omega_0}\left(Y_k\restriction_{\partial B_i}\notin \Gamma_i \right)\,,
\end{equation}
where $Y_k\restriction_{\partial B_i}$ is the boundary condition induced on $\partial B_i$ by the configuration $Y_k$ on $G\setminus B_i^o$. Then for every $\epsilon>0$, every integer $k\geq 0$, and $T$ as in~\eqref{eq:T},
\begin{align}\label{eq-censoring-block}
\left\|\P_{\omega_0}\left(\bar{X}_{k T}\in\cdot\right)-\P_{\omega_0}\left(Y_k\in\cdot\right) \right\|_{\tv}\leq k(\rho+\epsilon)\,.
\end{align}
\end{proposition}

\begin{remark}
Although we defined the systematic and censored block dynamics for deterministic block updates, one could easily formulate the same bound for the usual block dynamics with random updates, where the $s$ sub-blocks are each assigned i.i.d.\ Poisson clocks (cf.~\cite{Ma97}), by also randomizing the order in which the censored block dynamics updates sub-blocks, using the identity coupling on the corresponding clocks.
\end{remark}

\begin{proof}[\textbf{\emph{Proof of Proposition~\ref{prop:censoring-block-dynamics}}}]
We now prove Eq.~\eqref{eq-censoring-block} by induction on $k$. Fix any $\omega_0$ and let $\delta_k = \left\|\P_{\omega_0}\left(\bar{X}_{k T}\in\cdot\right)-\P_{\omega_0}\left(Y_k\in\cdot\right) \right\|_{\tv}$ denote its left-hand side; observe that $\delta_0=0$ by definition, and
suppose that $\delta_k \leq k(\rho+\epsilon)$ for some $k$.
Denote by $i=i_{k+1}$ the block that is updated at time $k+1$ by the systematic block dynamics, and let $\bar{X}_t^{(i)}$ and $Y_k^{(i)}$ be the censored and systematic chains corresponding to the block sequence $(B_{(i+\ell)\bmod s})_{\ell \geq 0}$ (where the block $B_i$ is the first to be updated). By the Markov property and the triangle inequality,
\begin{align}
\delta_{k+1} &\leq \frac12 \sum_{\omega,\omega'}\bigg( \Big| \P_{\omega_0}(\bar X_{k T} = \omega')-\P_{\omega_0}(Y_k = \omega')\Big| \,\P_{\omega'}(\bar X_T^{(i)}=\omega) \nonumber\\
&\qquad\qquad \!\!\! +\Big| \P_{\omega'}(\bar X_{T}^{(i)} = \omega)-\P_{\omega'}(Y_1^{(i)} = \omega)\Big|\, \P_{\omega_0}(Y_k=\omega')\bigg) \nonumber\\
&=\delta_{k} + \sum_{\omega'} \P_{\omega_0}(Y_k=\omega')
 \left\| \P_{\omega'}(\bar X_{T}^{(i)} \in\cdot)-\P_{\omega'}(Y_1^{(i)} \in\cdot)\right\|_\tv \,.\label{eq:induct-step}
\end{align}
The last summand in~\eqref{eq:induct-step} satisfies
\begin{align*}
\sum_{\omega} &\P_{\omega_0}(Y_k=\omega)
 \left\| \P_{\omega}(\bar X_{T}^{(i)} \in\cdot)-\P_{\omega}(Y_1^{(i)} \in\cdot)\right\|_\tv \\
 &\leq \P_{\omega_0}(Y_k\restriction_{\partial B_i}\in \Gamma_i)\max_{\omega:\omega\restriction_{\partial B_i}\in\Gamma_i} \left\|\P_\omega(\bar{X}_T^{(i)}\in\cdot)-\P_\omega(Y_1^{(i)}\in\cdot)\right\|_\tv
 + \P_{\omega_0}(Y_k\restriction_{\partial B_i}\notin \Gamma_i) \\
 &\leq (1-\rho)\epsilon + \rho\,,
\end{align*}
by the definition of $T=T_1(\epsilon)$ and $\rho$; here we identified the configuration on $G-B_i^o$ with the boundary it induces on $\partial B_i$. Combined with Eq.~\eqref{eq:induct-step}, this completes the proof of Eq.~\eqref{eq-censoring-block}.
\end{proof}

\begin{remark}\label{rem:improved-cens-sys-max-state}
In the setting of Proposition~\ref{prop:censoring-block-dynamics}, when the initial state is $\omega_0 \in \{0,1\}$ (either minimal or maximal), one can obtain the following improved bound. Set
\begin{equation}\label{eq:T-relax} T=\max_{i} \max_{\xi\in\Gamma_i} \tmix^{\xi,B_i}(\omega_0\restriction_{B_i},\epsilon)\,, \end{equation}
where $\tmix^{\xi,B_i}(\omega_0,\epsilon) = \inf \{ t : d_\tv^{(\omega_0,\xi)}(t) \leq \epsilon\}$, relaxing the previous definition~\eqref{eq:T} of $T$ to only consider the initial state $\omega_0$. Let $(\bar X_t)$ be the censored block dynamics w.r.t.\ this new value of $T$, and denote by $(\bar X'_{t})$ the modification of $(\bar X_t)$ where, for every $k\geq 1$, the configuration of the block $B_{i_{k}}$ (i.e., the block that is to be updated in the interval $((k-1)T,kT]$) is reset at time $(k-1) T$ to the original value of $\omega_0$ on that block.
We claim that~\eqref{eq-censoring-block} holds\footnote{In fact,~\eqref{eq-censoring-block} is valid for $\bar X'_t$ with the relaxed $T$ in~\eqref{eq:T-relax}  for
every $\omega_0$, not just for the maximal and minimal configurations; however, it is when $\omega_0\in\{0,1\}$ that the modified dynamics $\bar X_t'$ can easily be compared to $\bar X_t$, and thereafter to $X_t$, via the censoring inequality of Theorem~\ref{thm:censoring}.}
for the relaxed value of $T$ in~\eqref{eq:T-relax} if we replace $\bar X_t$ by $\bar X'_t$.
Indeed, all the steps in the above proof of Proposition~\ref{prop:censoring-block-dynamics} remain valid up to the final inequality, at which point the fact that we consider $\bar X'_t$ (as opposed to $\bar X_t$) implies that
\[
\max_{\omega:\omega\restriction_{\partial B_i}\in\Gamma_i} \left\|\P_\omega(\bar{X}_T^{(i)}\in\cdot)-\P_\omega(Y_1^{(i)}\in\cdot)\right\|_\tv = \max_{\xi\in\Gamma_i} \left\|\P_{\omega_0\restriction_{B_i}}(\bar{X}_T^{(i)}\in\cdot)-\pi_{B_i}^\xi \right\|_\tv\,,
\]
which is at most $\epsilon$ when $T$ is as defined in~\eqref{eq:T-relax}.
\end{remark}

\section{Proof of main result}

In this section, we prove Theorem~\ref{mainthm:1} by combining the equilibrium estimates of~\S\ref{sec:equilibrium-estimates} with the dynamical tools provided in~\S\ref{sec:dynamical-techniques}. We first establish an analog of Theorem~\ref{mainthm:1}  (Theorem~\ref{thm:mainthm-fixed-bc}) for ``typical'' boundary conditions (defined in~\S\ref{sub:typical-bc} below), and then, using Proposition~\ref{prop:disjoint-crossings},
 derive from it the case of periodic boundary conditions in~\S\ref{sub:proof-mainthm}.
The effect of boundary bridges (which may foil the multiscale coupling approach, as described in~\S\ref{sub:main-techniques}) is controlled by restricting the analysis to those boundary conditions that have $O(\log n)$ bridges, and applying Proposition~\ref{prop:censoring-block-dynamics} to bound the mixing time under such  boundary conditions. We now define the favorable boundary conditions for which we prove a mixing time upper bound of $n^{O(\log n)}$.

\subsection{Typical boundary conditions}\label{sub:typical-bc}
We first define the class of ``typical" boundary conditions on a segment (e.g., $\partial_\north \Lambda$).
\begin{definition}[typical boundary conditions on a segment]\label{def:xi}
For $K> 0$, $N\geq 1$, and a segment $L$, let $\Xi_{K,N}$ be the set of boundary conditions $\xi$ on $L$ such that
\[ |\Gamma^e(\xi)|\leq K\log N\quad\mbox{ for every $e\in L$}\,.\]
\end{definition}

We will later see (as a consequence of Lemma~\ref{lem:retain-bc-1} below) that the boundary conditions on each of the sides of a box $\Lambda$ induced by the infinite-volume FK measure $\pi_{\Z^2}$  belong to the class of ``typical'' boundary conditions with high probability.

Next, we define the global property we require of typical boundary conditions.

\begin{definition}[typical boundary conditions on $\partial \Lambda$]\label{def:upsilon}
Let $\Upsilon_{K_1,K_2,N}=\Upsilon^\Lambda_{K_1,K_2,N}$ be the set of boundary conditions $\xi$ on $\partial \Lambda$ such that $\xi\restriction_{\partial_i\Lambda}\in \Xi_{K_1,N}$ for every $i=\north,\south,\east,\west$, and $\xi$ has at most $K_2\log N$ distinct components with vertices on different sides of $\partial \Lambda$.
\end{definition}

\begin{remark}\label{rem:wired-free}
The wired and free boundary conditions on a side $\partial_i \Lambda$ are always in $\Upsilon_{K_1,K_2,N}$ whenever $K_1\log N\geq 1$ and $K_2 \log N\geq 1$ (in the former all vertices are in just one component and in the latter no two vertices are in the same component).
\end{remark}

\subsection{Mixing under typical boundary conditions}\label{sub:proof-fixed-bc}

Since periodic boundary conditions are not in $\Upsilon_{K_1,K_2,N}$ for any $K_2>0$, we first bound the mixing time on rectangles $\Lambda_{N,N'}$ where $N'=\lfloor \bar \alpha N\rfloor$ for $\bar \alpha\in (0,1]$, with boundary conditions $\xi \in \Upsilon_{K_1,K_2,N}$.

\begin{theorem}\label{thm:mainthm-fixed-bc}
Let $q\in (1,4]$ and fix $\bar \alpha \in (0,1]$ and $K_1,K_2>0$. Consider the Glauber dynamics for the critical FK model on $\Lambda_{N,N'}$ with $\bar \alpha N \leq  N' \leq N$ and boundary conditions $\xi\in \Upsilon_{K_1,K_2,N}$. Then there exists $c=c(\bar \alpha,q, K_1,K_2)>0$ such that
\[\tmix\lesssim N^{c\log N}\,.
\]
\end{theorem}

Observe that if we define
\begin{equation}\label{eq-Upsilon-K-N}\Upsilon_{K,N}:=\Upsilon_{K,2K,N}\,,
\end{equation}
clearly $\Upsilon_{K_1,K_2,N}\subset \Upsilon_{\max\{K_1,K_2\},N}$, so  it suffices to consider $\Upsilon_{K,N}$ for general $K>0$.

The proof of Theorem~\ref{thm:mainthm-fixed-bc} proceeds by analyzing the censored and systematic block dynamics on $\Lambda$, obtaining good control on the systematic block dynamics using the RSW estimates of~\cite{DST15}, then comparing it to the censored block dynamics.
The choice of parameters for which we will apply Proposition~\ref{prop:censoring-block-dynamics} is the following.

\begin{definition}[block choice for censored / systematic block dynamics] \label{def:block-choice} Let $q\in (1,4]$ and for any $n' \leq n\leq N$, consider the critical FK Glauber dynamics on $\Lambda_{n,n'}$.  Let
\begin{align*}
B_\east &=\llb \tfrac n4,n\rrb \times\llb 0, n'\rrb \,, \\
B_\west &=\llb 0,\tfrac {3n}4\rrb \times\llb 0, n'\rrb \,,
\end{align*}
ordered as $B_0=B_\east,B_1=B_\west$ as in the setup of Proposition~\ref{prop:censoring-block-dynamics}.
For $K=\max\{K_1,  K_2\}$ given by Theorem~\ref{thm:mainthm-fixed-bc}, let $\Gamma_i=\Upsilon_{K,N}$ be the set of permissible boundary conditions for the block $B_i$ in $\Lambda_{n,n'}$.
\end{definition}

Before proving Theorem~\ref{thm:mainthm-fixed-bc} we will prove two lemmas that will be necessary for the application of Proposition~\ref{prop:censoring-block-dynamics}. We first introduce some preliminary notation.

For any $n\leq N$, label the following edges in $\partial \Lambda_{n,n'}$:
\[
e^\star_\south=(\lfloor \tfrac n2\rfloor+\tfrac 12,0)\,, \qquad and \qquad
e^\star_\north=(\lfloor \tfrac n2\rfloor+\tfrac 12, n' )\,.
\]

Recall the definitions of the bridge modification $\xi^{e}$ and the side modification $\xi^s$ from Definitions~\ref{def:bridge-modification}--\ref{def:side-modification}.
We will, throughout the proof of Theorem~\ref{thm:mainthm-fixed-bc}, for any boundary condition $\xi$ on $\partial \Lambda_{n,n'}$, let the modification $\xi' \leq \xi$ be given by
\begin{align}\label{eq:modified-bc}
\xi' := \xi^{e_\south^\star} \wedge \xi^{e_\north^\star} \wedge \xi^s\,,
\end{align}
i.e.,  the bridge modification of $\xi$ on $e_\south^\star$ and $e_\north^\star$, combined with the side modification $\xi^s$.

If $\Xi_{K,N},\Upsilon_{K,N}$ are the sets of boundary conditions defined in Definition~\ref{def:block-choice}, we let $\Xi'_{K,N},\Upsilon'_{K,N}$ be the sets corresponding to the modification $\xi\mapsto \xi'$ of every element in the original sets. Observe that $\Upsilon'_{K,N} \subset \Upsilon_{K,N}$ and likewise, $\Xi_{K,N}'\subset \Xi_{K,N}$.

\begin{lemma}\label{lem:systematic-bound}
Let $\alpha\in (0,1]$ and consider the systematic block dynamics $\{Y_k\}_{k\in \mathbb N}$ on $\Lambda_{n,n'}$ with $\lfloor \alpha n\rfloor \leq n'\leq n$ and blocks given by Definition~\ref{def:block-choice}. There exist $c_Y,c_\star(\alpha,q)>0$ such that for every two initial configurations $\omega_1,\omega_2$, and every boundary condition $\xi$ on $\partial \Lambda_{n,n'}$, modified to $\xi'$ by Eq.~\eqref{eq:modified-bc}, for all $k\geq 2$,
\[\|\mathbb P^{\xi'}_{\omega_1} (Y_{k}\in \cdot)-\mathbb P^{\xi'}_{\omega_2}(Y_k\in \cdot)\|_\tv \leq \exp (-c_Ykn^{-c_\star})\,.
\]
In particular, for all $k\geq 2$,
\[\max_{\omega_0}\|\mathbb P^{\xi'}_{\omega_0} (Y_{k}\in \cdot)-\pi^{\xi'}_{\Lambda_{n,n'}}\|_\tv \leq \exp (-c_Ykn^{-c_\star})\,.
\]
\end{lemma}
\begin{proof}
We construct a coupling between the two systematic block dynamics chains, starting from two arbitrary initial configurations $\omega_1,\omega_2$, as follows.  The systematic block dynamics first samples a configuration on $B_\east^o$ (the interior of $B_\east$) according to $\pi_{B_\east}^{\xi',\omega_i}$ for $i=1,2$, where $(\xi',\omega_i)$ is the boundary condition induced by $\omega_i\restriction_{B_\west-B_\east^o}\cup \xi'$ on $\partial B_\east$. By Proposition~\ref{prop:point-to-point-crossing}, applied to the box
\[ R^\ast=B_\west^*\cap B_\east^*\,,\]  and monotonicity in boundary conditions,
\[\pi_{B_\east}^1(e^\star_\south \stackrel {R^\ast}\longleftrightarrow e^\star _\north)\gtrsim n^{-c_\star}\,,
\]
where $c_\star(\min\{\frac 12,\alpha\},\epsilon=\tfrac14,q)>0$ is given by that proposition.

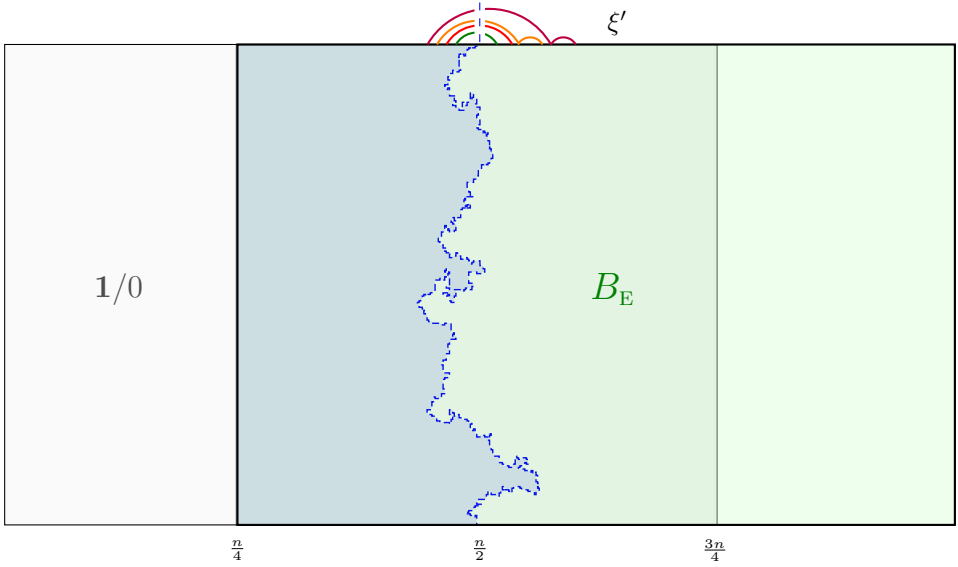
\begin{figure}
  \hspace{-0.15in}
  \begin{tikzpicture}

    \newcommand{\xfigshift}{-80pt}
    \newcommand{\yfigshift}{40pt}

      \node (plot) at (0,0){\includegraphics[width=0.46\textwidth]{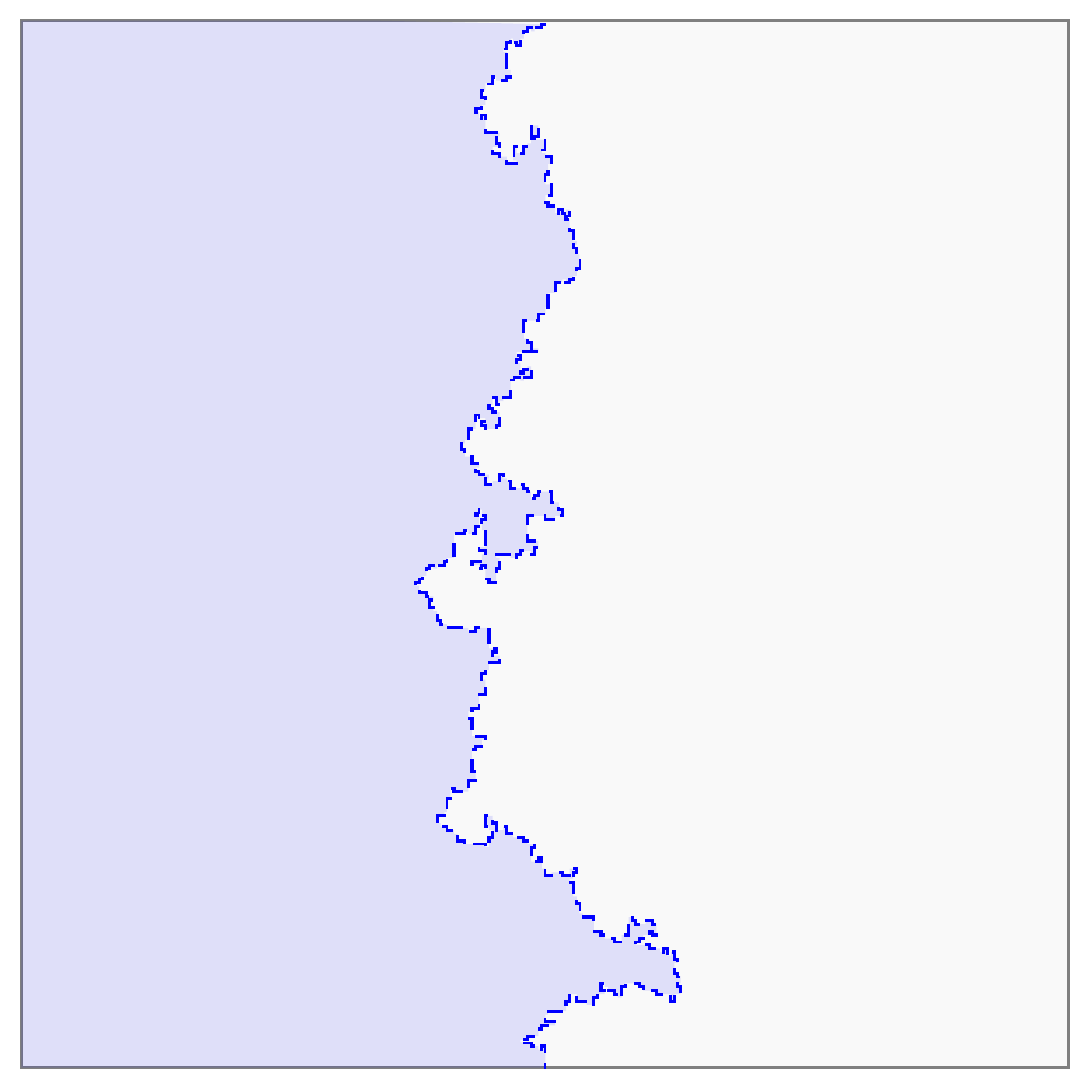}};

    \begin{scope}[shift={($(plot.south west)+(\xfigshift,\yfigshift)$)}, x={($(plot.south east)+(\xfigshift,\yfigshift)$)},y={($(plot.north west)+(\xfigshift,\yfigshift)$)}, font=\small,scale=.91]

      \draw[color=black] (2,.835) rectangle (0,-0.177) ;
      \draw[thick, draw=black, fill=green, fill opacity=.07] (2,.835) rectangle (0.49,-0.177);
      \draw[fill=gray, opacity=.04] (1.495,.835) rectangle (0,-0.177);

         \draw[thick, color=red] (.93,.835) arc (150:30:.08) ;
     \draw[thick, color=DarkGreen] (.95,.835) arc (150:30:.05) ;
      \draw[thick, color=orange] (.91,.835) arc (150:30:.1) ;
     \draw[thick, color=purple] (.89,.835) arc (150:30:.15) ;
     \draw[thick, color=purple] (1.15,.835) arc (150:30:.03) ;
      \draw[thick, color=orange] (1.08,.835) arc (150:30:.03) ;

      \draw[white, fill=white] (.99,.84) rectangle (1.01,.92);
      \draw[color=blue, dashed] (1,.84) -- (1,.93);

      \node[font=\tiny] at (0.49,-0.23) {$\tfrac n4$};
      \node[font=\tiny] at (1,-0.23) {$\tfrac n2$};
      \node[font=\tiny] at (1.495,-0.23) {$\tfrac {3n}4$};
      \node[font=\small] at (1.29,.88) {$\xi'$};

      \node[color=DarkGray,font=\large] at (0.24,0.32) {$\boldsymbol 1/0$};
      \node[font=\Large, color=DarkGreen] at (1.28,.32) {$B_\east$};

    \end{scope}
  \end{tikzpicture}
  \caption{If the depicted dual-crossing exists under any $(\xi',\omega_i)$, and the bridges over $e_\south^\ast,e_\north^\ast$ are disconnected, one can couple the two chains on the green shaded region, and in particular on $B_\east -B_\west^o$.}
  \label{fig:systematic-block-dynamics}
\end{figure}

We can condition on the west-most vertical dual-crossing between $e^\star_\south$ and $e^\star_\north$ (if such a dual-crossing exists) as follows: reveal the open components of $\partial B_\east\cap \llb 0,\lfloor \tfrac n2\rfloor \rrb \times \llb 0,n'\rrb$ as in~\cite{LScritical} or~\cite{GL16}, so that no edges in other components are revealed.  If the open components do not connect to the eastern half of $\partial R^\ast$, i.e., to $\partial R^\ast \cap \llb \lfloor \frac n2 \rfloor +1,n\rrb \times \llb 0,n'\rrb$, then it must be the case that the desired open dual-crossing exists and can be exposed without revealing any information about edges east of it.

By monotonicity in boundary conditions, if under $\pi^1_{B_\east}$ such a vertical dual-connection from $e_\south^\star$ to $e_\north^\star$ exists, the grand coupling (see~\S\ref{sub:prelim-dynamics}) ensures that the same under $\pi^{\xi',\omega_i}_{B_\east}$ for any $\omega_i \restriction_{\Lambda_{n,n'}-B^o_\east}$.
By definition of the modification $\xi'$, there are no bridges over $e_\south^\star$, no bridges over $e_\north^\star$, and no components of $\xi'$ with vertices in multiple sides of $\partial \Lambda_{n,n'}$; thus, conditional on  this vertical  dual-crossing, the following event holds:
\begin{align*}
\bigcap \left\{v\stackrel{\xi'}\nconn w: \begin{array}{rl}
v\in &\!\!\! \partial \Lambda_{n,n'} \cap\llb 0,\tfrac n2\rrb \times \llb 0,n'\rrb \\
\noalign{\medskip}
w\in &\!\!\!\partial \Lambda_{n,n'} \cap\llb \tfrac n2,n\rrb \times \llb 0,n'\rrb\ \,
\end{array}
\right\}\,.
\end{align*}
By the domain Markov property (see Fig.~\ref{fig:systematic-block-dynamics}), for any pair $\omega_1\restriction_{B_\west -B_\east^o}$ and $ \omega_2\restriction_{B_\west-B_\east^o}$,
\[\pi_{B_\east}^{\xi',\omega_1} \left(\omega \restriction _{\llb \frac{3n}4,n\rrb \times \llb 0,n'\rrb} \given e_\south^\star \stackrel{R^\ast}\longleftrightarrow e_\north^\star \right)\stackrel{d}= \pi_{B_\east}^{\xi',\omega_2} \left (\omega \restriction _{\llb \frac{3n}4,n\rrb \times \llb 0,n'\rrb}\given e_\south^\star \stackrel{R^\ast}\longleftrightarrow e_\north^\star \right)\,,
\]
using that the boundary conditions to the east of the vertical dual-crossing are the same under both measures. (In the presence of bridges over $e_\south^\star$ or $e_\north^\star$ the above distributional equality does \emph{not} hold; different configurations west of such a dual-crossing could still induce different boundary conditions east of the dual-crossing, preventing coupling (as illustrated in Fig.~\ref{fig:long-range-bc})---cf.\ the case of integer $q$ where this problem does not arise.)

This implies that, on the event $e^\star_\south \stackrel {{R^\ast}}\longleftrightarrow e^\star _\north$, the grand coupling couples the two systematic block dynamics chains so that they agree on $\Lambda_{n,n'}-B_\west^o$ with probability $1$. In this case, let $\eta$ be the resulting configuration on $B_\east -B_\west^o$, so that
\[\eta=Y_1\restriction_{B_\east -B_\west^o}\,.
\]
If the two chains were coupled on $B_\east-B_\west^o$, the boundary conditions $(\xi',\eta)$ on $\partial B_\west$ would be the same for any pair of systematic block dynamics chains with initial configurations $\omega_1,\omega_2$; in particular the identity coupling would couple them on all of $\Lambda_{n,n'}$ in the next step when $B_\west$ is resampled from $\pi_{B_\west}^{\xi',\eta}$. Thus, for some $c>0$,
\[\|\mathbb P_{\omega_1}^{\xi'}(Y_2 \in \cdot)-\mathbb P_{\omega_2}^{\xi'}(Y_2\in \cdot)\|_\tv\leq 1-c n^{-c_\star}\,.
\]

Since the systematic block dynamics is Markovian and all of the above estimates were uniform in $\omega_1$ and $\omega_2$, 
 the probability of not having coupled in time $k$ under the grand coupling is bounded above by
\[(1-cn^{-c_\star})^{\lfloor k/2\rfloor }\leq\exp(- c \lfloor k/2\rfloor  n^{-c_\star})\,. \qedhere
\]
\end{proof}

The next lemma will be key to obtaining the desired upper bound on $\rho$ as defined in~\eqref{eq:def-rho}; it shows that with high probability, the boundary conditions induced by the FK measure on a segment will be in $\Xi_{K,N}$, hence the term  ``typical" boundary conditions.

\begin{lemma}\label{lem:retain-bc-1}
Fix $q\in(1,4]$. There exists $c_\Upsilon(q)>0$ so that, for every $\Xi_{K,N}$ given by Definition~\ref{def:xi} on $\Lambda_{n,n'}$ with $ n'\leq n\leq N$ and $K>0$, and every boundary condition $\xi$,
\[\pi_{B_\east}^\xi(\omega\restriction_{\partial _\east B_\west}\notin \Xi_{K,N})\lesssim N^{-c_\Upsilon K}\,,
\]
where $\omega\restriction_{\partial_\east B_\west}$ denotes the boundary conditions induced on $\partial_\east B_\west$ by $ \omega\restriction_{B_\east -B_\west^o}\cup \xi$.
The same statement holds when exchanging $\east$ and $\west$.
\end{lemma}
\begin{proof}
By symmetry, it suffices to prove the bound for the boundary conditions on $\partial_\east B_\west$. Consider the rectangle
\[
R=\llb \tfrac n2, n \rrb \times \llb 0,n'\rrb\,.
\]
By Proposition~\ref{prop:bridge-bound} with aspect ratio $\frac 12$, there exists $c(q)=c(\alpha=\frac 12,q)>0$ such that, for every edge $e\in \partial _\east B_\west$ and every boundary condition $\eta$ on $\partial R$,
\[\pi_R^{\eta}(|\Gamma^e|\geq K\log N)\lesssim N^{-cK}\,,
\]
where, for a configuration $\omega_R$ on $R$, we recall that $|\Gamma^e|$ is the number of disjoint bridges in $\omega_R \restriction_{R-B_\west^o}\cup \xi_R$ over $e$. A union bound over all $n'$ edges on $\partial_\east B_\west$ implies that
\[\max_\eta \pi_R^{\eta}(\omega \restriction_{\partial _\east B_\west} \notin \Xi_{K,N})\lesssim n'N^{-cK}\lesssim N^{-cK+1}\,,
\]
using $n\leq N$. Consequently,
 \[\pi_{B_\east}^\xi(\omega \restriction_{\partial _\east B_\west} \notin \Xi_{K,N})=\mathbb E _{\pi_{B_\east}^\xi}\left[\pi_R^{\xi_R} (\omega \restriction_{\partial _\east B_\west} \notin \Xi_{K,N}) \right]\lesssim N^{-cK+1}\,,
 \] where the expectation is w.r.t.\ $\pi^{\xi}_{B_\east}$ over the boundary conditions $\xi_R$ induced on $R$ by  $\xi$  and the configuration on $B_\east -R^o$.
This concludes the proof of the lemma.
\end{proof}

\begin{corollary}\label{cor:retain-bc-2}
Fix $q\in(1,4]$, and consider the systematic block dynamics on $\Lambda_{n,n'}$ for $ n'\leq n\leq N$
 with block choices as given in Definition~\ref{def:block-choice}.
 There exists $c_\Upsilon(q)>0$ so that, for every fixed $K>0$ and every boundary condition $\xi'\in\Upsilon'_{K,N}$ on $\partial \Lambda_{n,n'}$,
\[\rho \lesssim N^{-c_\Upsilon K}\,,
\]
where
 $\rho$ is as defined as in~\eqref{eq:def-rho} w.r.t.\ the initial configuration $\omega_0\in \{0,1\}$ and the permissible boundary conditions $\Upsilon_{K,N}$.
\end{corollary}
\begin{proof}
Let $Y_k$ be the systematic block dynamics on $\Lambda_{n,n'}$ where $n\leq N$. Recall the definition of $\rho$ in Eq.~\eqref{eq:def-rho}, so that in the present setting,
\[\rho=\max_{\omega_0\in\{0,1\}} \max_{k\geq 1} \max_{i\in \{\east,\west\}} \mathbb P^{\xi'}_{\omega_0}(Y_k \restriction_{\partial B_i} \notin \Upsilon_{K,N})\,.
\]
In the first time step, $\omega_0\restriction_{B_\east}$ induces wired or free boundary conditions on $\partial_\west B_\east$ and so, by Remark~\ref{rem:wired-free}, the boundary condition on $\partial_\west B_\east$ is trivially in $\Xi_{K,N}$. Furthermore,  the boundary conditions on $\partial_{\north,\east,\south} B_\east$ also belong to $\Xi_{K,N}$ by the hypothesis  $\xi'\in \Upsilon_{K,N}$. Finally, there cannot be more than $2K \log  N$ components in the boundary condition on $\partial B_\east$ consisting of vertices on multiple sides for the following reason: as a result of the side modification on $\xi'$, such components can only arise from connections between $\partial_\west B_\east$ and the bridges in $\Gamma^{(n/4,0)}$ and $\Gamma^{(n/4,n')}$; however, there are at most $K\log N$ bridges in each set under \emph{any} configuration on $\Lambda-B_\east^o$ (summing to at most $2K\log N$ components, as claimed). Altogether,  $Y_1\restriction_{\partial B_\east}\in \Upsilon_{K,N}$ deterministically.

To address all subsequent time steps, by reflection symmetry and the definition of the systematic block dynamics, is suffices to consider $Y_2 \restriction_{\partial B_\west}$.
By Lemma~\ref{lem:retain-bc-1}, the probability that a boundary condition on $\partial _\east B_\west$ induced by the systematic dynamics will not be in $\Xi_{K,N}$ is $O(N^{-c_\Upsilon K})$, with $c_\Upsilon>0$ from that lemma.
The fact that, deterministically, the boundary conditions on $\partial_{\north,\south,\west} B_\west$ are in $\Xi_{K,N}$, and
there are  at most  $2K\log N$ components of the boundary condition on $\partial B_\west$ containing vertices of multiple sides of $\partial B_\west$, follows by the same reasoning argued for the first time step.
\end{proof}

We are now in a position to prove Theorem~\ref{thm:mainthm-fixed-bc}.
\begin{proof}[\textbf{\emph{Proof of Theorem~\ref{thm:mainthm-fixed-bc}}}]
Consider $\Lambda=\Lambda_{N,N'}$ with aspect ratio $\bar \alpha\in (0,1]$ and boundary conditions $\xi\in \Upsilon_{K,N}$ for a fixed
\begin{align}\label{eq:kappa}
K\geq K_0:=6(c_\star +1)\max\{c_\Upsilon^{-1},1\}\,,
\end{align}
where $c_\star=c_\star(\min\{\bar \alpha,\frac 12\},\frac 14,q)$ is the constant given by Proposition~\ref{prop:point-to-point-crossing}, and $c_\Upsilon=c_\Upsilon(q)$ is given by Corollary~\ref{cor:retain-bc-2}. It suffices to prove the proposition for all $K$ sufficiently large, as $\Upsilon_{K,N}\subset \Upsilon_{K',N}$ for every $K\leq K'$.

We prove the following inductively in $n\in\llb 1 , N \rrb$: for every $K>K_0$ as above,
every $(\bar \alpha\wedge \frac 12) n \leq n'  \leq n$, and every $\xi\in \Upsilon_{K,N}$, if
\[t_n=N^{2(c_\star +\lambda+1) \log_{4/3} n}\quad \mbox{ where }\quad \lambda:=32K \log q+5\,,
\]
then Glauber dynamics for the critical FK model on $\Lambda_{n,n'}$ has
\begin{align}\label{eq:inductive-assumption}
\|\mathbb P^{\xi}_{1} (X_{t_n}\in \cdot)-\mathbb P^{\xi}_0(X_{t_n} \in \cdot)\|_\tv \leq N^{-3}\,.
\end{align}
To see that Eq.~\eqref{eq:inductive-assumption} implies Theorem~\ref{thm:mainthm-fixed-bc}, note that~\eqref{eq-dbar-upper-bound},  with the choice $n=N$, implies that $\bar d_\tv(N^{c(\bar \alpha,q) \log N}) =O(1/N) = o(1)$ for some $c(\bar \alpha,q)>0$.

For the base case, fix a large constant $M$, where clearly $\tmix=O(1)$ for all $n\leq M$.
Next, let $m\in \llb M,N\rrb$, and assume~\eqref{eq:inductive-assumption} holds for all $n\in \llb 1,m-1\rrb$. Consider the censored  and systematic block dynamics, $(\bar X_t)_{t\geq 0}$ and $\{Y_k\}_{k\geq 0}$, respectively, on the blocks defined in Definition~\ref{def:block-choice} on $\Lambda_m=\Lambda_{m,m'}$ for some $(\bar \alpha\wedge \frac 12)m \leq m' \leq m$ and boundary conditions $\xi\in \Upsilon_{K,N}$.

Recall that $\xi\in \Upsilon_{K,N}$ has at most $K\log N$ bridges over any edge and at most $2K\log N$ components spanning multiple sides of $\partial \Lambda_m$; thus, by Fact~\ref{fact:modification-distance}, the boundary modification $\xi'$ defined in~\eqref{eq:modified-bc} satisfies $d(\xi',\xi)\leq 8K \log N$.
By the definition of $\lambda$, we have $|E|^2 q^{4d(\xi',\xi)}=o(N^{\lambda})$. Hence, by Lemma~\ref{lem:bc-perturbation} (Eq.~\eqref{eq:bc-pert-2}, where we increased the time on the right-hand to $N^\lambda$, for large enough $N$, by the monotonicity of $d_\tv$) and the above bound on $d(\xi',\xi)$, we have that for all $k,T\geq 0$,
\begin{align*}
\|\mathbb P^{\xi}_1(X_{N^{\lambda} kT}\in \cdot)-\mathbb P^{\xi}_0(X_{N^{\lambda} kT} \in \cdot)\|_\tv & \leq 2\max_{\omega_0\in\{0,1\}} \|\mathbb P_{\omega_0}^{\xi}(X_{N^{\lambda}kT}\in \cdot)-\pi_{\Lambda_m}^{\xi}\|_\tv  \\
&\leq  16\max_{\omega_0\in \{0,1\}}\|\mathbb P^{\xi'}_{\omega_0}(X_{kT}\in \cdot)-\pi_{\Lambda_{m}}^{\xi'}\|_\tv + 2e^{-N^{\lambda/4}}\,,
\end{align*}
and subsequently, by Theorem~\ref{thm:censoring},
\begin{equation}\label{eq:modification+censoring}
\|\mathbb P^{\xi}_1(X_{N^{\lambda} kT}\in \cdot)-\mathbb P^{\xi}_0(X_{N^{\lambda} kT} \in \cdot)\|_\tv \leq  16\max_{\omega_0\in \{0,1\}}\|\mathbb P^{\xi'}_{\omega_0}(\bar X_{kT}\in \cdot)-\pi_{\Lambda_{m}}^{\xi'}\|_\tv + 2e^{-N^{\lambda/4}}\,.
\end{equation}
We will next show that the first term in the right-hand above satisfies
\begin{align}\label{eq:Xbar-N3-bound}
\max_{\omega_0\in \{0,1\}}\|\mathbb P^{\xi'}_{\omega_0}(\bar X_{kT}\in \cdot)-\pi_{\Lambda_{m}}^{\xi'}\|_\tv = o(N^{-3})\,,
\end{align}
which will imply~\eqref{eq:inductive-assumption} if we choose $k,T$ such that $N^\lambda k T \leq t_m$. By triangle inequality,
\begin{align*}
\max_{\omega_0\in \{0,1\}}&\|\mathbb P^{\xi'}_{\omega_0}(\bar X_{kT}\in \cdot)-\pi_{\Lambda_{m}}^{\xi'}\|_\tv\\
 &\leq \max_{\omega_0\in\{0,1\}}\|\mathbb P^{\xi'}_{\omega_0} (\bar X_{kT}\in \cdot)-\mathbb P^{\xi'}_{\omega_0}(Y_k\in \cdot)\|_\tv  +\max_{\omega_0\in \{0,1\}}\|\mathbb P^{\xi'}_{\omega_0} (Y_{k}\in \cdot)-\pi_{\Lambda_m}^{\xi'}\|_\tv\\
&\leq  \max_{\omega_0\in\{0,1\}}\|\mathbb P^{\xi'}_{\omega_0} (\bar X_{kT}\in \cdot)-\mathbb P^{\xi'}_{\omega_0}(Y_k\in \cdot)\|_\tv  +e^{-c_Y km^{-c_\star}}\,,
\end{align*}
where the last inequality is valid for every $k\geq 2$ by Lemma~\ref{lem:systematic-bound}.
Using $\Upsilon'_{K,N}\subset \Upsilon_{K,N}$ and Proposition~\ref{prop:censoring-block-dynamics},
\begin{align*}
\max_{\omega_0\in \{0,1\}}\|\mathbb P^{\xi'}_{\omega_0} (\bar X_{kT}\in \cdot)-\pi_{\Lambda_{m}}^{\xi'}\|_\tv \leq k(\rho+\epsilon(T))+
e^{-c_Y km^{-c_\star}}\,,
\end{align*}
and so, combined with~\eqref{eq:modification+censoring},
\begin{equation}\label{eq:modification+censoring-2}
\|\mathbb P^{\xi}_1(X_{N^{\lambda} kT}\in \cdot)-\mathbb P^{\xi}_0(X_{N^{\lambda} kT} \in \cdot)\|_\tv \leq  16 
k(\rho+\epsilon(T))+
16 e^{-c_Y km^{-c_\star}}
 + 2e^{-N^{\lambda/4}}\,,
\end{equation}
where $\rho$ and $\epsilon$ were given in~\eqref{eq:T}--\eqref{eq:def-rho}, that is, in our context,
\begin{align*}
 \epsilon(T)&=  \max_{\omega'\in\Omega} \max_{i\in\{\east,\west\}} \max_{\zeta\in \Upsilon_{K,N}^{B_i}} \|\mathbb P_{\omega'}^{\zeta,B_i} (X_T\in \cdot )-\pi_{B_i}^{\zeta}\|_\tv\,,\\
\rho&=\max_{k\ge 1}\max_{ i\in\{ \east,\west\}} \P_{\omega_0}\left(Y_k\restriction_{\partial B_i}\notin \Upsilon_{K,N}^{B_i} \right)\,.
\end{align*}
We will bound $\epsilon(T)$  by the inductive assumption for the choice of
\begin{equation}\label{eq:choice-k-T}
 T:=  k t_{\left\lfloor 3m/4\right \rfloor } N^{\lambda} K \log N\,, \qquad \mbox{where}\qquad k:=c_Y^{-1} (c_\star+6)N^{c_\star} \log N\,.
\end{equation}
In order to apply the induction hypothesis for a box whose side lengths are smaller by a constant factor vs.\ the original dimensions of $m\times m'$, we repeat the above analysis for the sub-block $B_i$ (whose dimensions are $\lfloor \frac34 m\rfloor \times m'$), and get from Fact~\ref{fact:init-config-comparison} and the above arguments that
\begin{align*}
\epsilon(T) &\lesssim N^2 \max_{i\in\{\east,\west\}} \max_{\zeta\in \Upsilon_{K,N}^{B_i}} \|\mathbb P_{0}^{\zeta,B_i} (X_T\in \cdot )-\mathbb P_{1}^{\zeta,B_i} (X_T\in \cdot )\|_\tv\,,
\end{align*}
which by reapplying~\eqref{eq:modification+censoring-2} at the lower scale of the $B_i$'s implies that
\begin{align*}
\epsilon(T) &\lesssim N^2 k\left(\rho' + \epsilon'(\tfrac{T}{k N^{\lambda}})\right) +
N^2 e^{-c_Y km^{-c_\star}}
 + N^2 e^{-N^{\lambda/4}}\,,
\end{align*}
where $\epsilon'(T)$ and $\rho'$ are the counterparts of $\epsilon(T)$ and $\rho$ w.r.t.\ the sub-blocks (as per Definition~\ref{def:block-choice}) of $B_i$ rotated by $\pi/2$. (N.b.\ this rotation is crucial to ensuring that the aspect ratios of the rectangles we consider remain uniformly bounded as we recurse down in scale, and consequently the coupling probabilities satisfy the same lower bound; this rotation is also what forces us to maintain ``typical" boundary conditions on all four sides of the rectangles we are considering as opposed to, say, just on $\partial_{\east,\west} \Lambda$.)
 
This yields the following new bound on~\eqref{eq:Xbar-N3-bound}:
\begin{align*}
\max_{\omega_0\in\{0,1\}} \|\mathbb P^{\xi'}_{\omega_0}(\bar X_{kT}\in \cdot)-\pi_{\Lambda_m}^{\xi'}\|_\tv\lesssim N^2 k^2 \left(\rho+\rho'+\epsilon'(\tfrac{T}{k N^{\lambda}})\right)+k N^2 e^{-c_Y km^{-c_\star}}+o(N^{-3}).
\end{align*}
Note that the dimensions of the sub-blocks of $B_i$ (those under consideration in $\epsilon'(T)$) are $\lfloor \frac34 m\rfloor \times \lfloor \frac34 m'\rfloor$.
Hence, by the inductive assumption at scale $\lfloor\frac34 m\rfloor$ and Fact~\ref{fact:init-config-comparison},
\[ \epsilon'\left(t_{\left\lfloor 3m/4\right\rfloor}\right) = O(1/N)\,,\]
which, along with~\eqref{eq-dbar-upper-bound} and the sub-multiplicativity of $ \bar d_\tv(t)$, yields that for $T$ from~\eqref{eq:choice-k-T},
\[
\epsilon'(\tfrac {T}{kN^{\lambda}}) = \epsilon'\left(t_{\left\lfloor 3m/4\right\rfloor} K\log N\right) \lesssim N^{-K} \leq
N^{-6(c_\star+1)}\,.
\]
By Corollary~\ref{cor:retain-bc-2}, we have  $\rho\lesssim N^{-c_\Upsilon K} \leq N^{-6(c_\star+1)}$ by our choice of $K_0$, and similarly for $\rho'$. So, for $k= N^{c_\star+o(1)} $ as in~\eqref{eq:choice-k-T}, $k^2\rho\lesssim N^{-4c_\star-6+o(1)}=o(N^{-5})$, and similarly, $k^2\rho' =o( N^{-5})$. Finally, this choice of $k$ guarantees that
$k N^2 \exp(-c_Y k m^{-c_\star})$ is at most $k N^{-c_\star - 4} = o(N^{-3})$.
Combining the last three displays with these bounds yields~\eqref{eq:Xbar-N3-bound}.
The proof is concluded by noting that indeed  $N^{\lambda}  k T  \leq N^{2c_\star +2\lambda+o(1)} t_{\lfloor 3m/4\rfloor} \leq t_m$.
\end{proof}

\subsection{Mixing on the torus}\label{sub:proof-mainthm}

Here we extend Theorem~\ref{thm:mainthm-fixed-bc} to the $n\times n$ torus, proving Theorem~\ref{mainthm:1}. Observe that the periodic FK boundary conditions identified with $(\mathbb Z/n\mathbb Z)^2$ in fact have order $n$ components with vertices on multiple sides of $\partial \Lambda$. We thus have to extend the bound of Theorem~\ref{thm:mainthm-fixed-bc} to periodic boundary conditions using the topological structure of $(\mathbb Z/n\mathbb Z)^2$.
The proof draws from the extension of mixing time bounds in~\cite{LScritical} and~\cite{GL16} from fixed  boundary conditions to $(\mathbb Z/n\mathbb Z)^2$. In the present setting, having to deal with a specific class of boundary conditions forces us to reapply the bridge modification and the censored and systematic block dynamics techniques.

We first bound the mixing time on a cylinder with typical boundary conditions on its non-periodic sides.
In what follows, for any $\Lambda_{n,n'}$, label the following edges:
\begin{align*}
e_{\south\west}^\star=(0, \lfloor \tfrac {n'}2\rfloor +\tfrac 12)\,, \qquad & e^\star_{\south\east}=(n, \lfloor\tfrac {n'}2\rfloor +\tfrac 12)\,, \\
e_{\north\west}^\star=(0, \lfloor\tfrac {9n'}{10}\rfloor +\tfrac 12)\,, \qquad & e^\star_{\north\east}=(n, \lfloor \tfrac {9n'}{10}\rfloor +\tfrac 12)\,.
\end{align*}
Then define the modification $\xi'$ of boundary conditions $\xi$ by
\begin{align}\label{eq:xi-modification-cylinder}
\xi'=\xi^{e^\star_{\south\west}}\wedge\xi^{e^\star_{\south\east}}\wedge \xi^{e^\star_{\north\west}}\wedge \xi^{e^\star_{\north\east}}\wedge \xi^s
\end{align}
 and define $\Xi'_{K,N},\Upsilon'_{K,N}$ as before, for the new modification. We say that a boundary condition on $\partial_{\north,\south} \Lambda$ is in $\Upsilon_{K, N}$ if its restriction to each side is in $\Xi_{K,N}$ and there are fewer than $2K \log N$ distinct components with vertices in $\partial_{\north} \Lambda$ and $\partial_{\south} \Lambda$, and analogously for boundary conditions on $\partial_{\east,\west} \Lambda$.

\begin{theorem} [Mixing time on a cylinder]\label{thm:cylinder}
Fix $q\in (1,4]$, $\alpha \in (0,1]$ and $K>0$. There exists some $c(\alpha,q,K)>0$ such that the critical FK model on $\Lambda=\Lambda_{N,N'}$ with $ \alpha N \leq N'\leq\alpha^{-1} N $ and boundary conditions, denoted by $(p,\xi)$, that are periodic on $\partial_{\north,\south} \Lambda$ and  $\xi\in \Upsilon_{K,N}$ on  $\partial_{\east,\west} \Lambda$, satisfies
$\tmix \lesssim N^{c \log N}\,.
$
\end{theorem}
\begin{proof}
We will use a similar approach as in the proof of Theorem~\ref{thm:mainthm-fixed-bc} to reduce the cylinder to rectangles with ``typical'' boundary conditions.
It suffices to prove the theorem for large $K$, since $\Upsilon_{K,N}\subset \Upsilon_{K',N}$ for $K\leq K'$. We establish it for every fixed
\[ K\geq K_0+K_0'\quad\mbox{where}\quad K_0 = 4(c_\star +1)(c_\Upsilon^{-1}\vee 1)\quad\mbox{and}\quad K'_0 = K_0(c^{-1}_\Psi\vee 1)\,,\]
in which $c_\star=c_\star(\frac\alpha5,\frac 1{4},q)>0$ is given by Proposition~\ref{prop:point-to-point-crossing}, the constant $c_\Upsilon$ is $c(\frac {2\alpha} 5,q)>0$ from Proposition~\ref{prop:bridge-bound}, and $c_{\Psi}=c_\Psi(\frac {3 \alpha} 5,q)>0$ is given by Proposition~\ref{prop:disjoint-crossings}.

Define, as in Definition~\ref{def:censoring-block}, the censored and systematic block dynamics on
\begin{align*}
B_0&:=\llb 0,N\rrb \times \llb 0,\tfrac {N'}5\rrb \cup \llb 0,N\rrb \times \llb \tfrac {2N'}5,N'\rrb\,, \\
B_1&:=\llb 0,N\rrb \times \llb 0,\tfrac {3N'}5\rrb \cup \llb 0,N\rrb \times \llb \tfrac {4N'}5,N'\rrb \,.
\end{align*}
The choice of boundary class on $B_i$ for $i=0,1$ is $\Gamma_i=\Upsilon_{3K,N}$.
Observe that  by translating vertically on the universal cover, the blocks $B_0$ and $B_1$ are, by construction, $N\times \frac 45 N'$ rectangles with non-periodic boundary conditions. These blocks and the coupling scheme are depicted in Figure~\ref{fig:torus-coupling}. 

\begin{figure}
\centering
\vspace{-0.075in}
  \begin{tikzpicture}
    \node (plot1) at (0,0) {};

    \node (plot2) at (5,0) {};
    
        \node (plot3) at (10,0) {};

    \begin{scope}[shift={(plot1.south west)},x={(plot1.south east)},y={(plot1.north west)}, font=\small]

     \draw[draw=black, thick] (0,0) -- (0,20);
     \draw[draw=black, thick] (14,0) -- (14,20);
     
     \draw[draw=black] (6.8,-.8) -- (6.8,.8);
     \draw[draw=black] (7.2,-.8) -- (7.2,.8);
     
          \draw[draw=black] (6.8,20-.8) -- (6.8,20+.8);
     \draw[draw=black] (7.2,20-.8) -- (7.2,20+.8);
     
     \draw[draw=black, thick] (0,8)--(14,8);
          \draw[draw=black, thick] (0,4)--(14,4);

      \node[font=\Large, color=green!20!black] at (7,14) {$B_0$};

     \draw[color=black,dashed] (0,0) rectangle (14,20);
     
     \draw[fill=green, opacity=.2] (0,0) rectangle (14,4);
          \draw[fill=green, opacity=.2] (0,8) rectangle (14,20);

    \end{scope}

    \begin{scope}[shift={(plot2.south west)},x={(plot2.south east)},y={(plot2.north west)}, font=\small]
    
     \draw[draw=black, thick] (0,0) -- (0,20);w.r.t
     \draw[draw=black, thick] (14,0) -- (14,20);
     
     \draw[draw=black] (6.8,-.8) -- (6.8,.8);
     \draw[draw=black] (7.2,-.8) -- (7.2,.8);
     
          \draw[draw=black] (6.8,20-.8) -- (6.8,20+.8);
     \draw[draw=black] (7.2,20-.8) -- (7.2,20+.8);
     
          \draw[draw=black, thick] (0,12)--(14,12);
          \draw[draw=black, thick] (0,16)--(14,16);
          
                \node[font=\Large, color=blue!20!black] at (7,6) {$B_1$};

     \draw[color=black,dashed] (0,0) rectangle (14,20);

     \draw[fill=blue, opacity=.2] (0,0) rectangle (14,12);
          \draw[fill=blue, opacity=.2] (0,16) rectangle (14,20);

    \end{scope}
    
        \begin{scope}[shift={(plot3.south west)},x={(plot3.south east)},y={(plot3.north west)}, font=\small]
    
     \draw[draw=black, thick] (0,0) -- (0,20);
     \draw[draw=black, thick] (14,0) -- (14,20);
     
     \draw[draw=black] (6.8,-.8) -- (6.8,.8);
     \draw[draw=black] (7.2,-.8) -- (7.2,.8);
     
          \draw[draw=black] (6.8,20-.8) -- (6.8,20+.8);
     \draw[draw=black] (7.2,20-.8) -- (7.2,20+.8);
     
          \draw[draw=black, thick] (0,8)--(14,8);
          \draw[draw=black, thick] (0,4)--(14,4);
          
        \node[font=\Large, color=black] at (7,14) {$R$};

     \draw[color=black,dashed] (0,0) rectangle (14,20);
     
          \draw[color=black] (0,12) rectangle (14,16);

     \draw[fill=green, opacity=.2] (0,0) rectangle (14,4);
          \draw[fill=green, opacity=.2] (0,8) rectangle (14,20);

               \node (fig1) at (7,13) {
	\includegraphics[width=113pt,height=50pt]{fig_interface}};
	
	               \node (fig1) at (7,18) {
	\includegraphics[width=107pt, height=30pt]{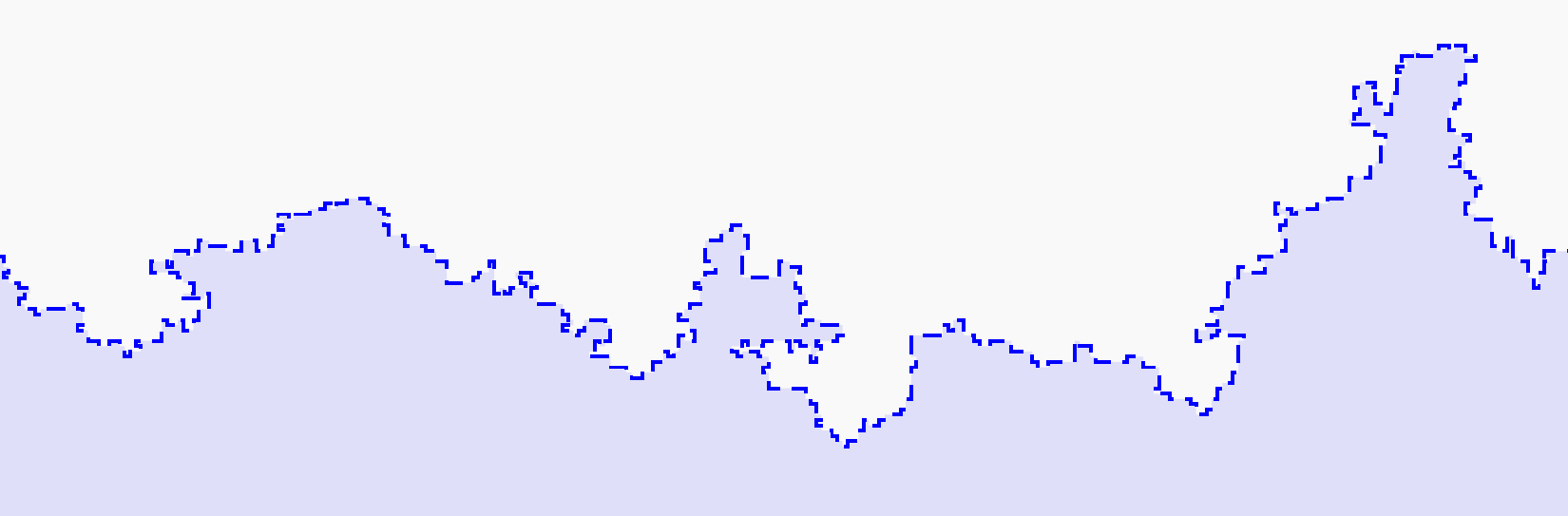}};

	      \draw[color=blue, dashed] (0,10) -- (-1.2,10);
	            \draw[color=blue, dashed] (14,10) -- (15.2,10);
	            
	            	      \draw[color=blue, dashed] (0,18) -- (-1.2,18);
	            \draw[color=blue, dashed] (14,18) -- (15.2,18);

          \node[color=DarkGray,font=\large] at (7,6) {$\boldsymbol 1/0$};
          
          \node[color=black] at (-1,14) {$\xi'$}; 
                    \node[color=black] at (15,14) {$\xi'$};

    \end{scope}

  \end{tikzpicture}
  \vspace{-0.05in}
  \caption{Left and center: block choices $B_0,B_1$ for the censored and systematic block dynamics block dynamics chain on $\Lambda_{N,N'}$ with periodic boundary on $\partial_{\north,\south}\Lambda$, and $\xi'$ on $\partial_{\east,\west}\Lambda$. Right: the dual crossings in $R_\south^\ast$ and $R_N^\ast$ which allow coupling on the set $R$.}
  \label{fig:torus-coupling}
    \vspace{-0.07in}
\end{figure}

It again suffices, by Fact~\ref{fact:init-config-comparison}, to show that there exists some $c(\alpha,q,K)>0$ such that
\begin{align}\label{eq:want-to-show-cylinder}
\|\mathbb P_{1}^{p,\xi}(X_{N^{c\log N}}\in \cdot)-\mathbb P^{p,\xi}_{0}(X_{N^{c\log N}} \in \cdot)\|_\tv \leq N^{-3}\,.
\end{align}
In the setting of the cylinder, the side modification $(p,\xi^s)$ of $(p,\xi)$ only disconnects $\partial_\east \Lambda$ from $\partial_\west \Lambda$, and so, if $\xi'$ is as in~\eqref{eq:xi-modification-cylinder}, then
 $d(\xi',\xi) \leq 6K \log N$. Thus,  by~\eqref{eq:bc-pert-2}, the triangle inequality and Theorem~\ref{thm:censoring} (as explained in the derivation of~\eqref{eq:modification+censoring}), if
\[t_N = N^\lambda kT\quad\mbox{ for }\quad\lambda:=24K+5\]
 (so that $|E|^2 q^{4d(\xi',\xi)} = o(N^\lambda)$), then for every $k,T\geq 0$,
\begin{align*}
\|\mathbb P^{p,\xi}_{1}(X_{t_N}\in \cdot)-\mathbb P^{p,\xi}_{0}(X_{t_N}\in \cdot)\|_\tv \leq
16\max_{\omega_0\in\{0,1\}}\|\mathbb P^{p,\xi'}_{\omega_0}(\bar X_{kT}\in \cdot)-\pi_{\Lambda}^{p,\xi'}\|_\tv + 2e^{-N^{\lambda/4}}\,,
\end{align*}
which, by Proposition~\ref{prop:censoring-block-dynamics}, is at most
\begin{align}\label{eq:cylinder-comp}
& \,\, 16\max_{\omega_0\in\{0,1\}}\|\mathbb P^{p,\xi'}_{\omega_0}(Y_k \in \cdot)-\pi_{\Lambda}^{p,\xi'}\|_\tv + 16k(\rho+\epsilon(T))+ 2e^{-N^{\lambda/4}}\,,
\end{align}
where $\epsilon(T)$ and $\rho$ are given by~\eqref{eq:T} and~\eqref{eq:def-rho}, respectively, w.r.t.\ the blocks $B_0,B_1$, the permissible boundary conditions $\Upsilon_{3K,N}$, and the initial configuration $\omega_0\in\{0,1\}$: 
\begin{align*}
 \epsilon(T)&=  \max_{\omega'\in\Omega} \max_{i\in\{0,1\}} \max_{\zeta\in \Upsilon_{3K,N}^{B_i}} \|\mathbb P_{\omega'}^{\zeta,B_i} (X_T\in \cdot )-\pi_{B_i}^{\zeta}\|_\tv\,,\\
\rho&=\max_{k\ge 1}\max_{ i\in\{ 0,1\}} \P_{\omega_0}\left(Y_k\restriction_{\partial B_i}\notin \Upsilon_{3K,N}^{B_i} \right)\,.
\end{align*}

We next bound the first term in the right-hand side of Eq.~\eqref{eq:cylinder-comp} by the probability of not coupling the systematic block dynamics chains started from two arbitrary initial configurations under the grand coupling (cf.\ Lemma~\ref{lem:systematic-bound}). In the first time step, we try to couple the chains started from $\omega_1,\omega_2$ on
\[R:=\llb 0,N\rrb \times \llb \tfrac {3N'}5,\tfrac {4N'}5\rrb\,,\]
so that in the second step the identity coupling couples them on all of $\Lambda$. It suffices to couple the systematic chains started from $\omega_1=0$ and $\omega_2=1$ under the grand coupling.
In order to couple samples from the $(0,\xi')$ and $(1,\xi')$ boundary conditions on $R$ (induced by $\omega_1=0$ and $\omega_2=1$ resp.), define the following two sub-blocks of $B_0$:
\begin{align*}
R_\south:=  \llb 0,N\rrb \times \llb \tfrac {2N'}5,\tfrac {3N'}5\rrb\,, \qquad R_\north:= \llb 0,N\rrb \times \llb \tfrac {4N'}5,N'\rrb\,.
\end{align*}
By Proposition~\ref{prop:point-to-point-crossing}, monotonicity in boundary conditions, and the FKG inequality,
\[\min_\eta \pi^{\eta,\xi'}_{B_0}\left(e^\star_{\south\west}\stackrel{R^\ast_\south}\longleftrightarrow e^\star_{\south\east} \,,\, e^\star_{\north\west}\stackrel{R^\ast_\north}\longleftrightarrow e^\star_{\north\east}\right)\gtrsim N^{-2c_\star}\,.
\]
By the Domain Markov property, and the definition of the boundary modification $\xi'$,
\[\pi^{1,\xi'}_{B_0}\left (\omega\restriction_{R}\given e^\star_{\south\west}\stackrel{R^\ast_\south}\longleftrightarrow e^\star_{\south\east} \,,\, e^\star_{\north\west}\stackrel{R^\ast_\north}\longleftrightarrow e^\star_{\north\east}\right )\stackrel{d}=
\pi^{0,\xi'}_{B_0}\left (\omega\restriction_{R}\given e^\star_{\south\west}\stackrel{R^\ast_\south}\longleftrightarrow e^\star_{\south\east} \,,\, e^\star_{\north\west}\stackrel{R^\ast_\north}\longleftrightarrow e^\star_{\north\east}\right)\,.
\]
As before, using the grand coupling and revealing edges from $\partial_\south R_\south$ and $\partial_\north R_\north$ until we reveal a pair of such horizontal dual-crossings, by monotonicity, we can couple $\pi_{B_0}^{\omega_1,\xi'}$ and $\pi_{B_0}^{\omega_2,\xi'}$ on $R$ with probability at least $c N^{-2c_\star}$. On that event, the two chains are coupled in the next step (and thereafter) on all of $\Lambda$ with probability 1. By the definition of the systematic block dynamics, we conclude that, for some $c_Y>0$ and every $k\geq 2$,
\[\max_{\omega_0\in\{0,1\}} \|\mathbb P^{p,\xi'}_{\omega_0}(Y_k\in \cdot)-\pi_{\Lambda}^{p,\xi'}\|_\tv \leq \exp(-c_YkN^{-2c_\star})\,.
\]

To bound $\rho$, first note that, for $\omega_0\in\{0,1\}$, the block $B_0$ has boundary conditions $(0,\xi')$ or $(1,\xi')$, both of which are in $\Upsilon_{3K,N}$ by Remark~\ref{rem:wired-free}. Thereafter,  the uniformity of Proposition~\ref{prop:bridge-bound} in boundary conditions implies that for every $\eta$,
\[\pi_{B_0}^{\eta,\xi'}(\omega\restriction_{\partial B_1} \notin \Xi_{K,N})\lesssim N^{-c_{\Upsilon} K}\,,
\] and likewise when exchanging $B_0$ and $B_1$.
We need to also bound the number of boundary components intersecting distinct sides of $B_0$ or $B_1$.  We can bound the connections between $\partial_{\north,\south}B_i$ and $\partial_{\east,\west} B_i$ (for $i=0,1$) deterministically by $4K\log N$ as in the proof of Corollary~\ref{cor:retain-bc-2}. In the present setting there could also be (multiple) open components connecting $\partial_{\south} B_i$ to $\partial_{\north} B_i$ in $\Lambda-B_i$. By Proposition~\ref{prop:disjoint-crossings} and monotonicity in boundary conditions, for every $\eta$,
\[\pi_{B_0}^{\eta,\xi}(|\Psi_{\Lambda-B_1}|\geq K\log N)\lesssim N^{-c_\Psi K}\,,
\]
where, as in that proposition, $|\Psi_{\Lambda-B_1}|$ is the number of distinct vertical crossings of $\Lambda-B_1$.
By the choices of $K_0$ and $K_0'$, a union bound yields
\[\rho \lesssim \max_{\eta} \pi_{B_0}^{\eta,\xi'} (\omega \restriction_{\partial B_1}\notin \Upsilon_{3K,N}) \lesssim N^{-4c_\star -4}\,.
\]

Observe that on their respective translates, $B_0$ and $B_1$ are $N\times \frac 45 N'$ rectangles, so we can bound $\max_i \max_{\xi\in \Upsilon_{3K,N}} \tmix^{\xi,B_i}$ using Theorem~\ref{thm:mainthm-fixed-bc}; by that theorem, rotational symmetry, and the sub-multiplicativity of $\bar d_{\tv}$, we have that for some $c_B(\alpha,q,K)>0$,
\[\epsilon(T)\lesssim \exp(-c_B^{-1}TN^{-c_B\log N})\,,
\]
uniformly over $ \alpha N \leq N'\leq  \alpha^{-1} N$.
Altogether, combining the bounds on $\rho$, $\epsilon$, and the systematic block dynamics distance from stationarity, in Eq.~\eqref{eq:cylinder-comp}, we see that for
\[k=N^{2c_{\star}+1}\quad \mbox{ and } \quad T=N^{(c_B+1)\log N}\,
\]
one has
\[\|\mathbb P^{p,\xi}_{1}(X_{t_N}\in \cdot)-\mathbb P^{p,\xi}_0(X_{t_N} \in \cdot)\|_\tv= o(N^{-3})\,,
\]
implying Eq.~\eqref{eq:want-to-show-cylinder} and concluding the proof.
\end{proof}

\begin{proof}[\textbf{\emph{Proof of Theorem~\ref{mainthm:1}}}]
The theorem is obtained by reducing the mixing time on the torus to that on a cylinder and then applying Theorem~\ref{thm:cylinder}. Fix $\bar \alpha\in(0,1]$ and consider $\Lambda=\Lambda_{n,n'}$ with $ \bar \alpha n  \leq n'\leq  \bar \alpha^{-1} n $ and periodic boundary conditions, denoted by $(p)$, identified with $(\mathbb Z/n\mathbb Z)\times (\mathbb Z/n'\mathbb Z)$ (the special case $n'=n$ is formulated in  Theorem~\ref{mainthm:1}).

Let $c_\star=c_\star(\frac {\bar \alpha} 5,\frac 14,q)>0$ be given by Proposition~\ref{prop:point-to-point-crossing} and let $c_\Upsilon$,$c_\Psi$, $K_0$ and $K_0'$ be given as in the proof of Theorem~\ref{thm:cylinder}. Define $K=K_0+K_0'$. We consider the censored and systematic block dynamics with the block choices,
\begin{align*}
B_0&:=\llb 0,n\rrb \times \llb 0,\tfrac {n'}5\rrb \cup \llb 0,n\rrb \times \llb \tfrac {2n'}5,n'\rrb \qquad B_1:=\llb 0,n\rrb \times \llb 0,\tfrac {3n'}5\rrb \cup \llb 0,n\rrb \times \llb \tfrac {4n'}5,n'\rrb \,,
\end{align*}
and boundary class
\[\Upsilon^p_{3K,n}:=\left\{\xi: \xi\restriction_{\partial_{\north,\south} \Lambda}=p, \xi\restriction_{\partial_{\east,\west} \Lambda}\in \Upsilon_{3K,n}\right\}\,.
\]
By Theorem~\ref{thm:censoring}, the triangle inequality and Proposition~\ref{prop:censoring-block-dynamics}, for every $k,T\geq 0$,
\begin{align*}
\|\mathbb P^p_{1}(X_{kT}\in \cdot)-\mathbb P^p_0 (X_{kT}\in \cdot)\|_\tv \leq & \,\, 2\max_{\omega_0\in\{0,1\}}\|\mathbb P^{p}_{\omega_0}(Y_k \in \cdot)-\pi_{\Lambda}^{p}\|_\tv+2k(\rho+\epsilon(T))\,,
\end{align*}
where $\rho$ and $\epsilon(T)$ are w.r.t.\ the class $\Upsilon_{3K,n}^p$ of permissible boundary conditions. It suffices, as before, to prove that the right-hand side is $o(n^{-3})$ and then use~\eqref{eq-dbar-upper-bound} and the sub-multiplicativity of $ \bar d_\tv(t)$ to obtain the desired result.

Recall the edges $e^\star_{\south\west},e_{\north\west}^\star, e_{\south\east}^\star$ and $e_{\north\east}^\star$ on $\Lambda_{n,n'}$.
As in the proof of Theorem~\ref{thm:cylinder}, if
\begin{align*}
R_\south:=  \llb 0,n\rrb \times \llb \tfrac {2n'}5,\tfrac {3n'}5\rrb\,, \qquad R_\north:= \llb 0,n\rrb \times \llb \tfrac {4n'}5,n'\rrb\,,
\end{align*}
then by Proposition~\ref{prop:point-to-point-crossing} and the FKG inequality, we have
\[\pi_{B_0}^{1,p} \left(e_{\south\west}^\star\stackrel{R_\south^\ast}\longleftrightarrow e_{\south\east}^\star\,,\, e_{\north\west}^\star\stackrel{R_\north^\ast} \longleftrightarrow e^\star_{\north\east}\right)\gtrsim n^{-2c_\star}\,.
\]
Crucially, while no boundary modification was done in this case, the periodic sides of $B_0$ have no bridges over the four designated edges, and the two horizontal dual-crossings, from the event above, disconnect its non-periodic sides ($\partial_{\north} B_0$ and $\partial_\south B_0$) from $\partial_\north B_1$ and $\partial_\south B_1$. Therefore,
if that event occurs for the systematic block dynamics chain started from $\omega_0=1$, the grand coupling carries it to the chains started from all other initial states, and yields a coupling of all these chains  on $\llb \frac{3n}5,\frac{4n}5\rrb \times \llb 0,n'\rrb \supset \partial B_1$. By definition of the systematic block dynamics and  sub-multiplicativity of $\bar d_{\tv}(t)$, for $k\geq 2$,
\begin{align}\label{eq:syst-dist-torus}
\max_{\omega_0\in \{0,1\}}\|\mathbb P_{\omega_0}^p(Y_k \in \cdot) -\pi_{\Lambda}^p\|_\tv\leq  \exp(-c_Y k n^{-2c_\star})\,.
\end{align}
Observe that at every time step of the systematic block dynamics, the block $B_i$ ($i=0,1$) is an $n\times \frac 45 n'$ rectangle with periodic boundary conditions on $\partial_{\east,\west} B_i$ and boundary conditions $\eta$ induced by the chain on $\partial_{\north,\south} B_i$. By Theorem~\ref{thm:cylinder}, for some $c(\bar \alpha,q,K)>0$,
\[\max_i\max_{(p,\eta)\in\Upsilon^p_{3K,n}}\tmix^{(p,\eta),B_i} \lesssim n^{c\log n}\,,
\]
and by sub-multiplicativity of $\bar d_{\tv}(t)$, we have $\epsilon (T)\lesssim \exp(-c^{-1} Tn^{c\log n})$.
 As in the proof of Theorem~\ref{thm:cylinder}, since the estimate on $\rho$ was uniform in the boundary conditions,  we again have $\rho\lesssim n^{-4c_\star -4}$ (using Propositions~\ref{prop:bridge-bound} and~\ref{prop:disjoint-crossings}).
Combining the bounds on $\rho$ and $\epsilon$ with~\eqref{eq:syst-dist-torus}, there exists some $c(\bar \alpha,q,K)>0$ such that
\[\|\mathbb P^p_{1}(X_{n^{c\log n}}\in \cdot)-\mathbb P^p_{0} (X_{n^{c \log n}}\in\cdot)\|_\tv = o({n^{-3}})\,,
\]
as required.
\end{proof}

\subsection*{Acknowledgment}
The authors thank the anonymous referees for their helpful suggestions. R.G.\ was supported in part by NSF grant DMS-1507019. E.L.\ was supported in part by NSF grant DMS-1513403.

\bibliographystyle{abbrv}
\bibliography{fk_mixing}

\end{document}